\documentclass[11pt]{article}
\usepackage{amsmath, amssymb, theorem, latexsym, epsfig, graphics}
\usepackage{subfigure}
\numberwithin{equation}{section}

\theoremstyle{plain}
\theorembodyfont{\itshape}
\newtheorem{theorem}{Theorem}[section]
\newtheorem{proposition}[theorem]{Proposition}
\newtheorem{lemma}[theorem]{Lemma}
\newtheorem{corollary}[theorem]{Corollary}
                                                                                
\theorembodyfont{\rmfamily}
\newtheorem{definition}[theorem]{Definition}
\newtheorem{claim}[theorem]{Claim}
\newtheorem{example}[theorem]{Example}
\newtheorem{remark}[theorem]{Remark}

\newenvironment{proof}{{\noindent \textbf{Proof}\,\,}}{\hspace*{\fill}$\Box$\medskip}
\def\la{\lambda}

\def\cc{\mathbb C}

\def\rr{\mathbb R}

\def\zz{\mathbb Z}
\def\opint{\operatorname{Int}}
\def\mco{\mathcal O}
\def\La{\Lambda}
\def\mo{\operatorname{mod}}

\def\diag{\operatorname{diag}}
\def\sign{\operatorname{sign}}
\def\mcl{\mathcal L}
\def\rp{\mathbb{RP}}
\def\mcs{\mathcal S}

\def\mch{\mathcal H}

\def\nn{\mathbb N}

\def\mcp{\mathcal P}
\def\wh#1{\widehat#1}

\def\var{\varepsilon}
\def\mcr{\mathcal R}

\def\wt#1{\widetilde#1}

\def\tt{\mathbb T}

\title{On phase-lock area parquet in a special slow-fast limit of  model of Josephson junction}
\author{Alexey Glutsyuk\thanks{Higher School of Modern Mathematics MIPT, Moscow, Russia}
\thanks{HSE University, Moscow, Russia}
\thanks{CNRS, UMR 5669 (UMPA, ENS de Lyon), Lyon, France}\thanks{The research is supported by the MSHE RF GZ project FSMG-2026-0020.}}
\begin{document}
\maketitle
\begin{abstract} B.Josephson (Nobel 
Prize, 1973) predicted a {\it tunnelling effect} for  a system of two superconductors 
separated by a  narrow dielectric (such a system is called {\it Josephson junction}): existence of a supercurrent through it and equations governing it. 
The overdamped Josephson junction is modeled by the family of differential equations  on the torus $\tt^2=\rr^2\slash2\pi\zz^2$, $\frac{d\theta}{d\tau}=\frac1{\omega}(\cos\theta+B+A\cos\tau)$,  which is known as the {\it RSJ model.} It depends on three parameters: $B$ called the abscissa, $A$ called the ordinate, and a fixed frequency $\omega$. We study its 
{\it rotation number} $\rho(B,A;\omega)$ 
as a function of  $(B,A)$ and the {\it phase-lock areas:} those its level subsets that have non-empty interiors. They  exist only for integer values of the rotation number (Buchstaber, Karpov, Tertychnyi). In this paper we study asymptotics of the phase-lock area portrait in a special slow-fast limit, as $\omega\to0$ and $(B,A)\to(0,1)$ so that $(B,A)-(0,1)=O(\omega)$. We show that in the rescaled parameters $\ell:=\frac B{\omega}$ and $u:=\frac{A-1}{\omega}$ the phase-lock area portrait converges to a parquet with 
boundary lines being parallel to the lines $\{ u\pm\ell=0\}$. Namely, the limit of the phase-lock area with rotation number $r$ 
 is the union of an infinite chain of squares going up with integer vertices, having diagonals of length two lying on the line $\{\ell=r\}$,  and an infinite strip going down (sector in the case, when $r=0$). We state and prove a generalization of this result to a wide class of slow-fast systems on 2-torus. 
 \end{abstract} 
 \tableofcontents
\section{Introduction}
\subsection{RSJ Model of Josephson junction as a special slow-fast system. Main results}
The tunnelling effect predicted by B.Josephson in 1962 \cite{josephson} (Nobel 
Prize 1973) deals with a {\it Josephson junction:} a system of two superconductors 
separated by a  narrow dielectric. It states existence of a supercurrent through it and yields  equations governing it. It was confirmed experimentally by P.W.Anderson and J.M.Rowell in 1963 \cite{ar}. 

The RSJ model of  {\it overdamped Josephson junction},  
see \cite{stewart, mcc,  lev,  schmidt}, \cite[p. 306]{bar}, \cite[pp. 337--340]{lich}, 
\cite[p.193]{lich-rus}, \cite[p. 88]{likh-ulr} is the family of nonlinear differential equations
 \begin{equation}\frac{d\phi}{dt}=-\sin \phi + B + A \cos\omega t, \ \omega>0, \ B\geq0.\label{joso}\end{equation}
 Here $\phi$ is the  difference of phases (arguments) of the complex-valued 
 wave functions describing the quantum mechanic 
 states of the two superconductors. Its derivative is 
 equal to the voltage up to known constant factor.  
   
Equations (\ref{joso}) also arise in several models in physics, mechanics and geometry, e.g.,  
in  planimeters, see  \cite{Foote, foott}. 
 
 V.M.Buchstaber, O.V.Karpov and S.I.Tertychnyi suggested to present (\ref{joso}) as a dynamical system family on the 2-torus $\mathbb T^2=S^1\times S^1=\rr^2_{\theta,\tau}\slash2\pi\zz^2$. Namely,  the variable 
 change 
\begin{equation}\tau:=\omega t, \ \theta:=\phi+\frac{\pi}2\label{elmu}\end{equation}
 transforms (\ref{jos}) to a non-autonomous ordinary differential equation on $\tt^2$:  
\begin{equation} \frac{d\theta}{d\tau}=\frac1{\omega}(\cos\theta + B + A \cos \tau).\label{jostor}\end{equation}
The graphs of its solutions are the orbits of the vector field
\begin{equation}\begin{cases}\dot\theta=\cos\theta+B+A\cos\tau\\
\dot\tau=\omega.\end{cases}\label{jos}\end{equation}
on $\mathbb T^2$. The {\it rotation number} of its flow, see \cite[p. 104]{arn},  is a function $\rho(B,A)$ of parameters:
$$\rho(B,A;\omega)=\lim_{\tau\to+\infty}\frac{\theta(\tau)}{\tau}.$$
Here $\theta(\tau)$ is a general $\rr$-valued solution of  equation  (\ref{jostor}), which depends 
on the initial condition for $\tau=0$. Recall that the rotation number exists and is independent on the choice of the initial condition, see \cite[p.104]{arn}. 
The rotation number modulo $\zz$ is known to be equal to the rotation number of the {\it Poincar\'e first return map} 
\begin{equation}h: S^1_{\theta}\times\{0\}\to S^1_{\theta}\times\{0\},\label{poincmap}\end{equation}
 which sends an initial condition $(\theta_0,0)$ to the next intersection 
point of the corresponding phase curve of (\ref{jos}) with the cross-section $S^1_\theta\times\{0\}$. It is given by the time 
$\frac{2\pi}{\omega}$ flow map of the field (\ref{jos}). 

 The parameter $B$ is called {\it abscissa,}  $A$ is called the {\it ordinate,} $\omega$ is called {\it frequency.} 
 Recall the following well-known definition. 

\begin{definition} \label{defasl} (cf. \cite[definition 1.1]{4}) The {\it $r$-th planar phase-lock area} is the level set 
$$L_r=L_r(\omega)=\{(B,A)\in\rr^2 \ | \ \rho(B,A;\omega)=r\}\subset\rr^2_{B,A},$$ 
provided that it has a non-empty interior. 
\end{definition}
Phase-lock areas of family (\ref{jos}) were studied by V.M.Buchstaber, O.V.Karpov, S.I.Tertychnyi, 
Yu.P.Bibilo, the author et al, see \cite{bibgl, bibgl2}, \cite{bg}--\cite{bt1}, \cite{4}--\cite{gn19}, 
\cite{IRF, krs, RK},  \cite{tert, tert2} and references therein.  The  following  results are known and proved mathematically:

1) The {\it rotation number quantization effect}  \cite{buch2}: phase-lock areas exist only for integer rotation number  values.

2) {\it Transversal regularity} outside the phase-lock areas \cite[proposition 5.3]{bg2}. The restriction of the rotation number function $\rho(B,A)$ to the complement of the union of the phase-lock areas is an analytic  submersion onto $\rr\setminus\zz$. The preimage of each non-integer rotation number value $\rho$ is the graph $\{ B=g_\rho(A)\}$ of an analytic function $g_\rho$. This was proved in \cite{bg2}  using the fact that each Poincar\'e map  (\ref{poincmap})  is a M\"obius transformation, and if $\rho\notin\zz$, it  elliptic, i.e., analytically conjugated to a rotation. 

3) {\it Regularity and Bessel asymptotics of boundaries} \cite{RK}. The boundary of each $L_r(\omega)$
 consists of two  graphs of analytic functions $L_{r,\alpha}=\{ B=G_{r,\alpha}(A)\}$, $\alpha=0,\pi$. Each $L_{r,\alpha}$ consists of those parameter values for which the Poincar\'e map  fixes the point $(\alpha,0)$. 
  The functions $G_{r,\alpha}(A)$ have Bessel asymptotics, as $A\to\infty$. Their Bessel asymptotics was   
 observed and proved on physics level in ~\cite{shap}, see also \cite[p. 338]{lich},
 \cite[section 11.1]{bar}, ~\cite{buch2006}, and  proved mathematically in ~\cite{RK}.

 4) Each  phase-lock area is a garland  of infinitely many bounded domains going to infinity in the vertical direction,  separated by  points of intersection $L_{r,0}\cap L_{r,\pi}$, see \cite{RK}. Those  separation 
 points that lie on the horizontal $B$-axis, namely $A=0$, are the so-called
 {\it growth points} with $B=\sign r\sqrt{r^2\omega^2+1}$, see \cite[corollary 3]{buch1}. The other separation points, which  lie outside the horizontal $B$-axis, are called the  {\it constrictions}. See Figure 1 below.

  \begin{figure}[ht]
  \begin{center}
   \epsfig{file=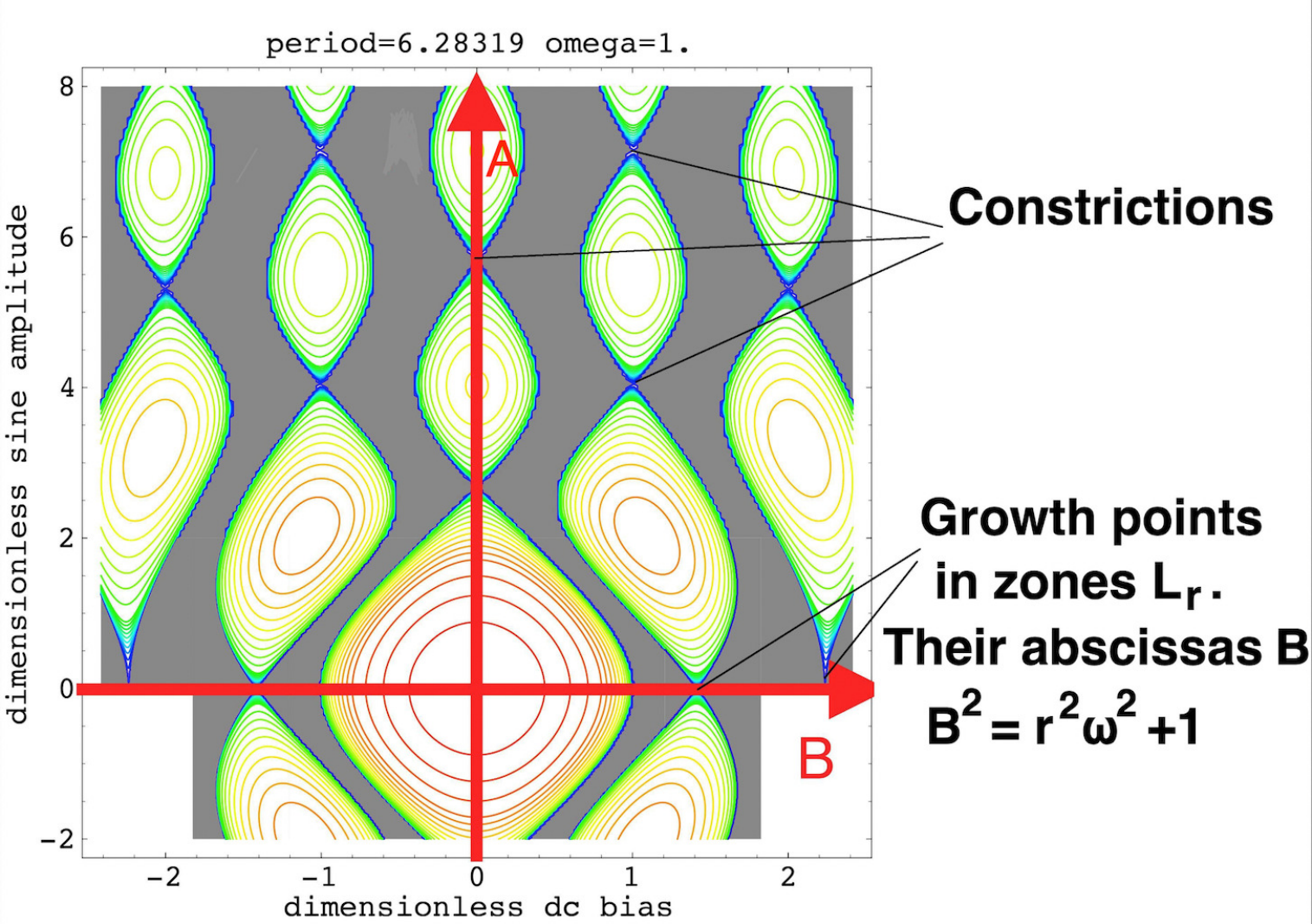, width=18em}
   \hspace{1.8em}
   \epsfig{file=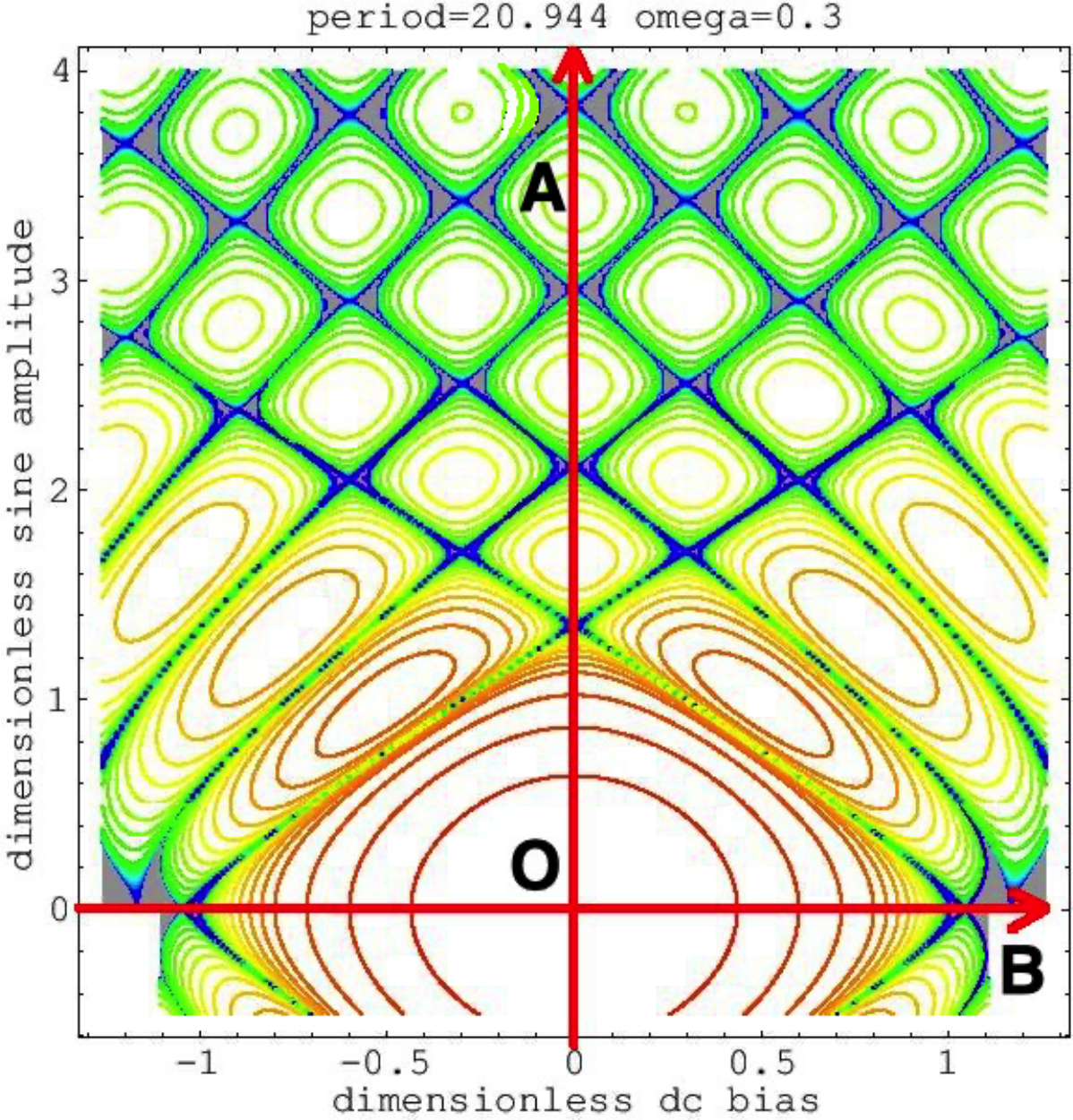, width=12em}
    \caption{Phase-lock areas and their constrictions for $\omega=1$ on the left and $\omega=0.3$ on the right. The abscissa is $B$, the ordinate is $A$. Figures 
    taken from papers \cite[fig. 1b,d,e)]{bg2}, \cite[p. 331]{bt1} with authors' permission, with coordinate axes added.}
  \end{center}
\end{figure}

 5) In each  phase-lock area $L_r(\omega)$ all its  constrictions lie in the same 
vertical line $\Lambda_r:=\{ B=r\omega\}$, see \cite[theorem 1.4]{bibgl}.

6) As $\omega\to\infty$, the renormalized ordinates $s=\frac A\omega$ of the constrictions lying in the upper half-plane $\{ A>0\}$ 
with rotation number $r$ converge to the positive zeros of the $r$-th Bessel function \cite{g23}.

7)  For every $r\in\zz$ and $\omega>0$ the $r$-th  phase-lock area $L_r(\omega)$ is symmetric to the $-r$-th one with respect to the vertical $A$-axis and is symmetric to itself with respect to the $B$-axis. See Figure 1.

The right picture in Fig. 1 is the numerical phase-lock area picture for $\omega=0.3$ obtained by V.M.Buchstaber, O.V.Karpov and S.I.Tertychnyi \cite{buch2}. It suggest that phase-lock areas may limit to a parquet structure, as $\omega\to0$. 

\smallskip

{\bf Question} (V.M.Buchstaber, 2010) {\it Study the asymptotics of phase-lock areas, as $\omega\to0$.} 
\smallskip

In the present paper we prove the next theorem, which is a first result  on explicit asymptotics of the phase-lock areas, as $\omega\to0$. We consider the case, when  $B$ and $A$ depend on  $\omega$ so that 
 \begin{equation}(B,A)\to(0,1), \ \ \ (B,A)-(0,1)=O(\omega), \ \text{ as } \omega\to0.\label{asab}\end{equation}
We work  in the rescaled parameters
\begin{equation}\ell:=\frac B\omega, \ \ \ u:=\frac{A-1}\omega.\label{ell-u}\end{equation}
Everywhere below by $L_r$ we denote the phase-lock area with the rotation number $r$ in the rescaled parameter 
plane $\rr^2_{\ell,u}$. Set 
\begin{equation}Z_m:=(2m-1,2m+1)\subset\rr \ \  \text{ for } m\in\nn, \ \ Z_0:=(-\infty,1).\label{zmz0}\end{equation}
\begin{theorem} \label{thm1}\footnote{Theorem \ref{thm1} with a brief proof was announced in the author's talk at Subtropical International Workshop on Dynamical Systems and Analytic Differential Equations.  Shenzhen, China, April 20-24, 2026, see https://math.sustech.edu.cn/conference/13603.html?lang=en A first preprint version of the present paper was posted at https://arxiv.org/abs/2607.26158}
\footnote{After the author obtained the results of this paper, Artem Alexandrov presented first asymptotic term in $\omega$ of  the phase-lock area boundaries in the $(B,A)$-plane at points far from $(0,1)$ in terms of elliptic integrals in his preprint \cite{alduck} very recently posted to arxiv. He used a completely different method based on reduction to Riccati equations, quantum mechanics, complex analytic extension of the slow curve as an elliptic curve and WKB arguments. The results of the present paper and method of proof, though valid only near the point $(0,1)$ in the case of the RSJ model (\ref{jos2}), are valid for a wide class of slow-fast systems, which are neither analytic, nor reducible to Riccati equations. See Theorem \ref{thm2} generalizing Theorem \ref{thm1}.}
1) For every $r\in\zz$ 
the phase-lock area $L_r=L_r(\omega)$ converges\footnote{Here and in Theorem \ref{thm2} convergence of closed subsets $M_\omega$, say, $L_r(\omega)$, in the  parameter space $U$, which will be either $\rr^n$, or a domain $U\subset\rr^n$, means the following. Take one point compactification 
$\overline U$ of the parameter space; it is $S^n$ for $U=\rr^n$. The above convergence means that the closures 
$\overline M_\omega$ in $\overline U$  converge to the closure in $\overline U$ of the limit set in question in Hausdorff sense, and the closures of the complements $U\setminus M_\omega$ similarly converge to the closure of the complement of the limit set.} 
to the closure of the union 
\begin{equation}L_r^0:=\cup_{k, k+r\in\zz_{\geq0}}\La_{k,r}, \ \La_{k,r}:=\{ u+\ell\in Z_{k+r}, \ u-\ell\in Z_k\}\subset\rr_{\ell,u}.\label{lr0}\end{equation}
The domains $\La_{k,r}$ are marked by $r$ at Figure 2.   Each $\La_{k,r}$ is

- a square  for $k, k+r\geq1$; 

- a half-strip if exactly one of the numbers $k$, $k+r$ is zero; 

- the  sector   $\{ u\pm\ell\leq1\}$ with vertex $(1,1)$, if $k=r=0$. 

2) The vertices  of the limit domain $L^0_r$ lying in the line $\{ \ell=r\}$  are exactly limits of  constrictions of the phase-lock area $L_r(\omega)$. 
\end{theorem}
 \begin{figure}[ht]
  \begin{center}
   \epsfig{file=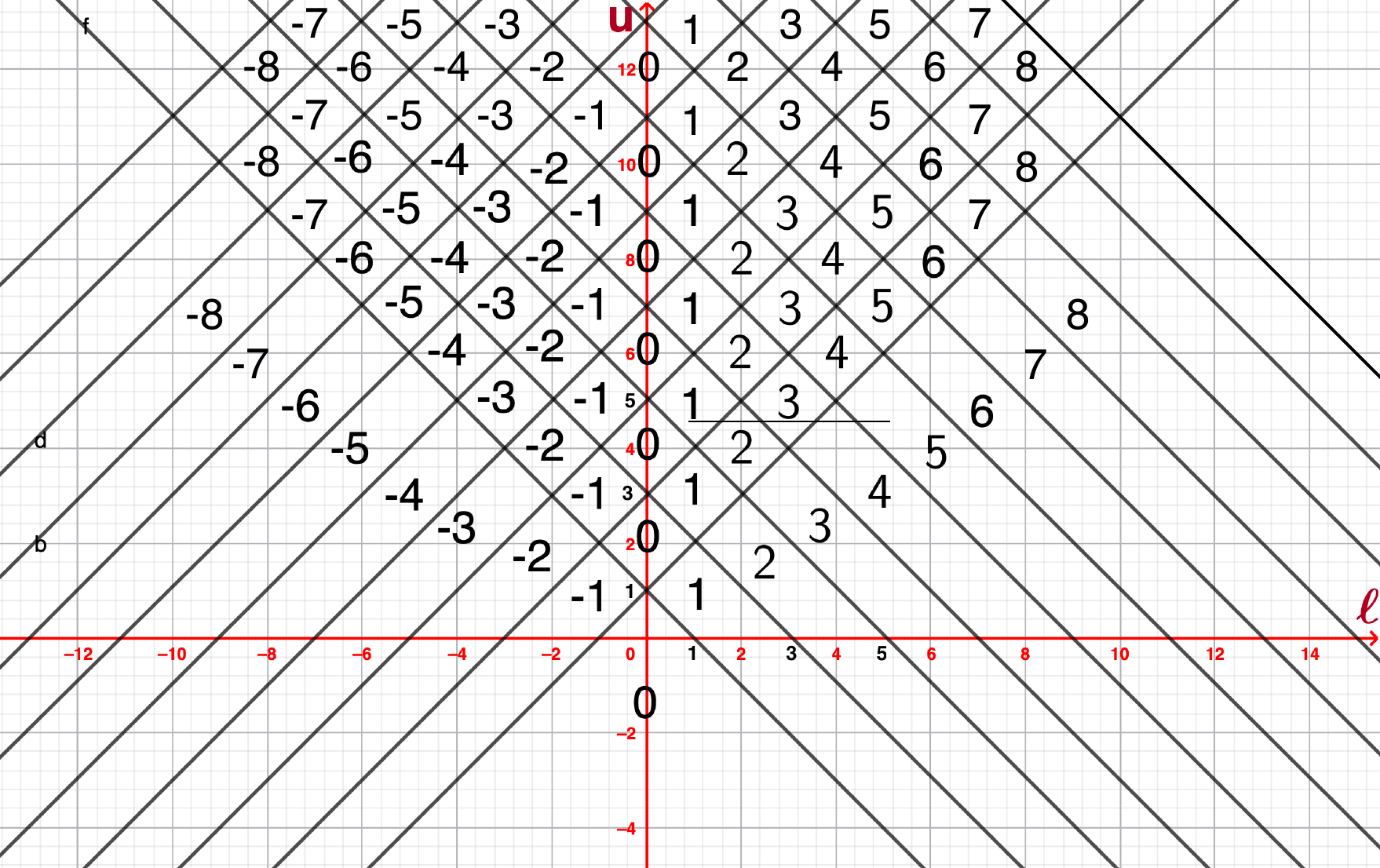, width=21em}
    \caption{Limit phase-lock areas.}
  \end{center}
\end{figure} 

\begin{remark} The statement of Theorem \ref{thm1}  is difficult to observe numerically. 
As $\omega$ is small,  some numerical experience artefacts arise and distort the pictures. 
The following numerical phase-lock area pictures for small $\omega$ made by Artem Alexandrov, see Fig. 3, seem to be among the best and more precise ones up to now. Even there one can see that most of constrictions do not come close enough to the corresponding limit parquet vertices. 
\end{remark}
\begin{figure}[ht]
  \begin{center}
   \epsfig{file=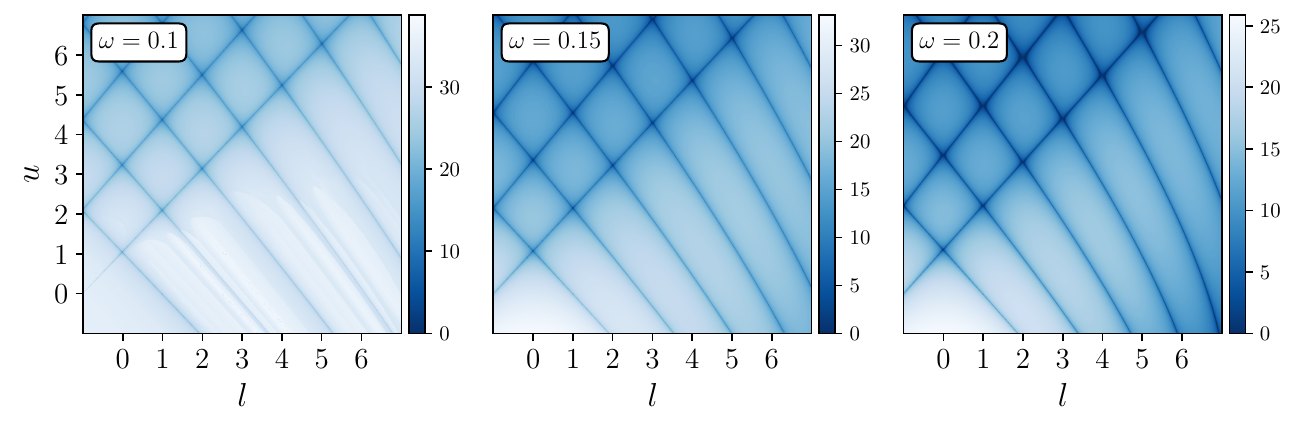, width=30em}
    \caption{Artem Alexandrov's numerical pictures for small $\omega$.}
  \end{center}
\end{figure} 
\begin{remark} \label{hist1} 
Asymptotics $A=1+\omega+o(\omega)$ of the lower constriction in the intersection $L_0\cap\{ A>0\}$ was early  found by M.J.Renne and D.Polder \cite{rp}. It was shown in \cite{krs} that as $\omega\to0$, the gaps between neighbor phase-lock areas converge to zero exponentially. It was observed in recent physical experiments \cite{stol} with a modified version of the RSJ model that includes a  half-harmonic in $\theta$  that for  small effective frequency of exterior current forcing there is also a kind of parquet structure. 
\end{remark}

Isomonodromic deformation method was applied to study of geometry of the phase-lock areas and constrictions in \cite{bibgl}, \cite{g23} and in a very recent paper \cite{tert26}. 

In Subsection 1.2 we present a more general Theorem \ref{thm2} for a large class of slow-fast systems on 2-torus. In Subsection 1.3 we give  sketch-proofs of Statement 1) of Theorem \ref{thm1} and of Theorem \ref{thm2}. In Subsection 1.4 we sketch proof of Statement 2) of Theorem \ref{thm1}. The proofs will be given in Section 2. 

\subsection{Generalization to slow-fast systems on two-torus with  slow curve projected 
to the whole  $\tau$-circle}

\begin{definition} \label{defclr} A critical point $z_0=(\theta_0,\tau_0)$ of a function $f(\theta,\tau)$ with critical value $f(z_0)=0$ 
is called {\it horizontally positive (negative) 
self-intersection},   if the Hessian form $Hf(z_0)$ of the function $f$ at $z_0$ is non-degenerate sign-indefinite and the partial derivative $\frac{\partial^2f(z_0)}{\partial\theta^2}$ is positive (respectively, negative). 
In this case the germ at $z_0$ of the level curve $C=\{ f=0\}$ consists of two transversely intersected regular germs of curves $C_{\mcl}=C_{\mcl}(z_0)$ and $C_\mcr=C_\mcr(z_0)$ that are both transversal to the horizontal line $\{\tau=\tau_0\}$. We consider their upper and lower parts numerated by $+$ and $-$ respectively:  
\begin{equation}C_{\mcl,\pm}:=C_{\mcl}\cap\{ \pm(\tau-\tau_0)>0\}, \ C_{\mcr,\pm}:=C_\mcr\cap\{\pm(\tau-\tau_0)>0\}.\label{clr}\end{equation}
We  name $C_{\mcl}$, $C_\mcr$ so that $C_{\mcl,+}$ lies on the left from  $C_{\mcr,+}$, thus 
$C_{\mcr,-}$ lies on the left from $C_{\mcl,-}$. 
The germs $C_{\mcl,+}$, $C_{\mcr,+}$ are called  {\it left and right upward branches} of the curve $C$ at $z_0$. The germs $C_{\mcr,-}$, $C_{\mcl,-}$ are called its  {\it left and right downward branches}. See Fig. 4. 
\end{definition}
\begin{figure}[ht] 
  \begin{center}
  \vspace{-1cm}
 \epsfig{file=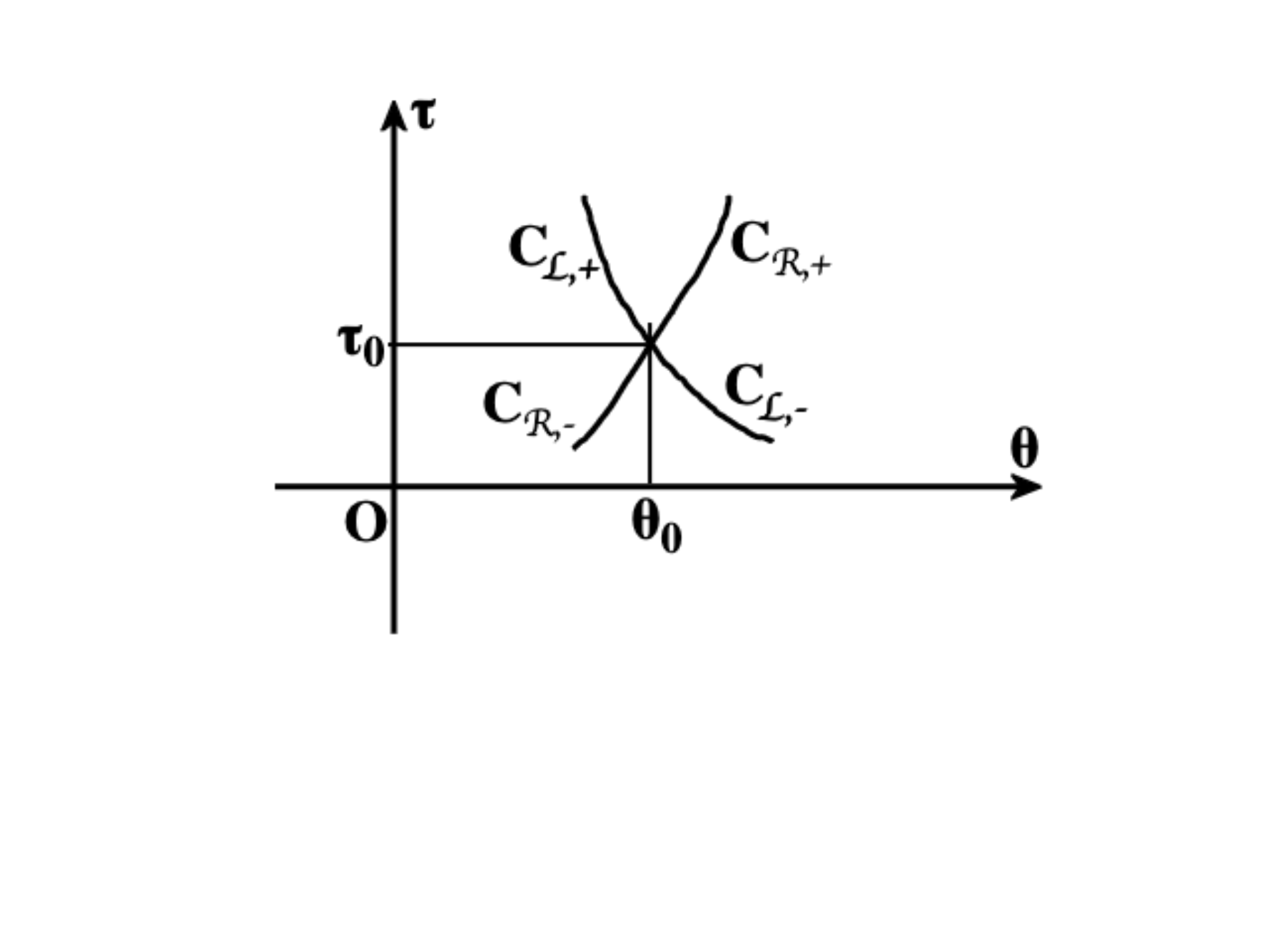, width=23em}
 \vspace{-1.8cm}
\caption{The branches $C_{\mcl,\pm}$, $C_{\mcr,\pm}$.} 
  \end{center}
\end{figure}
\begin{definition} \label{horpos} Consider a family  $f_{s;\omega}(\theta,\tau)$ of functions on a neighborhood of the origin $O=(0,0)$ depending on a parameter $s$ lying in a domain $U\subset\rr^n$ and $\omega\in(-\delta,\delta)$, continuous  in 
$(s,\omega,\theta,\tau)$  with continuous derivatives up to order three
in $(s,\omega,\theta,\tau)$. We 
consider that for every $s\in U$ one has $f_{s;0}(O)=0$ and $O$ is a horizontally positive (negative) self-intersection. 
Then we say that its {\it germ at $O$ is horizontally positive (negative),} and $O$ is {\it right (left) flow point} for the corresponding system 
$$\dot\theta=f_{s;\omega}(\theta,\tau), \ \ \dot\tau=\omega.$$
\end{definition}

To each horizontally positive (negative) germ $f_{s;\omega}$ we associate a number $c=c(s)$ defined as follows, where 
in $\pm$ the signs "$+$", "$-$" correspond to horizontally positive, respectively negative germ. 
The Hessian form of the function $f_{s;0}$, which is given by its quadratic terms, takes the form 
\begin{equation}Hf_{s;0}(d\theta,d\tau)=\pm\frac12(a_{11}d\theta^2+2a_{12}d\theta d\tau+a_{22}d\tau^2),\label{hesf}
\end{equation}
\begin{equation}a_{11}=\pm\frac{\partial^2f_{s;0}(O)}{\partial\theta^2}>0, \ \ \ 
\Delta:=a_{12}^2-a_{11}a_{22}>0.\label{hfelu2}\end{equation}
Set
$$\nu:=\pm\frac{\partial f_{s;\omega}(O)}{\partial\omega}|_{\omega=0},$$
\begin{equation} c=c(s)=c_{\pm}(s):=(\pm a_{12}+\nu a_{11})\Delta^{-\frac12}.\label{defc}\end{equation}
Here we deal with a family of dynamical systems on torus $\tt^2_{\theta,\tau}$ of the type 
\begin{equation}\begin{cases}\dot\theta=f_{s;\omega}(\theta,\tau)\\
\dot\tau=\omega,\end{cases}\label{jos2g}\end{equation}
depending on $(s;\omega)$, where $s$ lies in a connected domain $U\subset\rr^n$, with continuous derivatives as in the above definition. Thus, 
$f_{s;\omega}(\theta,\tau)$ are $2\pi$-periodic in $\theta$ and $\tau$. We consider that the  unperturbed slow curve 
$$C_{s;0}:=\{ f_{s;0}(\theta,\tau)=0\} \subset\tt^2$$
satisfies the following conditions for every $s\in U$:

(i) It contains two critical points $z_{0}$, $w_0$ 
of the function $f_{s;0}$; thus, $C_{s;0}$ is an immersed curve. The function family  $f_{\ell,u;\omega}$ has horizontally positive (negative) germs at $z_0$ ($w_0$). 

(ii) For every $\psi=z_0, w_0$ the  circle $\{\tau=\tau(\psi)\}$ intersects $C_{s;0}$ only at $\psi$. 

(iii) The projection $\pi_\tau:C_{s;0}\to S^1_\tau$ considered as a projection of parametrized immersed curve 
has a finite number, uniformly bounded in $s$, of critical points with pairwise distinct critical values that are (if any) all local minima or maxima.

(iv) There exists a closed  path $\gamma=\gamma_s\subset C_{s;0}$ starting at $z_{0}$, going along the left upward branch $C_{\mcl,+}(z_0)$
 of the curve $C_{s;0}$ 
at $z_0$, then arriving to $w_0$ along  $C_{\mcl,-}(w_0)$, then leaving $w_0$ along  $C_{\mcr,+}(w_0)$ and then closing up arriving to $z_0$ along $C_{\mcr,-}(z_0)$. The path $\gamma$ is homotopic to the circle $\{0\}\times S^1_{\tau}$. 

(v) The map $\Sigma:s\mapsto(c_+(s), c_-(s))$ sending $s$ to the  values $c_+$, $c_-$, see (\ref{defc}), corresponding 
to the points $z_0$ and $w_0$ respectively is a submersion. 

(vi) The function $f_{s,\omega}(\theta,\tau)$ is monotonous in some $s_j$, say, $s_1$, for all $\omega$ 
small enough, and each $c_{\pm}(s)$ is locally non-constant as a function of the same variable $s_j$ with fixed other 
 values $s_i$, $i\neq j$.   
\begin{example} The function family 
\begin{equation}f_{\ell,u;\omega}(\theta,\tau)=\cos\theta+B+A\cos\tau, \ \ \ B=\ell\omega, \ \ A=1+u\omega,\label{fluo}\end{equation}
 see (\ref{ell-u}), satisfies the above assumptions (i)--(v), with $s=(\ell,u)$, the curve $C_{\ell,u;0}=\{\cos\theta+\cos\tau=0\}$ independent on $(\ell,u)$ and presented at Fig. 4. The right (left) flow points $z_0=(\pi,0)$ (respectively, $w_0=(0,\pi)$) are horizontally positive (negative) self-intersections. The path $\gamma$ lifted to $\rr^2_{\theta,\tau}$ consists of two straightline segments adjacent to $(0,\pi)$, connecting successive points $(\pi,0)$, $(0,\pi)$, $(\pi,2\pi)$. 
  The values $c=c_+$, $c=c_-$ given by (\ref{defc})  corresponding to the points $z_0$ and $w_0$ respectively 
can be taken as parameters and 
\begin{equation}u\pm\ell=c_{\pm}.\label{ulgen}
\end{equation}
 Formula (\ref{ulgen}) follows from (\ref{defc}) by straightforward calculation. 
\end{example}

In fact, we need a weaker condition than (iv), see the next theorem. 
To state it, let us recall the following well-known definitions from slow-fast system theory. 
\begin{definition} An arc $\gamma$ of the unperturbed slow curve $C_{s;0}$ of a slow-fast system (\ref{jos2g}) is called {\it stable,} if it is diffeomorphically projected to 
an interval of the $\tau$-axis and the unperturbed horizontal field (\ref{jos2g}) is directed to $\gamma$ on its both sides: on its left and on its right. See Fig. 5a). 
\end{definition}
\begin{definition}  Let now $C_{s;0}$ satisfy conditions (i)--(iii). A {\it singular orbit}  is a piecewise smooth oriented curve consisting of stable arcs of the curve $C_{s;0}$ and horizontal segments. The stable arcs are oriented up. Their upper endpoints belong to the union of critical points of the function 
$f_{s;0}$ and  local maxima of the projection $\pi_\tau$.  Each horizontal segment is bounded by a local maximum and another point lying in $C_{s;0}$. It is oriented by the unperturbed horizontal vector field from the local maximum to its  other end,  
and its interior is disjoint from $C_{s;0}$.  See Fig. 5b).
\end{definition}
\begin{figure}[ht] 
  \begin{center}
  \vspace{-1cm}
 \epsfig{file=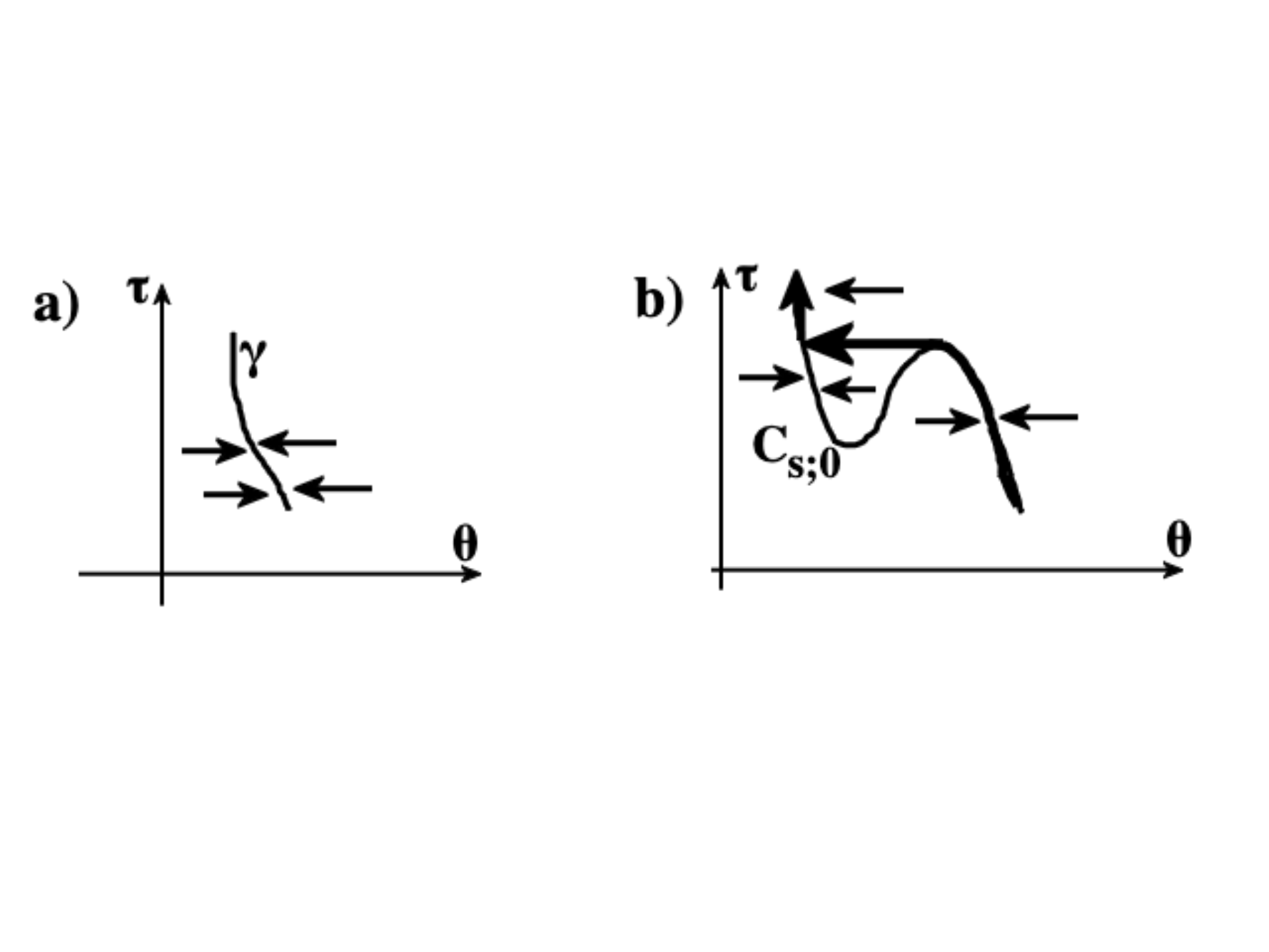, width=23em}
 \vspace{-1.5cm}
  \caption{a) A stable arc. b) A piece of singular orbit, marked in bold.} 
  \end{center}
\end{figure}
\begin{example} If  system (\ref{jos2g}) is horizontally positive (negative) at $z_0$, then the branches $C_{\mcr,-}$, $C_{\mcl,+}$ (respectively, $C_{\mcr,+}$, $C_{\mcl,-}$) in $C_{s;0}$ are stable, and its two other branches aren't. If in   (iv)  $\gamma$ is bijectively projected to $S^1_{\tau}$, as e.g. in the case (\ref{fluo}), then $\gamma$ oriented up is a singular orbit.
\end{example}
\begin{proposition} 1) Let conditions (i)--(iii) hold. Then  a closed singular orbit exists, if and only if the curve $C_{s;0}$ is projected to the whole circle $S^1_{\tau}$. In this case the latter orbit  is unique and depends continuously on $s$.  

2) If there exists a path $\gamma\subset C_{s;0}$ 
  satisfying  (iv)  without requirement of homotopy to  $\{0\}\times S^1_{\tau}$, then the above closed orbit 
   is homotopic to $\gamma$.
\end{proposition}
\begin{proof} The condition $\pi_\tau(C_{s;0})=S^1_\tau$ is obviously necessary for existence of a closed singular orbit. Now assuming it holds, let  us construct a closed singular orbit. Take the initial arc $C_{\mcl,+}(z_0)$ of the curve $C_{s;0}$ going up from the point $q_0=z_0$ and extend it up until it 
reaches a local maximum $p_0$ of the projection $\pi_\tau$. The arc $q_0p_0$ thus constructed is the initial   
arc of the singular orbit.   We consider the horizontal circle $L$ tangent to $C_{s;0}$ at $p_0$. The restriction of the field (\ref{jos2g}) with $\omega=0$ to a small punctured neighborhood of the point $p_0$ in $L$ is horizontal and  orients $L$ in the same way at all points of the neighborhood. The circle $L$ 
  intersects $C_{s;0}$ at   points different from $p_0$. 
  
  Take the segment $[p_0,q_1]$ of the circle $L$ bounded by the point $p_0$ and another point, denoted $q_1$, of intersection $L\cap C_{s;0}$, such that the interval $(p_0,q_1)$  is disjoint from 
  $C_{s;0}$ and  is oriented by the unperturbed horizontal field (\ref{jos2g}) with $\omega=0$ from $p_0$ to $q_1$. 
  The segment $[p_0,q_1]$ is the next arc of the singular orbit. There exists a unique stable arc of the curve $C_{s;0}$ adjacent to $q_1$ from above and oriented up. 
  Let us extend it to the next local maximum, denoted by $p_1$, and construct the point $q_2$ 
  as above etc.  Making this procedure   for all the local maxima $p_j$ finishes in a finite number of steps and  yields a singular orbit  $\wh\gamma=q_0p_0q_1p_1...$ finishing at the point $w_0$ by the stable arc $C_{\mcl,-}(w_0)$. Then we extend it further on by the arc $C_{\mcr,+}(w_0)$ and repeat the above procedure for the latter arc. 
  The singular orbit thus constructed closes up at $z_0$. Its  uniqueness and continuity in $s$ follows from definition and 
  (i)--(iii).

   Let now there exist a path $\gamma\subset C_{s;0}$ 
  satisfying  (iv). Let us show that the  singular orbit $\wh\gamma$ is homotopic to $\gamma$. It 
  suffices to show that all the above $p_j$ and $q_j$ lie in $\gamma$ and each arc $p_jq_{j+1}$ in $\gamma$ is homotopic to $[p_j,q_{j+1}]$ as a path with given ends. This is proved as follows. 
  The cylinder 
  $\Psi=S^1_\theta\times(\tau(z_0),\tau(w_0))$ intersects $C_{s;0}$ by two curves: the intersection 
  $\gamma_\Psi:=\gamma\cap\Psi$ and the other curve, denoted $\gamma'_\Psi$, both going from 
  the point $z_0$ to $w_0$. This follows from Condition (ii).  The curve $\gamma_\Psi$ is distinguished by the property that 
   each its point where it is directed up belongs to a stable arc of the curve $C_{s;0}$. Indeed, this holds near the positive flow point $z_0$, by construction, and for other points of the curve 
   $\gamma_\Psi$ it follows by continuity. Points of local going up in the other curve $\gamma'_\Psi$, contrarily,  do not belong to stable arcs. All the points $p_j$, $q_j$ belong to closures of stable arcs and hence, lie in $\gamma_\Psi$. 
  The  curve $\wh\gamma\cap\Psi$  cuts the cylinder $\Psi$ from lower base to upper base; the cut cylinder  is homeomorphic to a rectangle. 
  Each arc $p_jq_{j+1}$ lies on one side of the latter cut curve locally near  its endpoints, which follows by construction. Therefore, it is homotopic to the segment $[p_j, q_{j+1}]$. 
  \end{proof}

\begin{theorem} \label{thm2} Let a family (\ref{jos2g}) satisfy conditions (i)--(vi). Here we can replace (iv) by the condition of existence of a closed singular orbit homotopic to $\{0\}\times S^1_\tau$ for some $s$, and then it holds for every $s$, by continuity. 
 Let $L_r=L_r(\omega)\subset U\subset\rr^n_s$ denote its phase-lock areas in the $s$-parameter space with fixed 
 $\omega$; here $r$ denotes the rotation number value. Then for every integer value $r$ the phase-lock area $L_r$  exists and converges, see Footnote 3, to the preimage of the closure of the corresponding parquet domain  $L_r^0$ from Theorem \ref{thm1} under the composition map 
$$s\mapsto (c_+,c_-)\mapsto(\ell,u), \ \ \ell=\frac{c_+-c_-}2, \ u=\frac{c_++c_-}2.$$
If $U=\rr^2$, $s=(\ell,u)$ and the above map is the identity, then $L_r\to \overline{L^0_r}$. 
\end{theorem}
\begin{remark} In general, a family (\ref{jos2g}) may have phase-lock areas corresponding to fractional non-integer rotatiton number values for every fixed $\omega\neq0$. Theorem \ref{thm2} states that for every $m\in\zz$ 
the union of the  phase-lock areas $L_r$ with 
$m<r<m+1$ shrinks to the piecewise linear curve separating the limit phase-lock areas $\overline{L^0_m}$, $\overline{L^0_{m+1}}$, as $\omega\to0$, and  its width tends to zero. 
\end{remark} 
\begin{remark} J.-L. Callot introduced analogues of numbers $c_{\pm}$ in a particular case, for the class of slow-fast families given by Riccati equations of type $\var\frac{du}{d\tau}=-f(\tau)u+g(\tau,s)u^2+\var$,  with $\var>0$ small, $f(\tau)$, $g(\tau,s)$ being $2\pi$-periodic in $\tau$, $g>0$, and $f$ having two simple zeros on the circle. Note that after the variable change 
$u=\tan\frac{\theta}2$  graphs of their solutions are transformed to orbits of vector fields on torus 
$$\begin{cases}\dot\theta=-f(\tau)\sin\theta+g(\tau,s)(1-\cos\theta)+\var(1+\cos\theta)\\ \dot\tau=\var.
\end{cases}$$
 See his unpublished PhD thesis:  \cite[deuxi\`eme partie, chapitre 2, section 1]{callot}. An implicit version of the statement of Theorem \ref{thm2} in this particular case is presented there: see the beginning of p.85 in loc. cit.
  \end{remark}

\subsection{Plan of proof of Theorems \ref{thm1} and \ref{thm2}} 
For simplicity we first sketch the proof of Theorem \ref{thm1}. The proof of Theorem \ref{thm2} is analogous and will be discussed at the end of the subsection. 
Let us introduce the  auxiliary  coordinates 
$$c_{\pm}:=u\pm\ell$$
  on the rescaled parameter plane $\rr^2_{\ell,m}$.  In the coordinates $(c_+,c_-)$ the limit parquet domains $L_r^0$ from Theorem \ref{thm1} are
$$L^0_r=\cup_{m-k=r}Z_m\times Z_k\subset\rr^2_{c_+,c_-}; \ \ m,k\in\zz_{\geq0},$$
see (\ref{zmz0}), (\ref{lr0}). 
For every $m\in\zz_{\geq0}$ and $\var\in(0,1)$  set 
\begin{equation}Z_{m,\var}:=\begin{cases}[2m-1+\var,2m+1-\var], \ \text{ if } m\geq1\\
[-\frac1\var,1-\var], \ \text{ if } m=0.\end{cases}\label{zm}\end{equation}
These are compact sets  exhausting $Z_m$, as $\var\to0$. 

For the proof of Theorem \ref{thm1} we have to show that for every $m,k\in\zz_{\geq0}$,  $\var\in(0,1)$ 
and every $\omega$ small enough depending on $m$, $k$, $\var$ the rectangle $Z_{m,\var}\times Z_{k,\var}\subset\rr^2_{c_+,c_-}$ lies in the phase-lock area $L_r=L_r(\omega)$ of system (\ref{jos}) with the rotation number 
$r=m-k$. To do this, we prove existence of a $2\pi$-periodic orbit with this rotation number. 

The proof is based on  slow-fast system theory. 
In the rescaled parameters $(\ell,u)$ system (\ref{jos}) takes the form 
\begin{equation}\begin{cases}\dot\theta=f_{\ell,u;\omega}(\theta,\tau):=\cos\theta+\ell\omega+(1+u\omega)\cos\tau\\
\dot\tau=\omega.\end{cases}\label{jos2}\end{equation}
It is a slow-fast system with small parameter $\omega$. For  $\omega=0$ it degenerates to the fast system 
\begin{equation}\begin{cases}\dot\theta=f_0(\theta,\tau):=\cos\theta+\cos\tau\\
\dot\tau=0,\end{cases}\label{jos0}\end{equation}
We consider liftings of (\ref{jos2}), (\ref{jos0}) to the universal cover $\rr^2_{\theta,\tau}$. The phase portrait 
of the fast system (\ref{jos0}) is presented at Fig. 6. Its orbits are horizontal lines. 
The unperturbed slow curve $\{ f_0=0\}$ is the union of lines:  
$$C_0=\{ f_0=0\}=\cup_{j\in\zz}\cup_{\pm}\{ \tau=\pm\theta+\pi(2j+1)\}.$$
The intersection points of the latter lines are Morse critical points of the function $f_0$ 
that are split into the two following groups:

- the {\it right flow points }  $z_{j,m}:=(\pi(2j+1), 2\pi m)$, $j,m\in\zz$;

- the {\it left flow points }  $w_{j,m}:=(2\pi j, \pi(2m+1))$.

They are named so according to Definition \ref{horpos}.  The horizontal orbits of  (\ref{jos0}) crossing the positive (negative) flow points are directed to the right (respectively, left) on both sides from them. 
\begin{figure}[ht] \label{figfast}
  \begin{center}
  \hskip3cm \epsfig{file=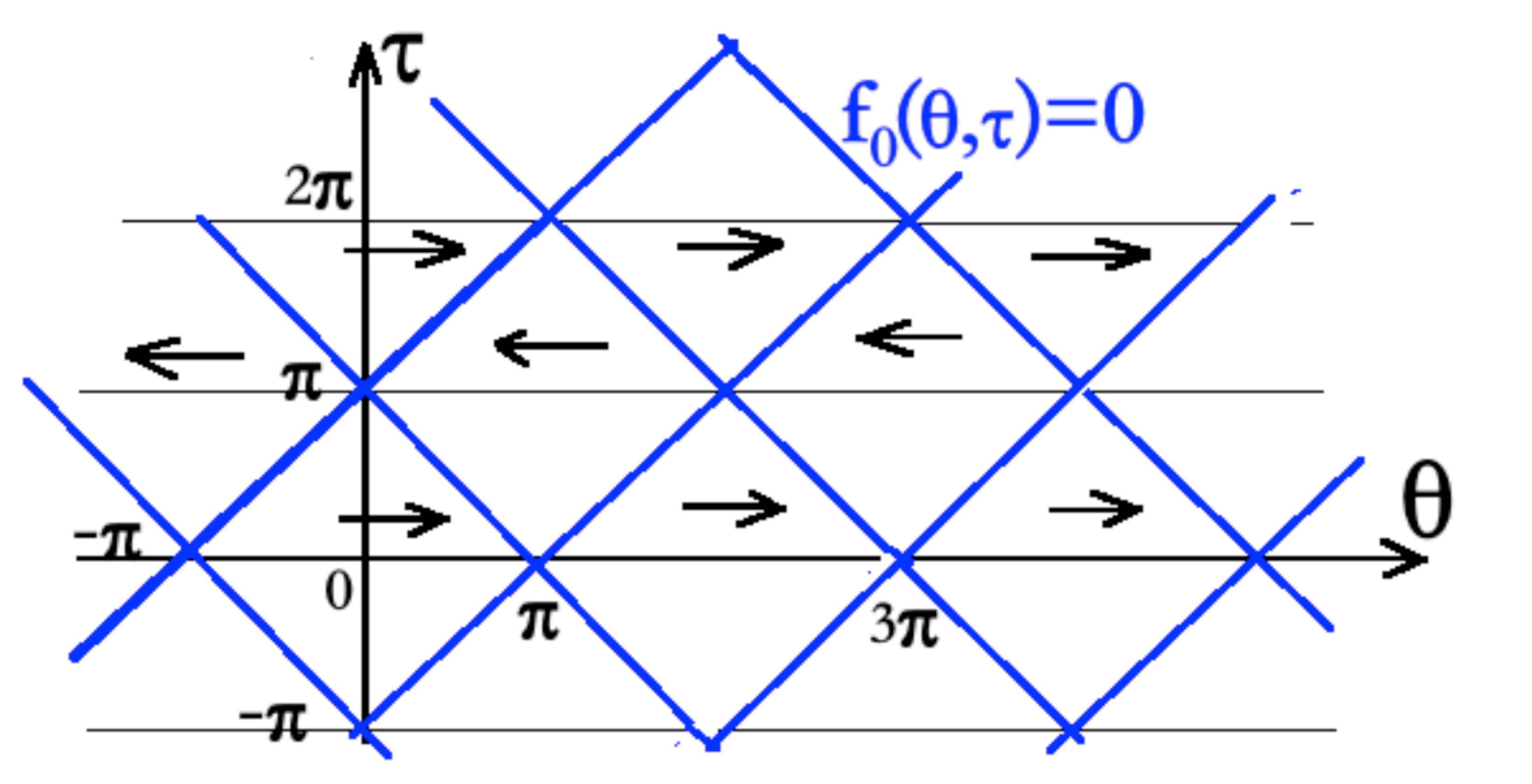, width=18em}
  \caption{Phase portrait of  field (\ref{jos0}) with $\omega=0$.} 
  \end{center}
\end{figure} 
Let us introduce the horizontal segments 
\begin{equation}I_0=[0,\pi]\times\{-\frac{2\pi}3\}, \ \ \ I_1:=[0,\pi]\times\{\frac\pi3\}.\label{i0i1}\end{equation}
Each of them crosses the curve $C_0$ and is contracted to itself by the flow of the unperturbed vector field (\ref{jos0}), 
which is directed to the right at its left end and to the left at its right end. 

Consider now the perturbed vector field (\ref{jos2}) lifted to $\rr^2$, with $\omega>0$ and with $(\ell,u)$ lying in 
a compact subset, e.g., the preimage of the set $Z_{m,\var}\times Z_{k,\var}\subset\rr^2_{c_+,c_-}$ under the map 
$(\ell,u)\mapsto(u+\ell, u-\ell)$. The slow  curve 
$$C_{\ell,u;\omega}=\{ f_{\ell,u;\omega}=0\}\subset\rr^2_{\theta,\tau}$$
converges to $C_0$, as $\omega\to0$. It is well-known from the slow-fast system theory that 
for every $\tau_1$, $\tau_2$, $0<\tau_1<\tau_2<\frac{2\pi}3$, for every $\omega$ small enough  
 the  forward orbit of each segment $I_j$, $j=0,1$, under the flow of (\ref{jos2}) in times $t\in[\frac{\tau_1}\omega,\frac{\tau_2}\omega]$ is a so-called stable flowbox, denoted $F_-(I_j)$,   $O(\omega)$-close to a stable arc of the slow curve $C_{\ell,u;0}$, and thus, to an arc of the curve 
 $C_{\ell,u;\omega}$. The latter arc  is 
the graph of a smooth function $\theta=\theta(\tau)$. The horizontal width of the flowbox 
is less than $\exp(-\frac d{\omega})$ with some constant $d$ independent on $\omega$. See Proposition \ref{pflboxes} and Figure 7 drawn for  
$u\pm\ell>0$. The  upward extension of the above  arc  may reach a local maximum  of the coordinate $\tau$ on  $C_{\ell,u;\omega}$. As the forward orbit of the flowbox reaches the latter maximum, it goes immediately to the right for $j=0$ (to the left for $j=1$).   
\begin{figure}
 \begin{center}
\epsfig{file=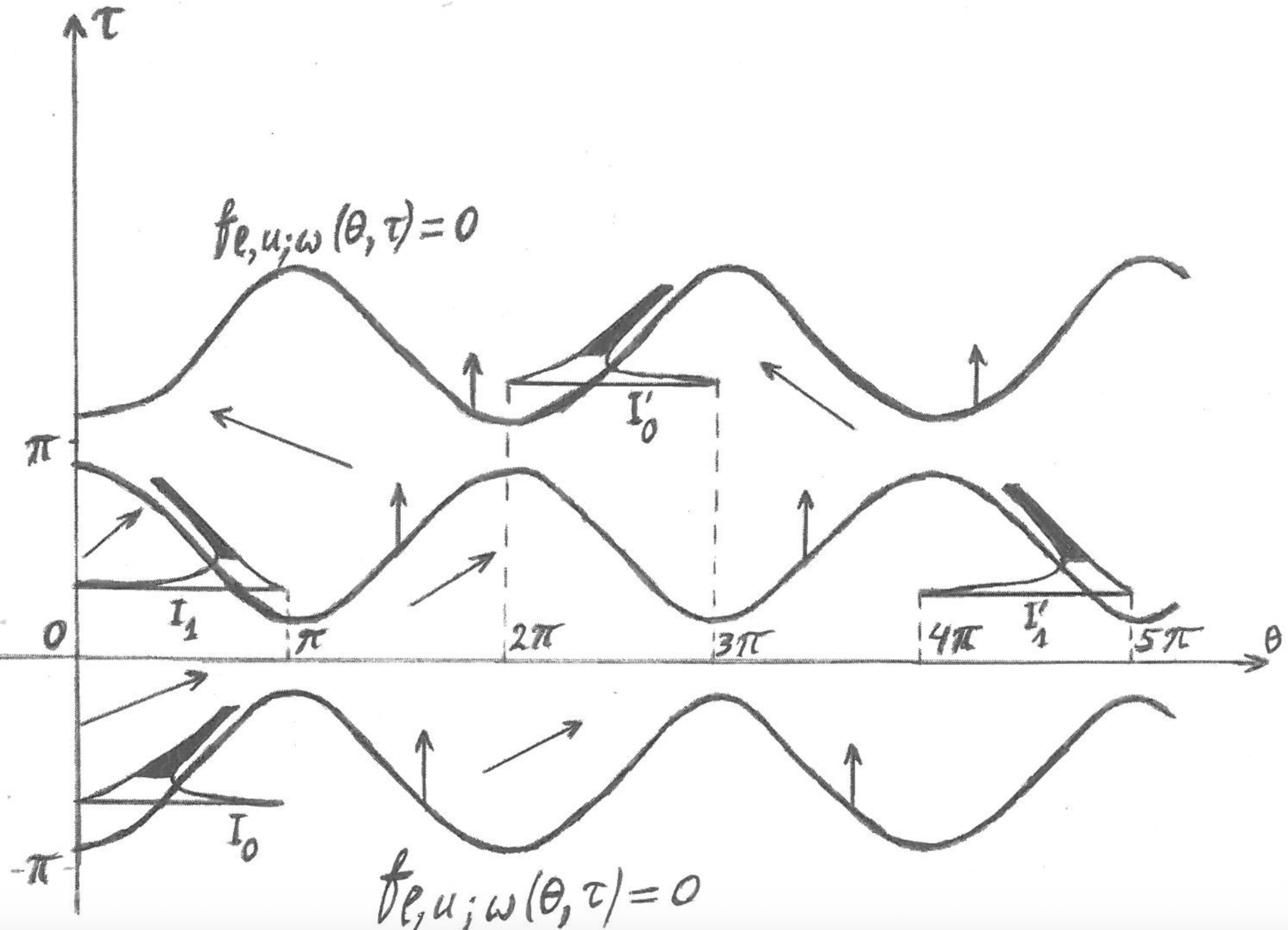, width=27em}
\caption{Phase portrait of field (\ref{jos2}) with $c_{\pm}=u\pm\ell>0$ and small $\omega>0$. Segments $I_j$, their periodic shifts,  stable flowboxes marked in bold thick.}
  \end{center}
\end{figure} 

Using the latter facts and a detailed analysis of what happens to the flowbox after it passes near the maximum we prove the following theorem, which will imply Statement 1) of Theorem \ref{thm1}.

Everywhere below for a subset $I\subset\rr^2_{\theta,\tau}$ and a vector field on $\rr^2$, usually either (\ref{jos2}), 
or (\ref{jos2g}), by $\mco(I)$ we denote its forward orbit under the flow 
of the  field in question.
\begin{theorem} \label{mt} For every $m,k\in\zz_{\geq0}$,  $\var\in(0,1)$ for every $\omega>0$ small enough depeding on 
$m$, $k$, $\var$ for every $(\ell,u)$ such that $c_+\in Z_{m,\var}$, $c_-\in Z_{k,\var}$ 

 1) the orbit $\mco(I_0)$ by  (\ref{jos2}) enters  $\opint(I_1')$, $I_1'=I_1+(2\pi m,0)$;
 
 2) the  orbit $\mco(I_1')$ enters $\opint(I_0')$, $I_0'=I_0+(2\pi(m-k),2\pi)$. See Fig. \ref{fig:orb}.
\end{theorem}
 \begin{figure}
 \begin{center}
\epsfig{file=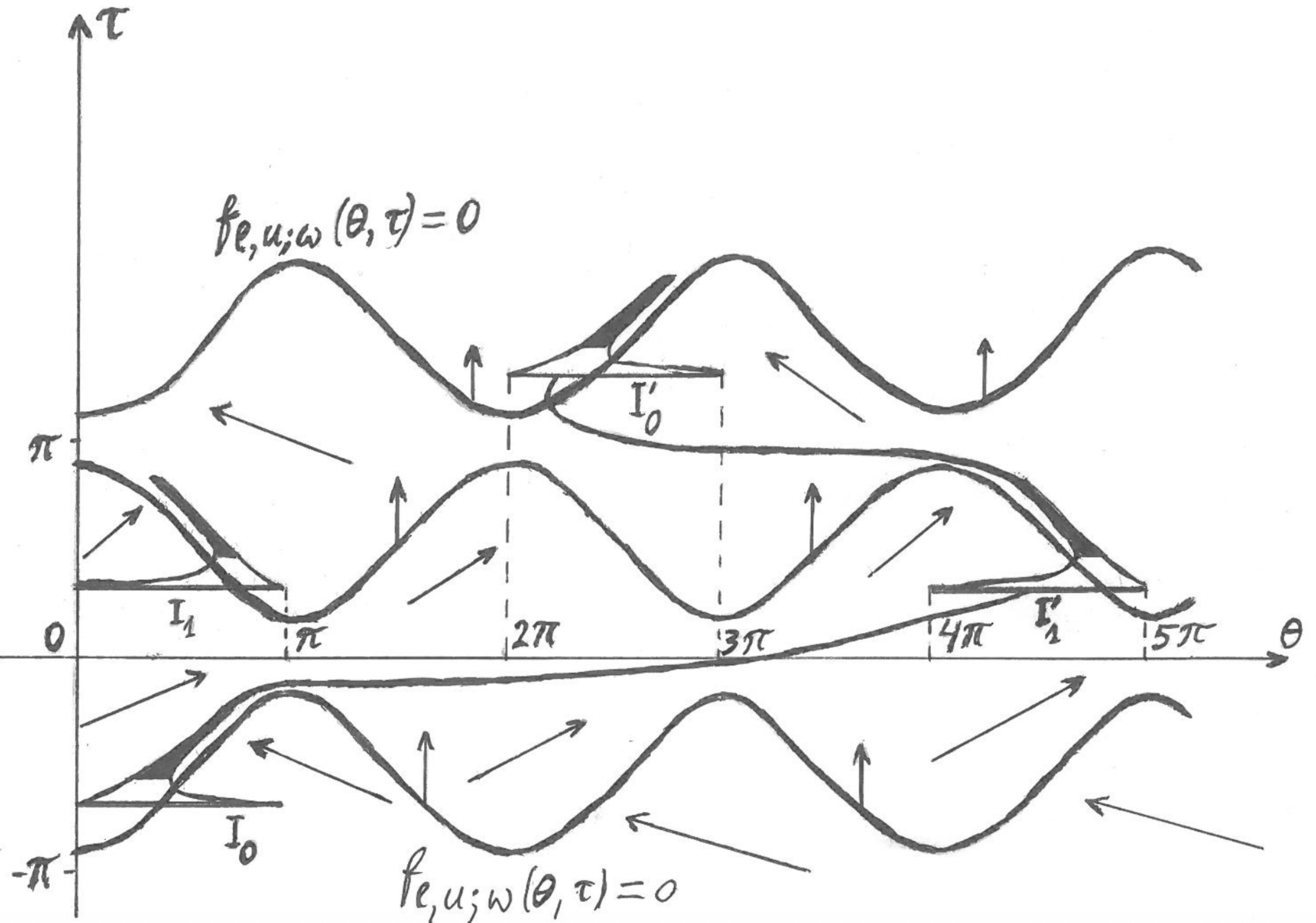, width=27em}
\caption{Theorem \ref{mt} on phase portrait of field (\ref{jos2}). Here $c_{\pm}>0$.}
\label{fig:orb}
  \end{center}
\end{figure} 
\begin{corollary} \label{corparq} In the above conditions for every $\omega$ small enough depending on $m$, $k$, $\var$ 
the segment $I_0$ contains a point lying in an attracting $\frac{2\pi}{\omega}$-periodic orbit of the flow (\ref{jos2}) as a flow on $\tt^2$  with rotation number $m-k$. 
\end{corollary}
\begin{proof} The  time $\frac{2\pi}{\omega}$ flow map, i.e., the Poincar\'e map $\{\tau=0\}\to\{\tau=2\pi\}$  sends $I_0$ to the interior of its image $I_0'$ under translation by 
the period lattice vector $(2\pi(m-k),2\pi)$, by Theorem \ref{mt}. The segment $I_0'$ being identified with $I_0$ by the above translation, the Poincar\'e map has  a fixed point, being a self-map of a segment. The fixed point is unique and attracting, since the Poincar\'e map is M\"obius, see \cite{buch2}, and its sends $I_0$ strictly inside itself. The rotation number of the corresponding   
$\frac{2\pi}{\omega}$-periodic orbit of the flow is equal to $m-k$. This proves the corollary. 
\end{proof}

Statement 1) of Theorem \ref{thm1} follows from the corollary and the discussion at the beginning of the subsection. See Subsection 2.7 for more details. 

Consider the auxiliary segments $\wh I_0:=I_0+(\pi,0)$, $\wh I_1:=I_1+(\pi,0)$. On the torus, the unions 
$I_0\cup\wh I_0$, $I_1\cup\wh I_1$ are the circles $S^1_{-\frac{2\pi}3}=S^1\times\{-\frac{2\pi}3\}$ and $S^1_{\frac{\pi}3}=S^1\times\{\frac{\pi}3\}$ respectively. The former circle is sent onto the latter by the Poincar\'e map $h$: the time $\frac{\pi}{\omega}$ flow map. Statement 1) of Theorem \ref{mt} holds for some $m$, if and only if the $h$-image of each point 
in $I_0$ lies in $\operatorname{Int}(I_1)$, not in  $\wh I_1$, and similar criterion holds for Statement 2) and segments $I_1$, $I_0$, $\wh I_0$.

The backward orbits of the segments $\wh I_j$, $j=0,1$ under flow  (\ref{jos2}) in times 
$t\in[\frac{\tau_1}\omega,\frac{\tau_2}\omega]$, $0<\tau_1<\tau_2<\frac{\pi}3$,  are so-called unstable flowboxes denoted $F_+(\wh I_j)$,  lying on $O(\omega)$-distance from the curve $C_{\ell,u;\omega}$ and also having exponentially 
small widths, as do $F_-(I_j)$.
\begin{remark}   The above Poincar\'e map statement for $I_j$ holds if and only if each orbit going out of  
the flowbox $F_-(I_j)$ meets $\operatorname{Int}(I_{1-j})$ in time less than $\frac{\pi}\omega$. In other terms, this shows that if for small 
$\omega$  we have a transition  from one phase-lock area to another one along a path in the $(\ell,u)$-plane,  then the latter path  contains a parameter value $(\ell,u)$ for which there is an orbit, called duck or canard, that crosses both stable and unstable flowboxes $F_-(I_j)$, $F_+(\wh I_{1-j})$ for a certain $j=0,1$ (discussion with A.A.Alexandrov). Canards were studied in \cite{chasse}. 
Canard limit cycles were studied in \cite{ilguk} for a subfamily of (\ref{jos2}) and in \cite{callot} for a 
large class of Riccati equations.
\end{remark}

Below we sketch the proof of Theorem \ref{mt}. The steps of proof are described below in the conditions of Theorem \ref{mt} just for simplicity. But in fact,  in Section 2 we do them  in the conditions of more general Theorem \ref{mt2} stated below, which implies Theorem \ref{thm2}. 

Step 1. Local asymptotic model of system (\ref{jos2}) given by Riccati equation 
near a right or left flow point $(\theta_0,\tau_0)$. We show that the rescaling 
$$ (\theta,\tau)\mapsto(x,y):=\left(\pm\frac{\theta-\theta_0}{\sqrt{2\omega}},\frac{\tau-\tau_0}{\sqrt{2\omega}}\right), \ \ + \text{ for right,} \ - \text{ for left,}$$
transforms the line field given by (\ref{jos2}) to a line field converging to the field 
\begin{equation} \frac{dx}{dy}=x^2-y^2+c, \ \ \ c=c_{\pm}=u\pm\ell,\label{ricric}\end{equation} 
as $\omega\to0$. See Subsection 2.1. This result, whose equivalent version is due to J-L.Callot \cite{callot}, extends a well-known similar result due to E.F.Mishchenko and L.S.Pontryagin near  a  local maximum of the coordinate $\tau$ on the unperturbed slow curve, see \cite[section 8]{miroz}, \cite{pontr}, \cite[chapter 4, section 3]{dyn5}. The coordinates $(x,y)$ will be called the {\it Riccati coordinates.} 

Step 2, Subsections 2.2--2.4. Properties of Riccati equation (\ref{ricric}) and its solutions. Equation (\ref{ricric})  
yields an analytic line field on the cylinder $\rp^1\times\rr_y$, $\rp^1=\rr_x\cup\{\infty\}$. 
Each its solution $x(y)$ extends meromorphically to all of $\cc_y$. We will deal with real solutions on $\rr_y$ 
as analytic $\rp^1$-valued functions that may have real poles. 

The three following results contained in \cite{callot, sib75} will be used in the proofs. They are recalled in Subsection 2.2--2.4 with brief proofs.

- For every $c\in\rr$ equation (\ref{ricric}) has  unique so-called {\it stable solution} $x_-(y)$ and {\it unstable solution} $x_{+}(y)$ with 
asymptotics  
$$x_{\pm}(y)\simeq y, \ \text{ as } y\to\pm\infty; \ \ \  x_+(y)=-x_-(-y).$$

- If $x_+(y)=x_-(y)$ as global  functions on $\rr$, then  we say that a {\it heteroclinic connection} takes place. This  
 happens if and only if $c=2m+1$, $m\in\zz_{\geq0}$.  
This comes from the fact that (\ref{ricric}) is the projectivization of  eigenfunction equation for the quantum harmonic oscillator with eigenvalue $c$, equivalence of  heteroclinic connection and existence of an $L_2$-eigenfunction,  and well-known description of its $L_2$-spectrum as $\{2m+1 \ | \ m\in\zz_{\ge0}\}$. 

- If $c\in Z_m$, then $x_-(y)$ has $m$   simple real poles with residue -1, and $x_-(y)\simeq -y$, as $y\to+\infty$.

For every $a>0$ we set 
$$Q_a:=\{|x|<2a, \ \ |y|<a\}\subset\rr^2_{x,y}.$$
For every right or left flow point $z_0=(\theta_0,\tau_0)$ by $Q_a(z_0)$ we denote its rectangular neighborhood given 
by $Q_a$ in the corresponding Riccati coordinates $(x,y)$. In the initial coordinates it is the rectangle
$$Q_a=\{|\theta-\theta_0|<2a\sqrt{2\omega}, \ |\tau-\tau_0|< a\sqrt{2\omega}\}.$$
It follows from the above statement that for every $\var>0$, $m\in\zz_{\geq0}$, $c\in Z_{m,\var}$ and $a$ large enough depending on $m$, $\var$ the graph $\{ x=x_-(y)\}$ intersects $Q_a$ by $m+1$ 
arcs $X_{0,a},\dots, X_{m,a}$. Here each two neighbor arcs $X_{j,a}$, $X_{j+1,a}$ are separated by exactly one infinite point $(\infty,y_*)$ of the graph, where $y_*$ is a real pole of the function $x_-(y)$. The orbit arc $X_{m,a}$ ends at the upper side of the rectangle $Q_a(z_0)$, at a point 
with $x$-coordinate lying in $(-\frac{11}{10}a, -\frac9{10}a)$.  

Step 3, Subsection 2.5. We extend the orbit $\mco(F_-(I_0))$ to a $O(\sqrt\omega)$-neighborhood of the point $z_0$. We show that its intersection with  the rectangle $Q_a(z_0)$ consists of orbits of (\ref{jos2}) converging uniformly to the arc $X_{0,a}$ of the graph 
 $\{ x=x_-(y)\}$ of the stable solution of the Riccati equation (\ref{ricric}) with $c=c_+$ in the 
Riccati coordinates $(x,y)$. This will be deduced by an equivalent characterization of the stable solution $x_-(y)$ as limit of 
solutions with initial conditions $(x_n,y_n)$ where $y_n\to-\infty$ and $x_n\leq0$ via an a priori bound on the orbit of 
$F_-(I_0)$.

Step 4, Subsection 2.6. We take arbitrary $(c_+,c_-)\in Z_{m,\var}\times Z_{k,\var}$ and consider the subsequent right flow points 
$z_0,\dots,z_m$, $z_j=z_0+(2\pi j,0)$. We consider the rectangles $Q_a(z_j)$ identified with $Q_a(z_0)$ by 
translations $(\theta,\tau)\mapsto(\theta-2\pi j,\tau)$. We also identify the corresponding Riccati coordinates  
$(x,y)$ centered at $z_j$, $x=x(\theta,\tau)$, $y=y(\theta,\tau)$, by pre-composition by the same translations. 
We show that for every $a>1$ big enough and every $\omega>0$ small enough depending on $a$ and $\var$ further extension of the orbit $\mco(F_-(I_0))$ crosses the rectangles 
$Q_a(z_j)$, $j=1,\dots,m$, by families of arcs of orbits of (\ref{jos2}) that converge uniformly to the arcs  $X_{j,a}$ of graph 
of the same solution $x_-(y)$ of the Riccati equation (\ref{ricric}) with $c=c_+$. To do this, we consider  bigger rectangles $Q_R$  and show that 
each orbit in $\mco(F_-(I_0))$ crosses  the right lateral side of $Q_R(z_j)$ inside its subsegment having the following property. Each orbit starting at a point of the latter subsegment 
arrives to a point in the left lateral side of $Q_R(z_{j+1})$, is directed up-right, and the increment of the coordinate  
$\tau$ along the  orbit arc between the two latter points is bounded by a quantity $O(\frac{\sqrt\omega}{R})$. This will be deduced from an a priori lower bound  $f_{s;\omega}(\theta,\tau)>\sigma_1(\theta-\theta(z_i))^2$, $i=j, j+1$, with $\sigma_1>0$ independent on $R,\omega$, in appropriate trapezoids with vertical bases, one being a subsegment in one of the above lateral sides, the other lying in the line $\{\theta=\theta(z_j)+\pi\}$.   

Let now $j=0$. For  every orbit $\mco$ in $\mco(F_-(I_0))$ its intersection with $Q_R(z_0)$ converges to $X_{0,R}$, see Step 3. It crosses the right lateral side in $Q_R(z_0)$ and left lateral side in $Q_R(z_1)$ at points $(2R,y_1)$, $(-2R,y_2)$ respectively in the Riccati coordinates, with $y_2-y_1=O(\frac1R)$, as $R$ is large enough and $\omega$ is small depending on $R$, by the above bounds. Thus, the points $(2R, y_1)$, $(-2R, y_2)$ are close  in the cylinder $\rp^1\times\rr_y$, and $(2R,y_1)$ converges to the endpoint of the arc $X_{0,R}$. Hence, they are both close to the infinite point $(\infty,y_*)\in\{ x=x_-(y)\}$ separating the arcs $X_{0,R}$ and $X_{1,R}$. 

Take an arbitrary sequence of orbits $\mco=\mco_n$ as above with $\omega=\omega_n\to0$. It contains  
 a subsequence along which the corresponding points $y_2=y_{2,n}$ converge to some $z$, and thus, the intersections $\mco_n\cap Q_R(z_1)$  converge to  an orbit of Riccati equation (\ref{ricric}) with $c=c_+$, with initial condition  $(-2R,z)$,  close to 
 $(\infty,y_*)$. 
 The forward orbit of the point $(\infty,y_*)$ crosses 
  the rectangle $Q_a(z_1)$ along the arc $X_{1,a}$. Therefore, the above limit orbit  
  crosses $Q_a(z_1)$ along an arc close to $X_{1,a}$, whenever $R$ is large enough.   Since $R$ is arbitrarily big, we get that the  intersection $\mco_n\cap Q_a(z_1)$  converges to $X_{1,a}$ along a subsequence. This implies that $\mco(F_-(I_0))$ crosses $Q_a(z_1)$ along orbit arcs converging uniformly to $X_{1,a}$, as $\omega\to0$. 
Then we get by  induction in $j$   that  $\mco(F_-(I_0))$ crosses each 
$Q_a(z_j)$, $1\leq j\leq m$, by  arcs converging to $X_{j,a}$.

The next steps 5, 6 finishing the proofs of Theorems \ref{mt}, \ref{mt2} and thus, Statement 1) of Theorem \ref{thm1}, are done in Subsection 2.7. 

Step 5. We take the rectangle $Q_R(z_m)$ and the subinterval $\{-\frac{11}{10}R<x<-\frac9{10}R\}$  in its upper side, which contains the endpoint of the arc $X_{m,R}$. 
As $\omega$ is small enough, the orbit $\mco(F_-(I_0))$ crosses the upper side of $Q_R(z_m)$ by 
a segment lying in the latter interval; this follows from results of Steps 2, 4. 
We show that the forward orbit of the closure of the above subinterval drifts along the curve $C_{\ell,u;\omega}$ until it crosses the interior of the 
segment $I_1'$. This is proved analogously to the proof of the result of Step 3. This proves Statement 1) of Theorem \ref{mt}.

Step 6:  proof of Statement 2) of Theorem \ref{mt}. We consider the narrow stable flowbox $F_-(I_1')=F_-(I_1)+(2\pi m,0)$ 
and study its forward orbit. It crosses the rectangle $Q_R(w_0)+(2\pi m,0)$ by arcs that all converge uniformly to an arc of the graph of the stable solution $x_-(y)$ of (\ref{ricric}) with $c=c_-$, as in Step 3. Recall that $c_-\in Z_{k,\var}$. Afterwards applying the arguments in Steps 4 and 5 with $c_+$ replaced by $c_-$ and $z_j$ replaced by $\wh w_j=w_0+(2\pi(m-j),0)$, $j=0,\dots,k$, we get that the  orbit  $\mco(I_1')$ enters the interior of the segment $I_0'$. This will prove Statement 2) of Theorem \ref{mt}. 

The proof of Statement 1) of Theorem \ref{thm2} repeats the above-sketched proof of Theorem \ref{thm1}.  It follows from the next generalization of Theorem \ref{mt}. To state it, let us introduce new segments $I_{0,1}$ in the more general settings of Theorem \ref{thm2}. Let $z_0=(\theta_0,\tau_0)$ denote the right flow critical point. Consider the germ at $z_0$ of the 
zero curve $C_{s;0}$. Let $C_{\mcl}$, $C_\mcr$ denote its components intersecting at $z_0$, and let 
$C_{\mcl,\pm}$, $C_{\mcr,\pm}$ denote their upper and lower parts, see (\ref{clr}). We fix  a neighborhood
$|\tau-\tau_0|<M$, $0<M<\frac{\pi}3$,  of the point $\tau_0$, such that the intersection of the curve $C_{s;0}$ with the strip $\rr_{\theta}\times[\tau_0-M,\tau_0+M]$ 
consists of two arcs lying in the curves $C_{\mcl}$ and $C_\mcr$ that are graph of $C^1$-smooth functions $\theta=\theta(\tau)$. 
Let $\Xi_{\pm}$ denote the open  curvilinear triangle bounded by $C_{\mcr,\pm}$, $C_{\mcl,\pm}$ and the line $\{\tau=\tau_0\pm M\}$. See Fig. 9. The function $f_{s;0}$ is negative on  
$\Xi_{\pm}$  and on its horizontal side with endpoints deleted. It vanishes on its lateral sides. 
Fix two arbitrary points $b_0\in\Xi_-$, $b_1\in\Xi_+$ and horizontal segments $I_j$: 
\begin{equation} I_j=[a_j,b_j], \ \ \tau(a_j)=\tau(b_j), \ \theta(a_j)<\theta(b_j), \ \ a_j\notin \cup_{\pm}\overline{\Xi_{\pm}}, \ j=0,1.\label{i0i1n}\end{equation}
Then $I_0$ intersects $C_{\mcr,-}$ and not $C_{\mcl,-}$, $I_1$ intersects 
$C_{\mcl,+}$ and not $C_{\mcr,+}$, 
 $$f_{\ell,u;0}(a_{j})>0>f_{\ell,u;0}(b_{j}).$$
 \begin{figure}
 \begin{center}
\epsfig{file=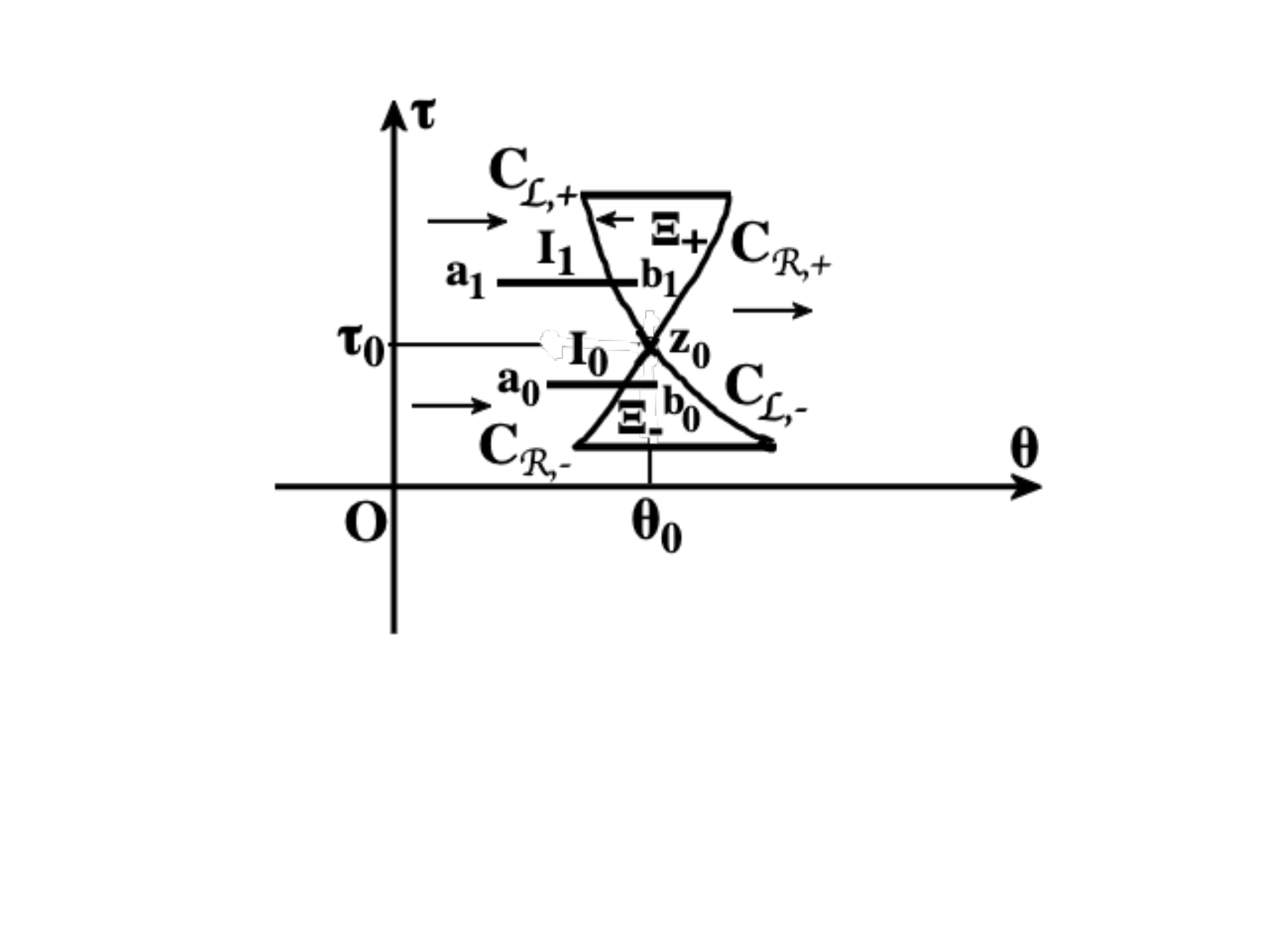, width=25em}
\vspace{-1.8cm}
\caption{The triangles $\Xi_{\pm}$ and the segments $I_0$, $I_1$.}
\label{fig:i01}
  \end{center}
  \end{figure} 
\begin{theorem} \label{mt2} Let a family (\ref{jos2g}) depending on parameters $s\in U\subset\rr^n$ and  
$\omega\in(-d,d)$, $d>0$, satisfy conditions (i)--(iv) in Subsection 1.2, where (iv) is allowed to be modified as in Theorem \ref{thm2}.  In the case, when either $n\neq2$, or the map $\Sigma:s\mapsto(c_+,c_-)$ is not a diffeomorphism $U\to\rr^2$, we pass to the restriction of (\ref{jos2g}) to a compact subset $\wt U\Subset U\subset\rr^n_s$.  Then for every $m,k\in\zz_{\geq0}$,  $\var\in(0,1)$ for every $\omega>0$ small enough depending on $m$, $k$, $\var$, $\wt U$ the following statement holds. Let  one have 
$c_+\in Z_{m,\var}$, $c_-\in Z_{k,\var}$, see (\ref{defc}).  Then 
 
 1) the orbit $\mco(I_0)$  enters  $\operatorname{Int}(I_1')$, $I_1'=I_1+(2\pi m,0)$.
 
 2) the  orbit $\mco(I_1')$ enters $\operatorname{Int}(I_0')$, $I_0'=I_0+(2\pi(m-k),2\pi)$.
\end{theorem}
\begin{remark} In the above conditions in $O(\sqrt\omega)$-neighborhood of right or left flow point there exists an 
affine coordinate change transforming line field given by (\ref{jos2g}) to a field converging to Riccati equation (\ref{ricric}) 
with $c=c_{\pm}$ the same, as in Theorem \ref{thm2}. But now the above affine coordinate changes 
are not just compositions of homotheties and translations. See Subsection 2.1. In the coordinates $(\theta,\tau)$ the  rectangles $Q_a$ become  parallelograms depending on the parameters $s$, as do the above affine transformations. 
\end{remark}
\subsection{Constrictions. Plan of proof of Statement 2) of Theorem \ref{thm1}} 
In the proof of Statement 2) of the Parquet Theorem we use symmetry 
$$(\theta,\tau)\mapsto(-\theta,-\tau)$$
of system (\ref{jos2}), see \cite{RK}. In the case, when its Poincar\'e  time $\frac{2\pi}{\omega}$  flow map $h:S^1_{\theta}\times\{0\}\to S^1_{\theta}\times\{2\pi\}$ is hyperbolic, i.e., has one attracting and one repelling fixed point, the symmetry permutes them. If it is parabolic and not the identity, its unique  fixed point is parabolic and fixed by  the involution $\theta\mapsto-\theta$: it is either zero, or $\pi$.

A priori it could happen that a vertex $X=(r,r+2k-1)$, $r\in\zz_{\geq0}$, $k\in\nn$, of the limit domain $L_r^0$ is not a limit of constrictions, but is a result of squeezing of $L_r(\omega)$, as $\omega\to0$. To show that $X$ is  a limit of constrictions, we prove the next proposition. It implies  that as we cross  $X$ along a small vertical segment centered at $X$ of length $2\var<1$, the  attractor of the Poincar\'e map  jumps from one half-circle $(0,\pm\pi)$ to the other one, whenever $\omega$ is fixed and is small. 
Hence, there is an intermediate value 
$u=u(\omega)\in(r+2k-1-\var, r+2k-1+\var)$ such that for $(\ell,u;\omega)=(r,u(\omega);\omega)$ the  Poincar\'e map fixes  either $0$, or $\pi$, and hence, $(r,u(\omega))$ lies in the boundary of a phase-lock area. Together with the inequality 
$u(\omega)>r-\frac12$ and \cite[lemma 5.1, p. 5462]{bibgl}, this implies that $(r,u(\omega))$ is a constriction lying in $L_r(\omega)$.

\begin{proposition} \label{adcomp} Consider an arbitrary vertex $X=(r,r+2k-1)$, $r\in\zz_{\geq0}$, $k\in\nn$, 
of the limit parquet domain $L^0_r$. Let $V$ and $W$ denote its components adjacent to $X$ from above 
and from below respectively. Take arbitrary compact subsets $\wt V\Subset V$, $\wt W\Subset W$. 
Then for every $\omega>0$ small enough depending on $r$, $k$, $\wt V$, $\wt W$ the following statement holds.    Consider  flows  (\ref{jos2}) with arbitrary 
parameters $\psi_1=(\ell_1,u_1)\in\wt V$, $\psi_2=(\ell_2,u_2)\in\wt W$
let us denote them by (\ref{jos2})$_j$, $j=1,2$. The Poincar\'e maps $S^1_\theta\times\{0\}\to S^1_\theta\times\{0\}$ of both (\ref{jos2})$_j$ 
 have  fixed points $(\theta_j,0)$ with $\theta_j(\mo2\pi\zz)$ lying on different half-circles $(0,\pm\pi)$.  They are attracting.
 \end{proposition}

\section{Proofs of main results}
\subsection{Local rescaling to asymptotically Riccati equations}
The next proposition is a version of a similar result from \cite{callot} stated for appropriate class of slow-fast families of Riccati equations. It extends a well-known result due to E.F.Mishchenko and L.S.Pontryagin for regular quadratic turning point of slow-fast system, see \cite[section 8]{miroz}, \cite{pontr}, \cite[chapter 4, section 3]{dyn5}, to a Morse critical point. 
\begin{proposition} \label{proresc}  For every right (left) flow point $(\theta_0,\tau_0)$ of (\ref{jos2}) the variable change 
\begin{equation} (\theta,\tau)\mapsto(x,y):=\left(\pm\frac{\theta-\theta_0}{\sqrt{2\omega}},\frac{\tau-\tau_0}{\sqrt{2\omega}}\right)
\label{change+}\end{equation}
with  "$+$" for right, "$-$" for left, transforms the line field defined by system (\ref{jos2}) to the line field given by a differential equation 
\begin{equation} \frac{dx}{dy}=x^2-y^2+c_++O(\omega), \ \ c_\pm:=u\pm\ell,\label{ric+}\end{equation} 
where the $O(\omega)$ is uniform on compact subsets in $\rr^2_{x,y}\times\rr^2_{\ell,u}$. 
\end{proposition}
\begin{remark} The proposition implies that if $u\pm\ell>0$ then for every $\omega>0$ small enough the zero locus 
\begin{equation}C_{\ell,u;\omega}:=\{ f_{\ell,u;\omega}=0\}\subset\rr^2_{\theta,\tau}\label{glu}\end{equation}
has  topological type presented at Figures 7, 8. 
\end{remark}
\begin{proof} {\bf of Proposition \ref{proresc}.} 
Let $(\theta_0,\tau_0)$ be a right flow point. Then it is equal to $(\pi,0)(\operatorname{mod}2\pi\zz^2)$. Set 
\begin{equation}\psi:=\theta-\theta_0, \ \  \xi:=\tau-\tau_0; \ \ \ \psi=\sqrt{2\omega}x, \ \ \xi=\sqrt{2\omega}y.\label{thxy}\end{equation}
One has $\cos\theta=-\cos\psi=-1+\frac{\psi^2}2+O(\psi^4)$, 
 $\cos\tau=\cos\xi=1-\frac{\xi^2}2+O(\xi^4)$. 
Substituting the two latter formulas to (\ref{jos2}) yields 
\begin{equation}\begin{cases}\dot{\psi}=f_{\ell,u;\omega}(\theta,\tau)=\frac{\psi^2-\xi^2+2\omega(\ell+u)}2+O(\omega\xi^2)+
O(\psi^4)+O(\xi^4)\\
\dot{\xi}=\omega.\end{cases}\label{fnew+}\end{equation}
Diviging the first equation in (\ref{fnew+}) by the second one and rescaling to $(x,y)$, see (\ref{thxy}), yields (\ref{ric+}) with "$+$". 

Case of left flow point $(\theta_0,\tau_0)$,  which is now  $(0,\pi)(\operatorname{mod}2\pi\zz^2)$, is treated analogously. Now $\cos\theta=\cos\psi$, $\cos\tau=-\cos\xi$, and we get 
\begin{equation}\begin{cases}\dot{\psi}=f_{\ell,u;\omega}(\theta,\tau)=\frac{-\psi^2+\xi^2-2\omega(u-\ell)}2+O(\omega\xi^2)+
O(\psi^4)+O(\xi^4)\\
\dot{\xi}=\omega.\end{cases}\label{fnew-}\end{equation}
Dividing the first equation in (\ref{fnew-}) by the second one and rescaling by $\psi=-\sqrt{2\omega}x$, 
$\xi=\sqrt{2\omega}y$ 
yields (\ref{ric+}) with "$-$".
\end{proof}

Let us now consider the case of Theorem \ref{thm2}.

\begin{proposition} \label{progenr} Let $f_{s;\omega}(\theta,\tau)$, $s\in U\subset\rr^n$, be a horizontally positive (negative) germ at a point 
$O=(0,0)$, see Definition \ref{horpos}. Let $a_{ij}=a_{ij}(s)$ be the coefficients of the Hessian form $Hf_{s;0}(O)$, see (\ref{hesf}), taken with sign "$+$" for positive and "$-$" for negative, and let $\Delta$, $\nu$, $c=c_{\pm}(s)$,  be the same, as  
in (\ref{hfelu2})--(\ref{defc}). 
Then the  variable change 
\begin{equation}(\theta,\tau)\mapsto(x,y), \ \ \theta=\la_1\sqrt{2\omega}(\pm x-\mu y), \ \ \tau=\la_2\sqrt{2\omega} y,
\label{varch}\end{equation}
$$\la_1=a_{11}^{-1}\Delta^{\frac14}, \ \ \la_2=\Delta^{-\frac14}, \ \ \mu=a_{12}\Delta^{-\frac12},$$
transforms (\ref{jos2g}) to a system converging as $\omega\to0$ to  Riccati equation (\ref{ricric})  with 
 $c=c(s)$ given by (\ref{defc}) uniformly on compact subsets in $\rr^2_{x,y}\times U$. 
\end{proposition}
\begin{proof} Let us find a variable change of type (\ref{varch}) transforming (\ref{jos2g}) to a system converging to 
(\ref{ricric}). Substituting (\ref{varch}) to (\ref{jos2g}) yields, as $\omega\to0$: 
$$\la_1\sqrt{2\omega}(\pm\dot x-\mu\dot y)=\pm\omega(a_{11}\la_1^2(\pm x-\mu y)^2+2a_{12}\la_1\la_2(\pm x-\mu y)y$$
$$+a_{22}\la_2^2y^2+\nu+o(1)), \ \ \ \la_2\sqrt{2\omega}\dot y=\omega.$$
Dividing the former equation by the latter yields
$$\frac{\la_1}{\la_2}(\pm x'_y-\mu)=\pm(a_{11}\la_1^2(\pm x-\mu y)^2+2a_{12}\la_1\la_2(\pm x-\mu y)y+a_{22}\la_2^2y^2+\nu+o(1)),$$ 
$$x'_y= a_{11}\la_1\la_2x^2\pm2(a_{12}\la_2^2-a_{11}\la_1\la_2\mu)xy$$
\begin{equation}+(a_{11}\la_1\la_2\mu^2-2a_{12}\la_2^2\mu+a_{22}\frac{\la_2^3}{\la_1})y^2+(\pm\mu+\frac{\nu\la_2}{\la_1})+o(1).\label{x'ym}\end{equation}
Here the  $o(1)$ are uniform in $(x,y;s)$ on compact subsets, as $\omega\to0$. Equation (\ref{x'ym}) converges to an equation (\ref{ricric}), if and only if, up to $o(1)$, one has  
$$a_{11}\la_1\la_2=1, \ \ a_{12}\la_2^2-a_{11}\la_1\la_2\mu=0,$$
 \begin{equation}a_{11}\la_1\la_2\mu^2-2a_{12}\la_2^2\mu+a_{22}\frac{\la_2^3}{\la_1}=-1\label{sysa11}\end{equation}
The system of two last equations in (\ref{sysa11}) with $\la_{1,2}\neq0$ is equivalent to 
$$\mu=\frac{a_{12}\la_2}{a_{11}\la_1}, \ \ a_{11}\la_1=\Delta\la_2^3.$$
Therefore, the  solution of (\ref{sysa11}) with $\la_{1,2}>0$ and the $c$ from (\ref{ricric}) are 
$$\la_2=\Delta^{-\frac14}, \ \ \la_1=a_{11}^{-1}\Delta^{\frac14}, \ \ \mu = a_{12}\Delta^{-\frac12}; \ \ c=\pm\mu+\frac{\nu\la_2}{\la_1}=(\nu a_{11}\pm a_{12})\Delta^{-\frac12}.$$
\end{proof}
\begin{remark} In the special case, when family (\ref{jos2g}) is (\ref{jos2}), Proposition \ref{progenr} implies Proposition 
\ref{proresc}.
\end{remark}

\begin{definition} The coordinates $(x,y)$ from Propositions \ref{proresc} or \ref{progenr} centered at a right or left flow point  will be called the {\it Riccati coordinates.}
\end{definition}

\subsection{Riccati equations.  Stable and unstable solutions. Heteroclinic connections}
Equivalent versions of the next two theorems can be found in \cite{callot, sib75}.
\begin{theorem} \label{stunst}  For every $c\in\rr$  Riccati equation (\ref{ricric}) 
has  the {\bf stable} solution $x_-(y)$ and the {\bf unstable} solution $x_+(y)$ with asymptotics 
\begin{equation}x_\pm(y)\simeq y, \ \ \ \text{ as } \ y\to\pm\infty.\label{xpm}\end{equation}
The solutions (\ref{xpm}) are unique, and $x_+(y)=-x_-(-y)$.
\end{theorem}
 \begin{theorem} \label{tother} If $x(y)$ is a solution of  (\ref{ricric}) different from 
 $x_+(y)$, then  $x(y)\simeq-y$, as $y\to+\infty$. If $x(y)\neq x_-(y)$, 
 then $x(y)\simeq-y$, as $y\to-\infty$. 
 \end{theorem}
 
\begin{remark} \label{rem-harm} Riccati equation (\ref{ricric}) is the projectivization of the 
quantum harmonic oscillator eigenfunction equation with eigenvalue $c$: 
\begin{equation} -\psi''+y^2\psi=c\psi.\label{qharm}\end{equation}
Namely, $x(y)$ is a solution of (\ref{ricric}), if and only if 
\begin{equation} x(y)=-\frac{\psi'(y)}{\psi(y)},\label{xpsi}\end{equation}
 where $\psi(y)$ is a solution of 
(\ref{qharm}). 
\end{remark}
Theorems \ref{stunst}, \ref{tother} can be proved  by using the above  remark and the Stokes phenomena theory 
\cite{jlp, bjl1, bjl2, sib, sib75}, which implies existence of solution bases $\psi_{1,\pm}$, $\psi_{2,\pm}$ of 
equation (\ref{qharm}) on $\rr_{\pm}$ such that 
$$\ln\psi_{1,\pm}(y)=\frac{y^2}2(1+o(1)), \ \ \ln\psi_{2,\pm}=-\frac{y^2}2(1+o(1)), \ \ \text{ as } y\to\pm\infty.$$
Solutions $\psi_{2,\pm}$ decrease superexponentially, as $y\to\pm\infty$, and they correspond to solutions 
$x_\pm$ via formula (\ref{xpsi}). Any other solution $\psi(y)$  grows superexponentially and formula (\ref{xpsi}) yields 
$x(y)\simeq -y$. For completeness of presentation we present proofs of Theorems \ref{stunst}, \ref{tother} in Subsection 2.3.

 \begin{definition} We say that a Riccati equation (\ref{ricric}) has a {\it heteroclinic connection,} if  $x_{\pm}(y)$ paste together in one  $\rp^1$-valued solution $x(y)$. 
 \end{definition}
\begin{theorem} \label{homcon} \cite{callot},  \cite[p.89, theorem 22.2]{sib75}  A Riccati equation  (\ref{ricric}) has a heteroclinic connection, 
if and only if $c=2m+1$, $m\in\zz_{\geq0}$. In this case the solution $x_{\pm}(y)$ is a rational function of degree $m+1$ having $m$ finite real poles and one infinite, 
 expressing via the Hermite polynomial $H_m(y)$ as
\begin{equation} x_-(y)=y-\frac{H_m'(y)}{H_m(y)},\ \ \ \ H_m(y)=(-1)^me^{y^2}\frac{d^m}{dy^m}(e^{-y^2}).\label{x-herm}\end{equation}
\end{theorem}
\begin{proof} A solution $x(y)$ has asymptotics $x(y)\simeq y$, as $y\to+\infty$, if and only if the corresponding 
solution $\psi(y)$ of the oscillator equation (\ref{qharm}), see (\ref{xpsi}), decreases superexponentially. Otherwise it grows superexponentially. See Remark \ref{rem-harm} and the discussion after it. Similar statement holds with $y\to-\infty$. 
Therefore, a solution $x(y)=x_{\pm}(y)$ represents 
a heteroclinic connection, if and only if the corresponding solution $\psi(y)$ lies in $L_2(\rr)$. Each solution $\psi(y)$ of (\ref{qharm}) is an eigenfunction of the harmonic oscillator operator with eigenvalue $c$.  An eigenfunction exists in $L_2(\rr)$, if and only if  $c=2m+1$, $m\in\zz_{\geq0}$, and then $\psi(y)=e^{-\frac{y^2}2}H_m(y)$ 
up to constant factor, by classical spectral theory \cite[section 23]{landlif}. Hence, 
$x(y)=y-\frac{H_m'(y)}{H_m(y)}$. 
\end{proof}

\subsection{Dynamics of Riccati equations (\ref{ricric})}
First we prove the following more precise version of Theorem \ref{stunst}.

\begin{lemma} \label{lxyc} 1) For every  $c\in \rr$ and every sequence $(x_n,y_n)$ with $y_n\to-\infty$ and $x_n\leq0$ the sequence of solutions $g_n(y)$ of (\ref{ricric}) with $g_n(y_n)=x_n$ converges to a solution $x_-(y)$ that has asymptotics  $y+O(1)$, as $y\to-\infty$. 

2) For every $c$ a solution $x_-(y)$ with the latter asymptotics is unique. 

3) Every other solution $x(y)$ has asymptotics $x(y)\simeq-y$, as $y\to-\infty$.

4) If for some $y_0$  and a solution $x(y)$ one has $x(y_0)<x_-(y_0)$, then $x(y)$ has a finite real pole $y_*<y_0$. 

5) If for some $y_1\leq-\sqrt{|c|}$  and a solution $x(y)$ one has $x_-(y_1)<x(y_1)\neq\infty$, 
then $x(y)$ has no  pole in $(-\infty,y_1]$.

6) Analogous statements hold for $x_-$, $-\infty$ replaced by $x_+$, $+\infty$ and "$<$", "$>$" in 4), 5) replaced by "$>$" 
and "$<$" respectively and $y_1\geq\sqrt{|c|}$. 

7) Statement 1) holds for uniform convergence: for every finite segment $I\subset\rr_c$ and every sequence 
$(x_n,y_n;c_n)\in\rr^2_{x,y}\times I$ with $y_n\to-\infty$, $x_n\leq0$ and $c_n\to c_*$ the sequence of solutions 
$g_n(y)$ of Riccati equation (\ref{ricric}) with $g_n(y_n)=x_n$ and $c=c_n$ converges uniformly to the 
stable solution $x_-(y)=x_-(y;c_*)$ of (\ref{ricric}) with $c=c_*$. 
\end{lemma}

In the proof of the lemma and in what follows we use the next notations. Let 
\begin{equation} v_{Ric, c}:=\begin{cases} \dot x=x^2-y^2+c\\ \dot y=1.\end{cases}\label{riccases}\end{equation}
be the vector field  given by Riccati equation (\ref{ricric}). It 
extends analytically to the cylinder $\rp^1\times\rr_y$, where 
$\rp^1=\rr_x\cup\{\infty\}$ is a topological circle. Thus extended vector field is  transversal to the infinity line $\{\infty\}\times\rr_y$ and  
\begin{equation} v_{Ric, c} \ \text{ is directed inside the domain } \ \rr_-\times\rr_y \ \text{ at } \{\infty\}\times\rr_y.
\label{atinf}\end{equation}
  
For every $c\in\rr$ set 
$$f_c(x,y):=x^2-y^2+c, \ \ \Gamma_c:=\{ f_c(x,y)=0\}\subset\rr^2_{x,y},$$  
$$\Gamma_{c,l}=\Gamma_c\cap\{ x\leq0\}, \ 
\Gamma_{c,r}=\Gamma_c\cap\{ x\geq0\},$$
$$\Gamma_{c,\pm}=\Gamma_c\cap\{\pm y>0\}, \ \Gamma_{c,l,\pm}=\Gamma_{c,l}\cap\{\pm y>0\}, 
\ \Gamma_{c,r,\pm}=\Gamma_{c,r}\cap\{\pm y>0\},$$
$$\Omega_{c,\pm}=\{ f_c<0\}\cap\{\pm y>0\}.$$
Note that for $c>0$ one has 
$$\partial\Omega_{c,\pm}=\Gamma_{c,\pm}=\Gamma_{c,l\pm}\cup\Gamma_{c,r,\pm},$$
and the common point $(0,\pm\sqrt c)$ of the left and right branches $\Gamma_{c,l,\pm}$, $\Gamma_{c,r,\pm}$ is the vertex of the component $\Gamma_{c,\pm}$ of the hyperbola $\Gamma_c$.

\begin{proof} {\bf of Lemma \ref{lxyc}.} It suffices to prove the statement of the lemma for $x_-$; its statements for $x_+$ then follow by symmetry $(x,y)\mapsto(-x,-y)$, which implies that $x_+(y)=-x_-(-y)$. First we find at least one sequence of initial conditions for which the corresponding solutions converge to a solution with required asymptotics. We prove uniformity of convergence in $c$ lying in a fixed finite segment $I\subset\rr$. To do this, let us introduce the following auxiliary curve and domain.  Set 
$$\Gamma_{c,l,-}':=\Gamma_{c,l,-}-(1,0).$$
For every  $d>0$ big enough depending on $I$ the line $\{ y=-d\}$ intersects each curve $\Gamma_{c,l,-}$, 
$\Gamma_{c,l,-}'$ at one point, which will be denoted $p_{c,d}=(x_{c,d},-d)$ and  $p_{c,d}'=(x_{c,d}-1,-d)$ respectively. 
Let $V_{c,d}\subset\{ y<-d\}$ denote the domain bounded by $\Gamma_{c,l,-}$, $\Gamma_{c,l,-}'$ and the segment 
$[p'_{c,d},p_{c,d}]$. 

\begin{proposition} \label{pogldom} For every $d>0$ big enough depending on $I$ for every 
$c\in I$  
the restriction of the field (\ref{riccases}) to the lateral sides of the domain $V_{c,d}$  is directed inside it. 
\end{proposition}
\begin{proof} On the curve $\Gamma_{c,l,-}$ one has $f_c(x,y)=x^2-y^2+c=0$, the field is equal to $(0,1)$  and thus, is directed inside $V_{c,d}$ for every $d>0$. 
As  $y\to-\infty$, one has $x_{c,y}\simeq y$ uniformly in $c\in I$. On the segment $[x_{c,y}-1,x_{c,y}]$ the function 
$f_c(x,y)$ is positive and  $\frac{\partial f_c(x,y)}{\partial x}=2x\simeq2y$.  Therefore, $f_c(p_{c,y}')\simeq-2y\to+\infty$. But if $d>0$ is big, the slopes of the lateral sides of $V_{c,d}$ are close to one. Hence, the restriction $(f_c(p_{c,y}'),1)$ of  field (\ref{riccases}) to its left lateral side lying in $\Gamma_{c,l,-}'$ is directed  inside the domain 
$V_{c,d}$, whenever $d$ is large enough. The proposition is proved.
\end{proof}

Let $I$ and $d$ be as in the proposition. Take a $c\in I$ and a  sequence $(x_n,y_n)$, $y_n\to-\infty$, 
 lying in $V_{c,d}$. Then the graph of corresponding solution $g_n(y)$ over the interval $(y_n,-d)$ lies in $V_{c,d}$, by Proposition \ref{pogldom}. Every subsequence of the above solutions $g_n(y)$ contains a converging 
 subsequence, and its limit, denoted $x_-(y)$, clearly has asymptotics $y+O(1)$, as $y\to-\infty$.
 
 Let us prove uniqueness of a solution $x_-(y)$ with the above asymptotics, i.e., Statement 2), and Statement 4). Take  another solution $x(y)$. Fix a  
 $y_0<0$ such that $x_-(y)<0$  and $f_c(0,y)=c-y^2<0$ for $y\leq y_0$.  One has $x(y_0)\neq x_-(y_0)$, 
 \begin{equation}\frac{d(x(y)-x_-(y))}{dy}=x^2(y)-x_-^2(y)=(x(y)-x_-(y))(x(y)+x_-(y)).\label{xy2}
 \end{equation}
The module of the latter expression is bigger than $(x(y)-x_-(y))^2$, whenever $x(y)$ and $x_-(y)$ are of the same sign. 
Therefore, if $x(y_0)<x_-(y_0)$, then  
$x(y)$ goes to infinity in finite negative time $y_*<y_0$. This already proves Statement 4). Analogously, if $x(y_0)>x_-(y_0)$, we get that the value $x(y)$ becomes positive in finite negative time $y_*<y_0$. Thus, one of the above happens for arbitrarily large $y_*$, since $y_0$ is arbitrary. Hence, $x(y)$ is not asymptotically equivalent to $y$. Statement 2) is proved. 
 
 Let us prove Statement 5). The field $v_{Ric,c}$ is directed to the left half-plane at the points of the $y$-axis with $y< y_1$, since $f_c(0,y)<0$ for $y<-\sqrt{|c|}$.  Let $x(y_1)>x_-(y_1)$. Then the latter inequality remains valid for $y<y_1$ close to $y_1$ and eventually one will have $x_-(y)<0<x(y)$ for $y$ less than some $y_2<y_1$, by the argument after  (\ref{xy2}) and the above field direction statement. Hence,  if $x(y)\to\infty$, as $y$ tends to some $y_*<y_2$, it tends to $+\infty$, i.e., to $\infty$ from $\rr_+$, -- a contradiction to (\ref{atinf}). 
 
  Let us prove Statement 3). To this end, we fix a big $d>0$ and consider the auxiliary domain $W_{c,d}\subset\{ y<-d\}$  bounded by 
  the translation images $\Gamma_{c,r,-}\pm(1,0)$ of the right low branch of the hyperbola $\Gamma_c$ and the segment 
  connecting their intersection points with the line $\{ y=-d\}$. 
  On its lateral sides one has $f_c(x,y)\simeq\mp 2y$, and the vector field $-v_{Ric,c}$ is directed inside $W_{c,d}$, as in 
  the proof of Proposition \ref{pogldom}. Analogously, as $y\to-\infty$, the restriction of the function $f_c(x,y)$ to the semiaxis  $[0,+\infty)\times\{-y\}$  is negative (positive) on the left (respectively, 
  right) from $W_{c,d}$, and its module is no less than $|y|(1+o(1))$ there. 
  
  Take now a solution $x(y)$ different from $x_-(y)$. To prove the asymptotics $x(y)\simeq -y$, it suffices to show 
  that $(x(y_1),y_1)\in W_{c,d}$ for some $y_1<-d$; then the latter inclusion remains valid for all $y<y_1$, and hence, $x(y)\simeq -y$, as $y\to-\infty$. 
  Suppose the contrary: the orbit $(x(y),y)|_{y<- d}$ is disjoint from $W_{c,d}$. 
  Fix a big $y_0<-d-2$. Without loss of generality we consider that 
  $x(y_0)>0>x_-(y_0)$. Indeed, if $x(y_0)<x_-(y_0)$, then $x(y)$ has a pole $y_*<y_0$, and for $y<y_*$ close to $y_ *$ the value $x(y)$ is big and positive, by (\ref{atinf}). 
  If $x(y_0)>x_-(y_0)$, then $x(y)$ remains finite on $(-\infty,y_0]$ and is positive for 
  every negative $y$ large enough, as was shown above. Then we can and will consider that the latter holds for every 
   $y<y_0$. One has 
  $(x(y)+y)'=f_c(x,y)+1$. For $y\leq -y_0$ the latter expression is negative on the left from $W_{c,d}$ and positive on 
  its right and has module bigger than $|y|(1+o(1))$ there, as $y\to-\infty$, by the argument with derivatives from the proof of Proposition \ref{pogldom}. Thus, if $(x(y),y)$ lies  on the left from  $W_{c,d}$ for all $y\leq y_0$,  then $x(y)+y$ grows at least as $\frac{y^2}2(1+o(1))$, as $y\to-\infty$, by the latter lower bound. But then the orbit $(x(y),y)$ crosses $W_{c,d}$, since its starts on its left and then arrives to its right, 
  -- a contradiction. The case, when it starts on the right is treated analogously. Statement 3) is proved. 
 
 Let us prove Statement 1). Take a sequence $(x_n,y_n)$ with $y_n\to-\infty$ and $x_n\leq0$. 
 The solutions $g_n(y)$ of (\ref{ricric}) with $g_n(y_n)=x_n$ are non-positive on the interval $[y_n,-\sqrt{|c|})$, since $v_{Ric,c}$ is directed to the left half-plane on $\{0\}\times(-\infty,-\sqrt{|c|})$. 
 Passing to a subsequence we can achieve that they 
 converge to a solution $x(y)$ of (\ref{ricric}), by compactness. The limit solution is also non-positive on $(-\infty,-\sqrt{|c|})$, 
 and thus,  
 cannot have asymptotics $-y$, as $y\to-\infty$. Hence, $x(y)=x_-(y)$, by Statement 3). Statement 1) is proved. 
  Statement 7) follows by compactness. Lemma \ref{lxyc} is proved.
 \end{proof}

\subsection{Behavior of the stable solution $x_-(y)$  in absence of heteroclinic connections} 
\begin{proposition} \label{cal1} \cite{callot} If $c\in Z_{m}$, $m\in\zz_{\geq0}$,  
then the restriction of the stable solution $x_-(y)$ of (\ref{ricric}) to the real line has $m$ distinct poles with residue -1, and $x_-(y)\simeq -y$, as $y\to+\infty$. 
\end{proposition}
We prove the next slightly stronger proposition. 
\begin{proposition} \label{cal2} For every $m\in\zz_{\geq0}$,  $0<\var<1$ there exists a $a_0=a_0(m,\var)>2$, such that for every $a\geq a_0$ and $c\in Z_{m,\var}$  the intersection with the rectangle $Q_a=\{|x|<2a, \ |y|<a\}$ of graph of the solution $x_-(y)$  consists of $m+1$ connected components $X_j=X_{j,a}$, $j=0,\dots,m$, satisfying the following statements, see Fig. \ref{fig:xj}.

(i) The components   $X_{j}$ are  graphs of the solution $x_-(y)$ over disjoint intervals $U_j=(\alpha_j,\beta_j)$, $j=0,\dots,m$, $\beta_{j-1}<\alpha_j$. 

(ii) Between any two neighbor intervals $U_j$, $U_{j+1}$, i.e., in $(\beta_j,\alpha_{j+1})$ there is exactly one real pole of the function $x_-(y)$. 

(iii) For every $j\leq m-1$ the restriction of the function $x_-$ to each $U_j$  is strictly increasing, and $x_-(\beta_j)=2a$, 
i.e., $X_{j}$  ends at the right lateral side of the rectangle $Q_a$; $x_-(\alpha_j)=-2a$ for $j=1,\dots,m$.  

(iv) $\alpha_0=-a$, $\beta_{m}=a$, i.e., the starting point of $X_0$ and endpoint of $X_m$ lies on the 
lower, respectively upper side of the rectangle $Q_a$. 

(v) $x_-(\alpha_0), x_-(\beta_m)\in(-\frac{11}{10}a, -\frac9{10}a)$.  
\end{proposition}
 \begin{figure}
 \begin{center}
\epsfig{file=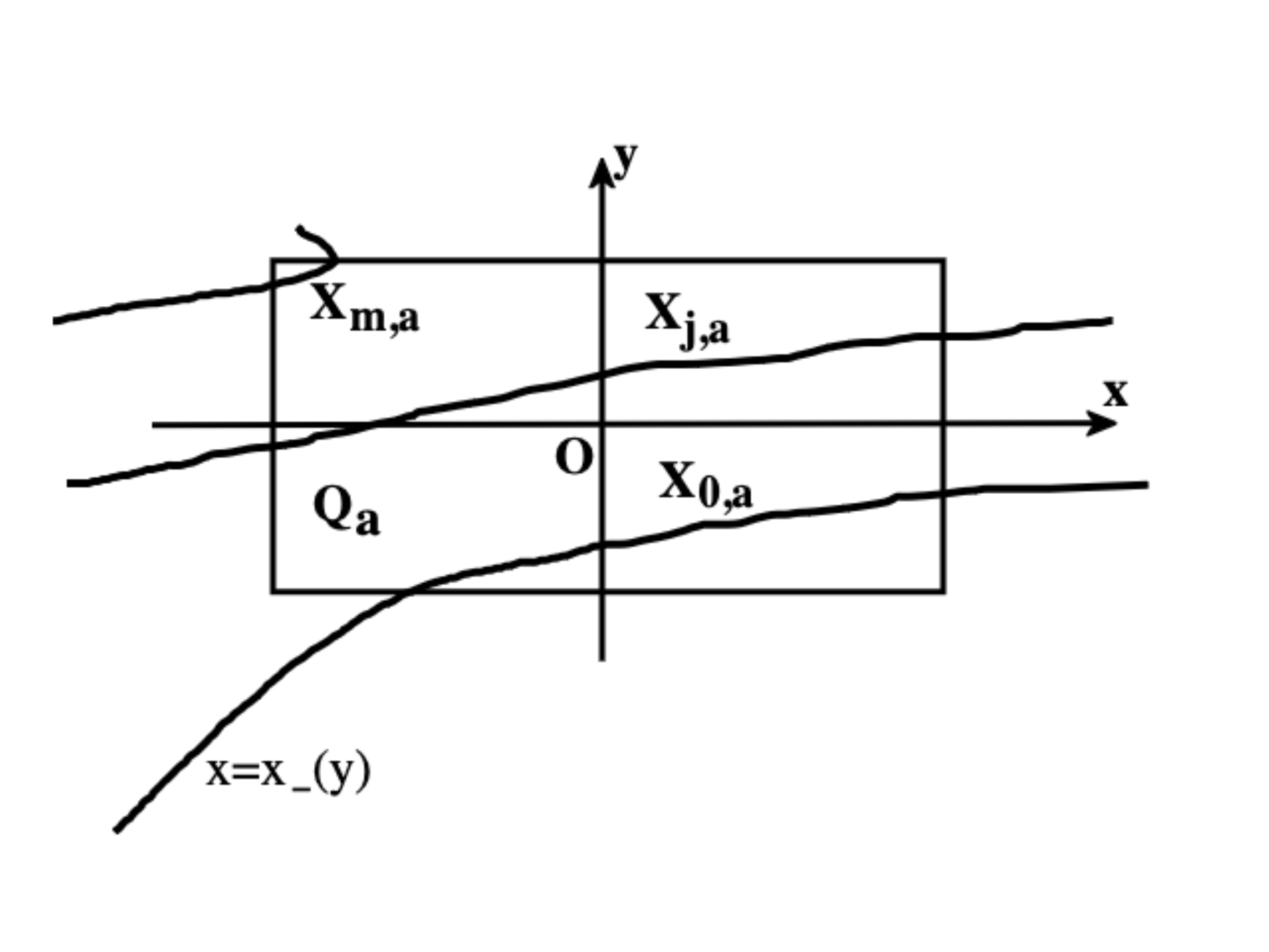, width=17em}
\caption{Graph of stable solution $x_-(y)$ of (\ref{ricric}) with $c\notin2\zz_{\geq0}+1$ and its crossing intervals $X_{j,a}$ with $Q_a$.}
\label{fig:xj}
  \end{center}
\end{figure} 
In the proof of Propositions \ref{cal1} and \ref{cal2} we use the next  proposition.

\begin{proposition} \label{pmon} Let $c_1<c_2$. Consider the backward orbits of the field $v_{Ric,c}=(f_c(x,y),1)$ giving Riccati equation (\ref{ricric}) for $c=c_1,c_2$. 

1) For every $z\in\rr^2_{x,y}$ the small arc of the  orbit starting at $z$ corresponding to $c=c_1$ lies on the right from 
that corresponding to $c=c_2$. 

2) Let $x_{1,\pm}(y)$, $x_{2,\pm}(y)$  be the stable and unstable solutions of Riccati equation (\ref{ricric}) corresponding to the values $c=c_1, c_2$. Then on every interval $(-\infty,-R)$, 
respectively $(R,+\infty)$ where both solutions $x_{1,-}$, $x_{2,-}$, respectively $x_{1,+}$, $x_{2,+}$ have no poles one has 
\begin{equation} x_{1,-}(y)<x_{2,-}(y), \ \ \ \text{ resp. } \ \ \ x_{1,+}(y)>x_{2,+}(y).\label{inx}\end{equation}
\end{proposition}
\begin{proof} Statement 1) is obvious by monotonicity in $c$ of the function $f_c$. Let us prove Statement 2). 
It suffices to prove it for the sign $-$. 
Suppose the contrary: there exists a $y_0\in(-\infty,-R)$ such that $x_{1,-}(y_0)\geq x_{2,-}(y_0)$. 
Then below the value $y=y_0$ the graph of the function $x_{1,-}$ lies on the right from the graph of $x_{2,-}(y)$, since its 
initial condition $(x_{1,-}(y_0),y_0)$ lies on the right and 
by Statement 1). Fix a $y_1<y_0$ close to $y_0$: one has $x_{1,-}(y_1)>x_{2,-}(y_1)$. The backward orbit of the point 
$z=(x_{1,-}(y_1),y_1)$ under the field $v_{Ric, c_2}$ lies on the right from the graph of the function $x_{2,-}(y)$, since 
the initial condition already lies on the right.  It is a graph of a solution $x(y;c_2)$ of (\ref{ricric}) with $c=c_2$. Therefore, the function $x(y;c_2)$  is asymptotic to $-y$, as $y\to-\infty$, and hence, is positive on some interval $(-\infty,R_1)$, 
 by Lemma \ref{lxyc}, Statement 3). But the backward orbit of the point $z$ under the field $v_{Ric,c_1}$
lies on the right from the graph of positive function $x(y,c_2)$, by Statement 1). It coincides with the graph of the function $x_{1,-}(y)\simeq y$, which is thus negative for big negative $y$. The contradiction thus obtained proves Statement 2). 
\end{proof}

\begin{proof} {\bf of Propositions \ref{cal1} and \ref{cal2}.}  The stable and unstable solutions $x_\pm(y)$ depend continuously on the parameter. Each pole of every solution $x(y)$ is simple with residue -1, since 
$x(y)=-\frac{\psi'(y)}{\psi(y)}$ where $\psi(y)$ is a solution of  linear second order differential equation (\ref{qharm}) and hence, 
may have only simple zeros. Fix a $m\in\zz_{\geq0}$. For every $\var\in(0,1)$ as $c$ varies in $Z_{m,\var}$, one has $x_-(y)\simeq -y$, as $y\to+\infty$, since there are no heteroclinic 
connections and by Statement 3) of Lemma \ref{lxyc}. This together with Implicit Function Theorem implies that  poles of the solution $x_-(y)=x_-(y;c)$ depend continuously on $c\in Z_{m,0}$ and their number is constant. It remains to find their number. 

For $c=2m+1$ the solution $x_-(y)$ has $m$ finite poles, which are zeros of the Hermite polynomial $H_m(y)$ and are real, see (\ref{x-herm}). Fix a big number $R>0$ such that the restriction to $(-\infty,-R]$ of the solution $x_-(y;2m+1)$ 
has no poles; then the same holds for the restriction to $[R,+\infty)$ of the function $x_+(y;2m+1)$. In addition we assume that $x_-(R;2m+1)$ is finite. As $c$ becomes slightly smaller than $2m+1$, the graph of the solution $x_-(y;c)$ over $(-\infty,-R]$ lies 
on the left from that of $x_-(y;2m+1)$, by Proposition \ref{pmon}, Statement 2). The same inequality remains valid on 
all of $\rr_y$ with "right" with respect to the $x$-coordinate on $\rr^2$ being replaced by "right" on the cylinder $\rp^1\times\rr_y=(\rr\slash\pi\zz)\times\rr_y$ with respect to the circular angular coordinate $2\arctan x(\mo2\pi)$. This follows from extended Statement 1) of Proposition \ref{pmon}  for arbitrary infinite point  
$z=(\infty,y)$ and local coordinates $(-\frac1x,y)$ on its neighborhood, where "on the right" means with 
respect to the coordinate $-\frac1x$. Finally, the graph of the function $x_-(y,c)$ remains on the left from that of $x_-(y,2m+1)$. Hence,  
\begin{equation}x_-(R;c)<x_-(R,2m+1)=x_+(R,2m+1)<x_+(R,c),\label{inec+}\end{equation}
by Proposition \ref{pmon}, Statement 2).  Therefore, the solution $x(y;2m+1)$ with initial condition $x(R;2m+1)=x_-(R,c)$ 
is less than $x_+(y;2m+1)$ for $y\geq R$.  But $x_-(y,c)<x(y;2m+1)$ for $y>R$, since $c<2m+1$, as in Proposition \ref{pmon}, Statement 1). 
Thus, $x_-(y;c)<x_+(y;2m+1)$ on $(R,+\infty)$, and hence, $x_-(y;c)$ has no poles  $y\geq R$, by  (\ref{atinf}). It is different from $x_+(y;c)$ by (\ref{inec+})  and hence, 
is asymptotic to $-y$. Finally, the only poles of the solution $x_-(y;c)$ are small perturbations of those of $x_-(y;2m+1)$. 
Hence, their number is equal to $m$. Proposition \ref{cal1} is proved. 

Proposition \ref{cal2} then follows by 
compactness of the subsets $Z_{m,\var}$. In more detail,  fix an $a$ bigger than the maximal module of a pole of 
the solution $x_-(y;c)$ with $c\in Z_{m,\var}$. Moreover, we can achieve that Statements (iv) and (v) of Proposition \ref{cal2} hold for all $c\in Z_{m,\var}$, again by compactness and uniform asymptotics $x_-(y)=\mp y(1+o(1))$, as  
$y\to\pm\infty$ for $c\in Z_{m,\var}$. Then the other  Statements (i)--(iii) hold automatically, except for monotonicity statement in (iii). Let us prove a stronger statement: 
the restriction of a solution $x(y)$  of (\ref{ricric}) to an arc between its two neighbor finite poles $y_1<y_2$
is strictly increasing. Indeed, this arc starts at a point $(\infty,y_1)$ and goes in the $x$-coordinate from $-\infty$ to $+\infty$. 
 If, to the contrary, the derivative $x'(y)$ were non-positive at some $y\in(y_1,y_2)$, then the above arc would enter the domain 
 $\{ f_c(x,y)<0\}$. Then it enters its upper part $\Omega_{c,+}$, since the vector field $v_{Ric,c}$ is directed vertically up on 
 $\partial\Omega_{c,-}$, hence, outside $\Omega_{c,-}$. But the field $v_{Ric,c}$ is directed inside $\Omega_{c,+}$ on its lateral sides for the same reason. Hence, the arc in question will remain in $\Omega_{c,+}$ in the future and will not go to infinity in finite time, -- a contradiction. Proposition \ref{cal2} is proved.
 \end{proof}

\subsection{Stable slow flowboxes in the Riccati chart. Convergence to the stable solution}
Here we deal  with a slow-fast system (\ref{jos2g}) on a neighborhood of the point $O=(0,0)\in\rr^2_{\theta,\tau}$ with  right-hand side $f_{s;\omega}(\theta,\tau)$ having a horizontally positive germ at $O$, see Definition \ref{horpos}. We pass to the following {\it compact restriction:}  
\begin{equation} s \text{ lies in a compact subset  of the ambient parameter space}.
\label{compc}\end{equation}
In what follows, whenever the contrary is not specified, by $U$ we will denote the latter compact subset. In the case of the RSJ model, see Theorem \ref{thm1}, in the coordinates $(c_+,c_-)$ it will be 
some of the subsets $Z_{m,\var}\times Z_{k,\var}\subset\rr^2_{c_+,c_-}$. 

Without loss of generality we consider that the Hessian form at $O$ of the function $f_{s}=f_{s;0}(\theta,\tau)$ 
is equal to $\alpha(d\theta^2-d\tau^2)$, $\alpha=\alpha(s)>0$. 
One can achieve this by $s$-depending linear change of variables of type 
$(\theta,\tau)\mapsto(\wt\theta,\wt\tau)$, 
$(\theta,\tau)=(\wt\theta-\mu\wt\tau,\la\wt\tau)$. Here $\mu$ is the same, as in (\ref{varch}), and $\la=\frac{\la_2}{\la_1}$. 
In what follows the latter new variables $(\wt\theta, \wt\tau)$ will be denoted by $(\theta,\tau)$. Moreover, we can and will consider that $\alpha=\frac12$, as in the RSJ model (\ref{jos2}), applying homothety. Then in the Riccati coordinates $(x,y)=(\frac\theta{\sqrt{2\omega}},\frac\tau{\sqrt{2\omega}})$ the line field defined by (\ref{jos2g}) converges to the line field given by the Riccati equation (\ref{ricric}) with $c=c_+(s)$.

Fix a rectangle $W$ centered at $O$, 
$$W=\{|\theta|<2M, \ |\tau|<M\}\subset\rr^2_{\theta,\tau}, \ \ \ \ 0<M<\frac\pi3,$$ 
 on whose closure all the $f_{s;\omega}$ are well-defined. 
 We choose it so that 
 
 - for $\omega=0$ for every $s\in U$  the intersection with $\overline W$ of the zero locus 
 $C_{s;0}=\{ f_{s;0}=0\}$ consists of two curves  $C_{\mcl}$, $C_\mcr$ intersecting at $O$ that are graphs of $C^1$-smooth functions $\tau=\tau(\theta)$ such that each $C_{\mcl,\mcr}$ starts in the interior of  the lower side of $W$ and ends in the 
 interior of  its upper side;
 
 - the above curves intersect the $\tau$-axis only at $O$ and transversally. 
 
 - the gradient map $(\theta,\tau)\mapsto\nabla f_{s;0}(\theta,\tau)$ is a diffeomorphism on $\overline W$, and in particular, 
 nonzero outside $O$. 
 
 Existence of $W$ as above follows by  compactness and Hessian form type. 
 Set
$$C_{\mcl,\pm}=C_{\mcl}\cap\{\pm\tau>0\}, \ \ C_{\mcr,\pm}=C_{\mcr}\cap\{\pm\tau>0\}.$$
Recall that we name the curves $C_{\mcl}$, $C_{\mcr}$ so that $C_{\mcl,+}$ lies on the left from $C_{\mcr,+}$. See Fig. 9 with $z_0=(\theta_0,\tau_0)=(0,0)$. 
 Let $\Xi_-$ denote the open curvilinear triangle 
   bounded by $C_{\mcr,-}$, $C_{\mcl,-}$ and the horizontal segment connecting their lower endpoints. For every $s\in U$ fix 
    a point $b_0=b_0(s)\in\Xi_-$ and a point $a_0\in W\cap\{\tau=\tau(b_0)\}$ separated from the point $b_0$ by $C_{\mcr,-}$.   
    Set 
    $$I_0=[a_0,b_0], \ \ \wh\tau_0:=\tau(a_0)=\tau(b_0)<0.$$

Everywhere below whenever we work with Riccati coordinates $(x,y)$ each subset in $\rr^2_{x,y}$ is identified with 
the corresponding subset in $\rr^2_{\theta,\tau}$ via the coordinate change $(\theta,\tau)\mapsto(x,y)$. 
The upper and lower horizontal sides (bases) of the rectangles $Q_R$, $W$ are denoted by 
$$Q_R^{up}, \ W^{up}= \text{ the upper base; } \ \ Q_R^{low}, \ W^{low} = 
\text{ the lower base.}$$

The main result of this subsection is the next lemma. 
\begin{lemma} \label{lconv1}{\bf (First Main Lemma).} Let (\ref{jos2g}) be a horizontally positive germ of slow-fast family on a neighborhood of the point $O$, see Definition \ref{horpos}, satisfying compactness condition (\ref{compc}). Let   $(x,y)$ be the  corresponding Riccati coordinates. 
Let $W$ and $I_0$ be as above. 
For every $R>0$ big enough for every $\omega>0$ small enough depending on $R$ the following statements hold.

 1) The orbit  $\mco(I_0)$  crosses the line 
 $\{ y=-R\}$ by a segment contained in the interior of the lower base $Q_R^{low}$, on the left from the $y$-axis. 
 
 2) It crosses the rectangle $Q_R$ by orbits of (\ref{jos2g}) in $Q_R$ 
 that all converge to the arc $X_{0,R}$, see Proposition \ref{cal2}, of the graph of the stable solution $x_-(y)$ of Riccati equation (\ref{ricric}) with $c=c_+(s)$ in the Riccati chart $(x,y)$, as $\omega\to0$. 
 The above orbits converge as graphs of functions uniformly in the parameter $s$ and in the transversal parameter of orbit.  
 \end{lemma}
\begin{proof} The classical slow-fast theory yields that as $\omega\to0$, the  orbit  $\mco(I_0)$ first becomes an exponentially narrow flowbox that drifts along a curve in the zero locus 
$$C_{s;\omega}=\{ f_{s;\omega}(\theta,\tau)=0\},$$ 
see Proposition \ref{pflboxes} below. Lemma \ref{lconv1} describes how the above flowbox extends further on, as the drift curve  approaches $O$. In the proof of the lemma we use an a priori bound  given by the next technical proposition. 

\begin{proposition} \label{proapr} There exists a $\sigma_1>0$ independent on $(s,\omega)$  
 such that for every $R>0$ big enough and $\omega$ small enough depending on $R$ the following statements hold.
 
1) The complement $(C_{s;\omega}\cap W)\setminus Q_R$ consists of four arcs that are graphs of $C^1$- functions 
 $\theta=\theta(\tau)$. Two arcs start in the interior of the base  $W^{low}$ and end in the interior of  
 $Q_R^{low}$; they are separated by the $\tau$-axis. Two other arcs have the same properties with "low" changed by "up". 
 See Fig. \ref{fig:piw}.  
 
 2)  The derivatives of the above functions $\theta(\tau)$ have modules less than $\sigma_1^{-1}$, i.e., the components 
 in question  have slopes greater than  $\sigma_1$.  
 
 3) For every $\eta>0$ for every $R$ big enough depending on $\eta$ and $\omega>0$ small enough depending on 
 $\eta$ and $R$   the coordinates of the above endpoints on the sides of $Q_R$ have $x$-coordinates $\eta R$-close to 
 $\pm R$. 
 \end{proposition}

\begin{proof} Statements 1), 2) of the proposition are obvious, if we replace $Q_R$ by any fixed rectangle $\wt W$ homothetic  to  $W$ and small enough, independent on $\omega$ and $R$, since they hold for $C_{s;0}$ and by convergence $C_{s;\omega}\to C_{s;0}$.   Let us choose appropriate auxiliary rectangle $\wt W$ and prove uniform lower bound of slope for the intersection  $C_{s;\omega}\cap(\wt W \setminus Q_R)$ for $R$ large enough and $\omega$ small enough depending on $R$. This will imply similar uniform bound in $W \setminus Q_R$. 

The critical point $O_{s;\omega}$ of $f_{s;\omega}$, which is $O_{s;0}=(0,0)$ for $\omega=0$, 
depends continuously on the parameters and has continuous derivative in $\omega$, by the Implicit Function Theorem and since the gradient map is a local diffeomorphism on a fixed neighborhood of $O$. 
Therefore, $O_{s;\omega}=O(\omega)$, as $\omega\to0$. 
For every $\omega$ small enough and $s\in U$ this is a unique critical point  in $W$, and  it lies in $Q_1$, by the above asymptotics. 

Recall that by assumption, the Hessian matrix of the function $f_{s;0}(\theta,\tau)$ at $O$ has form $A= \diag(1,-1)$.  
Take a big $R>1$ and a $z\in W\setminus \opint(Q_R)$ at which $f_{s;\omega}(z)=0$. Then 
\begin{equation}0=f_{s;\omega}(z)=f_{s;\omega}(O)+\nabla f_{s;\omega}(O)z+\frac12<Az,z> +o(||z||^2), \text{ as } z\to0,\label{0hess}\end{equation}
 $f_{s;\omega}(O)=O(\omega)$, since $f$ has a continuous derivative  in $\omega$. One also has  
 \begin{equation}\nabla f_{s;\omega}(O)=O(||O_{s;\omega}||)=O(\omega),\label{nablaf}\end{equation}
 by the above discussion. Let $v$ denote the unit vector directing the radial line through $z$. On the other hand,  $||z||\geq R\sqrt{2\omega}$, hence the middle quadratic term in the right-hand side in (\ref{0hess}) is no less than $R^2\omega<Av,v>$. 
 
 Fix a small rectangle $\wt W$ homothetic to $W$ on which the $o(||z||^2)$ in (\ref{0hess}) is much smaller than $||z||^2$. If $R$ is big, $\omega$ is small enough,  and $z\in \wt W\setminus Q_R$,  the middle term  in (\ref{0hess}) becomes dominating, hence the right-hand side in (\ref{0hess}) cannot vanish, unless $v$ is close to 
 a unit vector $w$ that is isotropic. This means that  $<Aw,w>=0$: $w=(w_1,w_2)$, $|w_1|=|w_2|=\frac1{\sqrt2}$. Hence, the vector $Az$ has angle bounded away from zero with the vertical axis.  One has 
  \begin{equation}\nabla f_{s;\omega}(z)=\nabla f_{s;\omega}(O)+Az+o(z), \ \ \text{ as } \ z\to0,\label{asn}\end{equation} 
by (\ref{0hess}),  and the latter asymptotics is uniform in $s$ and in small $\omega$, by compactness. We consider that the $o(z)$ in 
  (\ref{asn}) is much smaller than $||z||$ on $\wt W$ shrinking $\wt W$. Then for $z\in\wt W\setminus Q_R$  and small $\omega$ the term $Az$ becomes dominating in (\ref{asn}), by (\ref{nablaf}).  
  Therefore, the  gradient $\nabla f_{s;\omega}(z)$ has angle bounded from below with the ordinate axis at 
 all the points in the intersection $(\wt W\cap C_{s;\omega})\setminus Q_R$, as does $Az$. 
 Thus, the tangent lines to the latter intersection 
 have angles bounded from below with the abscissa axis. Uniform lower bounds of slope in $\wt W\setminus Q_R$, and hence, in $W\setminus Q_R$ are proved.
  
The above domination argument together with obvious uniform lower boundedness of the norm 
 $||\nabla f_{s;\omega}(z)||$ on $W\setminus\wt W$ also imply the inequality 
 \begin{equation} ||\nabla f_{s;\omega}(z)||\geq\sigma_0||z|| \ \text{ for every } \ z\in (C_{s;\omega}\cap W)\setminus Q_R\label{nablafs}\end{equation}
 with some  $\sigma_0>0$ independent on $R$ and $\omega$. It will be used further on.
 
 Let us now prove Statements 1) and 2). 
 In the Riccati coordinates the intersection $C_{s;\omega}\cap Q_R$  converges to the hyperbola  
$$\Gamma_c:=\{ x^2-y^2+c=0\}.$$
Take $R$ big enough so that  $\Gamma_c$ intersects the interior of each horizontal side of $Q_R$ at two distinct points and does not intersect its lateral sides. Then on each horizontal side the intersection points are separated by the $y$-axis. This implies that for every $\omega$ small enough the intersection 
$C_{s;\omega}\cap(W\setminus Q_R)$ is a one-dimensional manifold with eight endpoints: two on each horizontal side 
of each rectangle $W$, $Q_R$. This together with the lower bound of slope proved above imply that it consists of four components: 
two components going from the upper side of $Q_R$ to that of $W$; two components going from the lower side of 
$W$ to that of $Q_R$. Each pair of upper, respectively lower components is separated by the $\tau$-axis, as are their endpoints 
on horizontal sides. This together with the above slope bound imply Statements 1), 2). Statement 3) is obvious with $C_{s;\omega}$ replaced by 
the hyperbola $\Gamma_c$, and it remains valid for $C_{s;\omega}$ as well, if $\omega$ is small enough. 
The proposition is proved.
\end{proof}
 
 In what follows by $\gamma=\gamma_{s;\omega}$ we denote the left lower component of the submanifold 
 $C_{s;\omega}\cap(W\setminus Q_R)$, see Statement 1) of the above proposition. For every $\delta\in(0,1)$ set 
 \begin{equation}\gamma^{\pm}_\delta=\gamma_{\delta;s;\omega}^{\pm}:=\gamma\pm(\delta\sqrt{2\omega},0).
 \label{gdelta}\end{equation}
 The  curves (\ref{gdelta}) start at  $W^{low}$ and end inside the  side $Q_R^{low}$, as $R$ is large enough and $\omega$ is small enough depending on $R$, by Proposition \ref{proapr}. Set 
 $$\Pi_{\delta}=\Pi_\delta(s;\omega):=\text{ the curvilinear trapezoid bounded by  } \gamma^{\pm}_\delta, \ 
 W^{low},\  Q_R^{low}.$$
\begin{proposition} \label{cl1} For every $\delta\in(0,1)$ for every $R>0$ big enough depending on $\delta$ for every $\omega>0$ small enough depending on 
$(R,\delta)$ the restriction of vector field (\ref{jos2g}) to the curves (\ref{gdelta}) is directed inside the domain $\Pi_{\delta}$. See Fig. \ref{fig:piw}
\end{proposition}
\begin{figure}
 \begin{center}
\epsfig{file=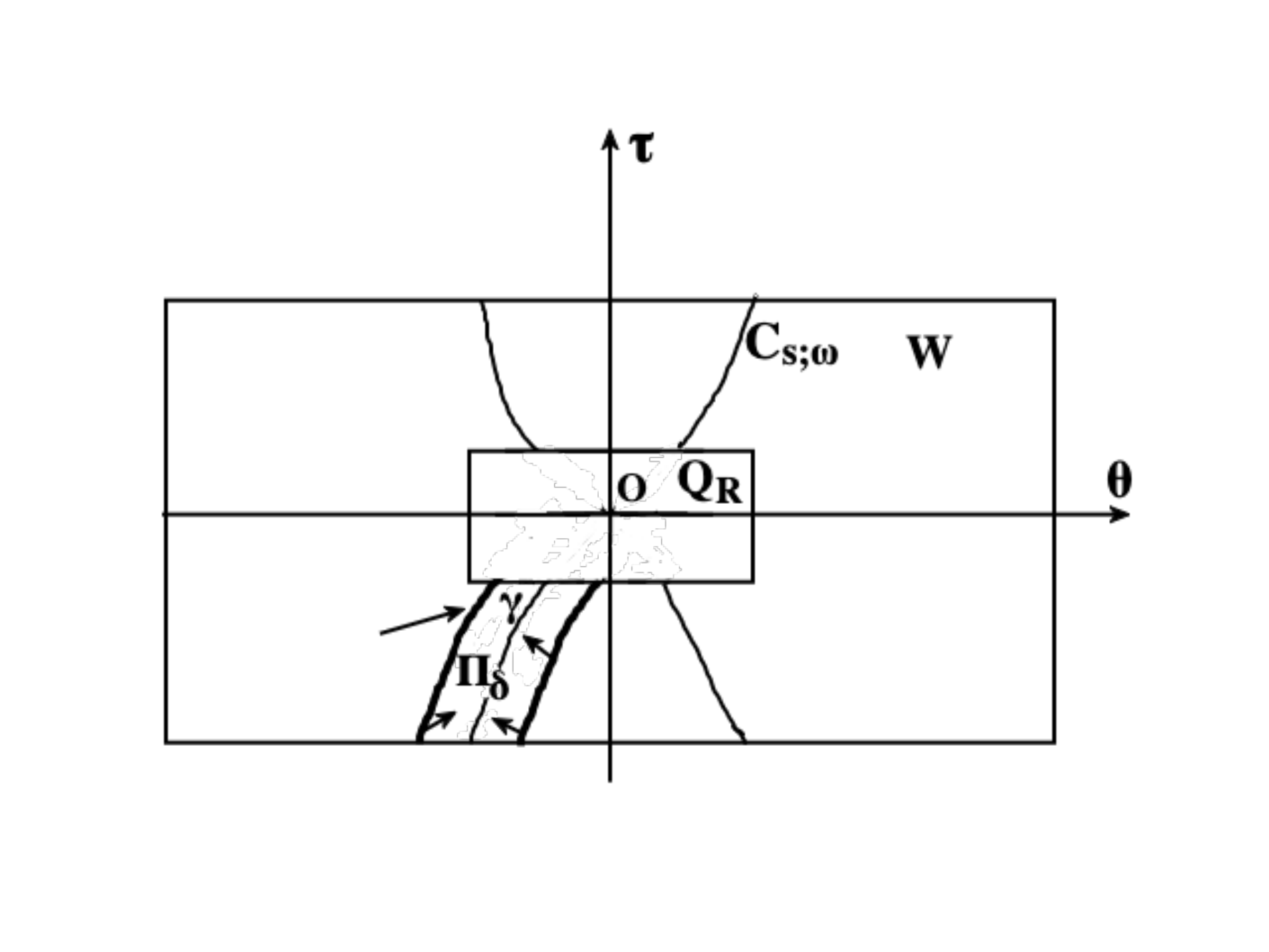, width=25em}
\vspace{-1.3cm}
\caption{The rectangles $W$, $Q_R$, the curve $\gamma$ and the trapezoid $\Pi_\delta$.}
\label{fig:piw}
  \end{center}
\end{figure} 
\begin{proof} 
We prove that the restriction to $\gamma^{-}_\delta$ of the field is directed to the right from the curve 
$\gamma^{-}_{\delta}$. The similar statement that on the curve $\gamma^{+}_\delta$ it is directed to the left is proved 
analogously. The field (\ref{jos2g}) is $(f_{s;\omega}(\theta,\tau),\omega)$. Its restriction to $\gamma$ is vertical and equal to $(0,\omega)$. For every $z\in\gamma$ one has 
\begin{equation}f_{s;\omega}(z)=0, \ f_{s;\omega}(z-(\delta\sqrt{2\omega},0))\geq-\delta\sqrt{2\omega}
\max\frac{\partial f_{s;\omega}(\theta,\tau)}{\partial\theta},\label{fsder}\end{equation}
where the maximum is taken over the segment connecting $z$ and $z-(\delta\sqrt{2\omega},0)$. We claim that 
\begin{equation}\max\frac{\partial f_{s;\omega}(\theta,\tau)}{\partial\theta}\leq-\sigma_0 R\sqrt{2\omega},\label{fs2}\end{equation}
where $\sigma_0$ is a constant independent on $R$ and $\omega$. Indeed, 
the latter derivative is the projection of the gradient to the horizontal axis. It is  negative 
at $z$, since the function $f_{s;\omega}$ is negative between the curves $C_{\mcl}$, $C_{\mcr}$, by the Hessian form assumption at $O$. 
We claim that along the segment in question the latter projection of the gradient is negative  and its module  
is no less than the norm of the gradient times a universal constant. Indeed,   
on the curve $\gamma$ the angle of the gradient with the vertical axis  is bounded from below by a universal constant,  by Proposition \ref{proapr}, Statement 2). Its norm at $z\in\gamma$ is no less than $||z||\geq R\sqrt{2\omega}$ times a universal constant independent on $(R,\omega)$, see (\ref{nablafs}). As $z$ is changed by a vector with norm no greater than $\delta\sqrt{2\omega}$, 
the gradient is changed by a vector with norm no greater than  $\delta\sqrt{2\omega}$ times a universal constant, since the derivatives of the gradient are uniformly bounded on $W$. Therefore, if $R$ is big enough and $\omega$ is small enough depending on 
$R$, then along the above segment the norm of the gradient is bounded from below by $R\sqrt{2\omega}$ times a universal constant and its angle with the vertical axis remains bounded from below by some universal constant. 
Thus, its projection remains negative and its module is bounded from below by $\wt\sigma_0 R\sqrt{2\omega}$ with a universal constant $\wt\sigma_0$. This proves (\ref{fs2}). Substituting it to (\ref{fsder}) we get that 
\begin{equation} f_{s;\omega}(z)\geq2\wt\sigma_0\delta R\omega.\label{fs3}\end{equation}

The slope of the curve $\gamma^-_\delta$ at $z_\delta$ is equal to that of the curve $\gamma$ at $z$, and hence, 
is bounded from below by a universal constant $\sigma_1$.  
 Choosing $R$  big enough 
so that $2R\delta\wt\sigma_0>\sigma_1^{-1}$ yields that the field is directed to the right from the curve $\gamma^-_\delta$. 
This proves Proposition \ref{cl1}. 
\end{proof} 

Fix  arbitrary $\tau_1$, $\tau_2$, $\wh\tau_0<\tau_1<\tau_2<0$, and set 
    \begin{equation}\mcs:=\{ \tau_1\leq\tau\leq\tau_2\}.\label{stripi}\end{equation}

\begin{proposition} \label{pflboxes} The orbit $\mco(I_0)$   intersects $\mcs$ by a narrow flowbox denoted by $F_-(I_0)$ and called {\bf stable flowbox} that is $O(\omega)$-close to $\gamma$. It is fibered by horizontal segments of widths no greater than $\exp(-\frac d{\omega})$, with $d>0$ depending only on $\tau_1$ and $\tau_2$. See Fig. \ref{fig:slow} for family (\ref{jos2}) with $c_{\pm}>0$, the base point $(\pi,0)$ and $I_0=[0,\pi]\times\{\wh\tau_0\}$, $\wh\tau_0=-\frac23\pi$, $\tau_1=-\frac{\pi}2$, $\tau_2=-\frac{\pi}6$. 
\end{proposition}

The proposition follows from the classical theory of slow-fast systems, see, e.g., 
\cite[theorem 3 and proposition 4]{ilguk}.
 \begin{figure}
 \begin{center}
 \vspace{-0.7cm}
\epsfig{file=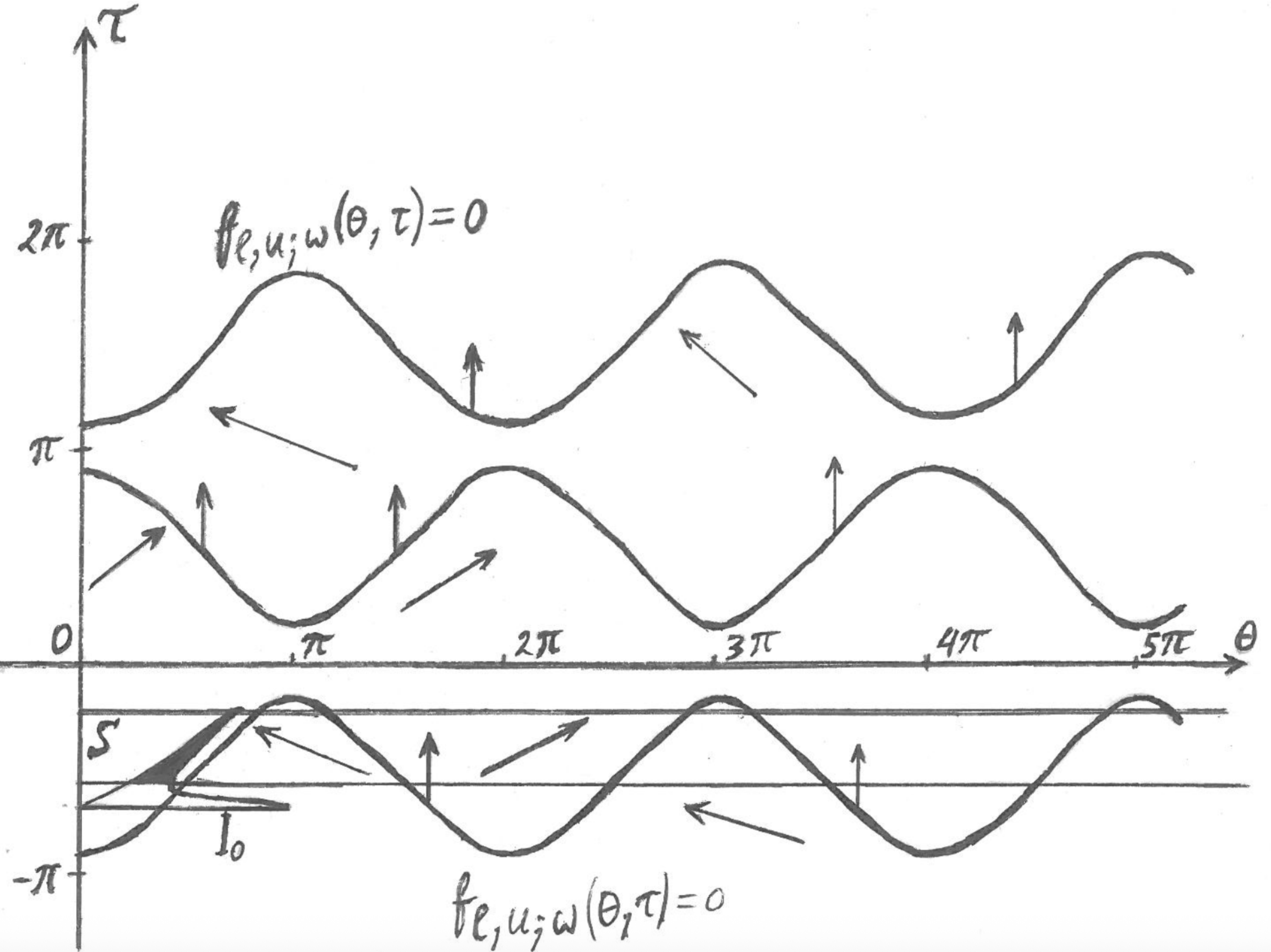, width=27em}
\caption{Phase portrait of field (\ref{jos2}) with $c_{\pm}=u\pm\ell>0$ and small $\omega>0$. The intersection of the orbit $\mco(I_0)$ with the strip $\mcs$ is the stable flowbox.}
\label{fig:slow}
  \end{center}
\end{figure}

\begin{corollary} If $R$ is big enough, as in Propositions \ref{proapr}, \ref{cl1}, then for every $\omega$ small enough the  flowbox 
$F_-(I_0)$ lies in the curvilinear trapezoid $\Pi_\delta$ and its forward orbit exits $\Pi_\delta$ through its upper base, see 
Fig. \ref{fig:pi}.  
\end{corollary}
\begin{proof} The flowbox lies on horizontal distance $O(\omega)$ from $\gamma$. For small $\omega$ it lies in $\Pi_\delta$:   $\delta\sqrt{2\omega}>O(\omega)$. Its forward orbit exits $\Pi_\delta$ through 
its upper base, since the field is directed inside $\Pi_\delta$ on its lateral sides and $\dot\tau=\omega>0$. 
\end{proof}
 \begin{figure}
 \begin{center}
\epsfig{file=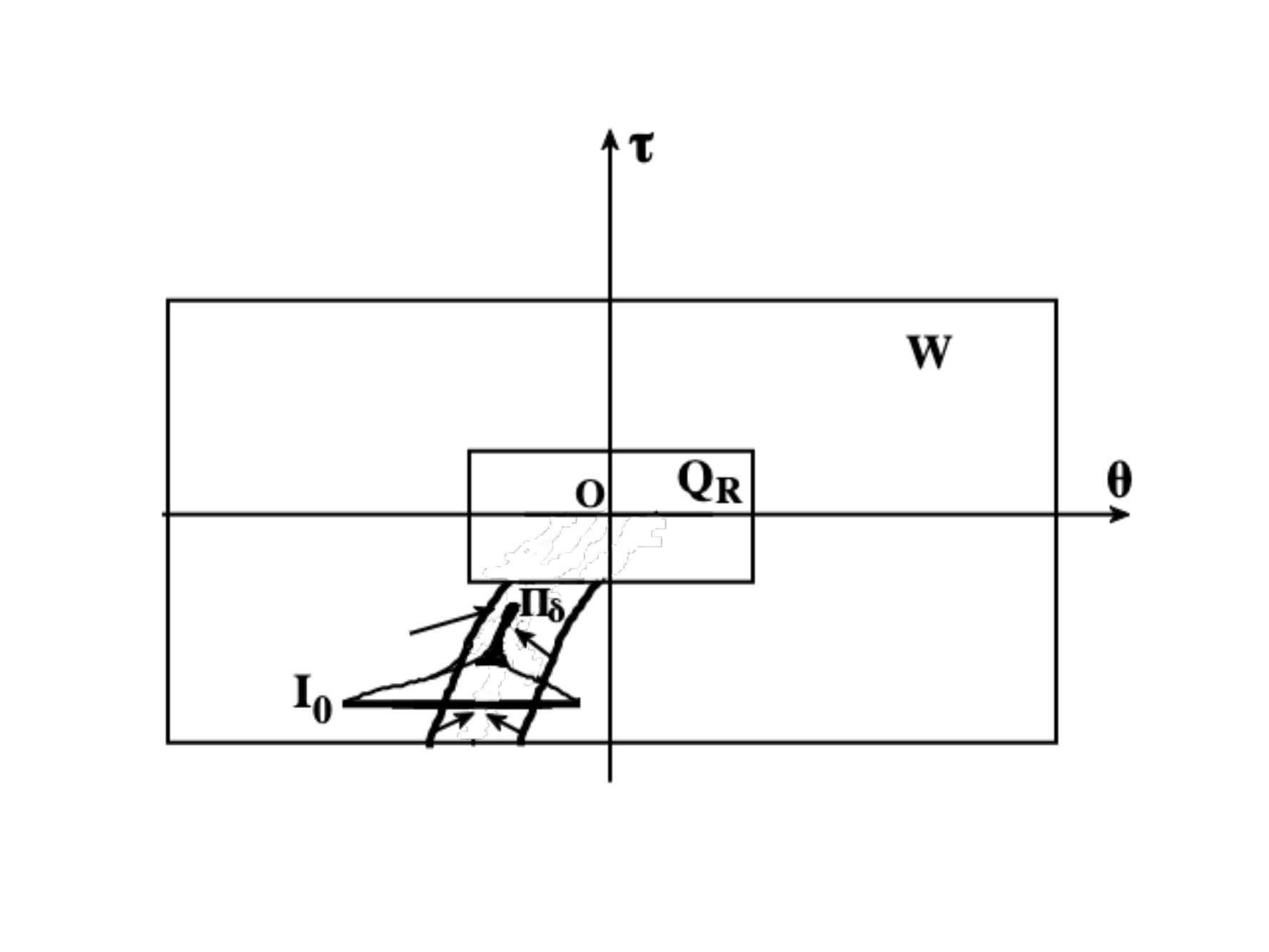, width=25em}
\vspace{-1.3cm}
\caption{The curvilinear trapezoid $\Pi_\delta$ contains the stable flowbox of the segment $I_0$ marked in bold thick.}
\label{fig:pi}
  \end{center}
\end{figure} 

Let us return to the proof of Lemma \ref{lconv1}. Fix a sequence $\omega_n\to0$ and a sequence  $\mco_n$ of forward extensions of orbits 
in the stable flowboxes $F_-(I_0)=F_-(I_0;\omega_n)$. Take a big $R>0$. For every $n$ large enough the orbit  
$\mco_n$ intersects the bottom side of the rectangle $Q_R$ at a point lying in the upper base of the curvilinear trapezoid $\Pi_\delta$, by the corollary. The latter upper base clearly lies in the left half-plane in the coordinates $(x,y)$, if $R$ is big enough, see Proposition \ref{proapr}, Statement 3). 
Hence, taking a subsequence one can achieve that in the 
Riccati chart the first arcs of intersections $\mco_n\cap Q_R$ converge to the graph of a solution $x=g(y)$ of Riccati equation (\ref{ricric}) starting at a point $(x_0,-R)$ in the bottom side, $-2R\leq x_0\leq 0$. Convergence is uniform in $s$, by compactness. The above statements remain valid for $R$ arbitrary 
large and $n$ large enough depending on $R$. Hence, $g(y)$ is a limit of solutions with initial conditions 
$(x_k,-R_k)$ with $R_k\to+\infty$ and $x_k\leq0$. Therefore, $g(y)=x_-(y)$, by Lemma \ref{lxyc}, Statement 1). Lemma \ref{lconv1} 
is proved.
\end{proof}

\subsection{Crossings of neighborhoods of subsequent critical points. Convergence to pieces of the same solution}  
Consider a slow-fast family (\ref{jos2g}) of vector fields on the annulus $S^1\times(-b,b)$, 
$b>0$, $S^1=\rr\slash2\pi\zz$, i.e., $f_{s;\omega}(\theta,\tau)$ is $2\pi$-periodic just in $\theta$. We consider that it satisfies 
compactness condition (\ref{compc}),  
 the function $f_{s;0}(\theta,0)$ vanishes only at $\theta=0(\mo2\pi)$, and 
the germ of $f_{s;\omega}(\theta,\tau)$ at $O=(0,0)$ is horizontally positive, see Definition \ref{horpos}. Thus $f_{s;0}(\theta,0)>0$ 
for $\theta\in(0,2\pi)$. Without loss of generality we consider that the Hessian form of the function $f_{s;0}$ at $(0,0)$ has form $\alpha(d\theta^2-d\tau^2)$, $\alpha=\alpha(s)>0$, 
as in the previous subsection; $\alpha(s)$ is uniformly bounded from above and from below, by compactness. For example, system (\ref{jos2}) restricted to parameters $(\ell,u)$, $\ell+u\in Z_{m,\var}$, 
$\ell-u\in\zz_{k,\var}$, with $\theta$ replaced by the variable $\wt\theta=\theta-\pi$ satisfies the above conditions. Then 
the corresponding Riccati coordinates $(x,y)$ can be taken as 
$$(x,y)=(\sqrt{\alpha\omega^{-1}}\theta,\sqrt{\alpha\omega^{-1}}\tau).$$

     We consider the lifting of system (\ref{jos2g}) to the strip $\rr_{\theta}\times(-b,b)$. The points $(0,0)$, $(2\pi,0)$ are its neighbor right flow points. 
      Let $U$ be the parameter $s$ space: it is compact. 
     For every $s\in U$ let $c=c(s)$ be the parameter of the corresponding 
Riccati equation (\ref{ricric}), see (\ref{defc}): it is the same for the points $(0,0)$ and $(2\pi,0)$. 
Let $(x,y)$, $(\tilde x,\tilde y)$ be the corresponding Riccati coordinates at $(0,0)$ and $(2\pi,0)$ respectively that differ by translation by $2\pi$: $(\tilde x,\tilde y)=(x(\theta-2\pi,\tau), y(\theta-2\pi,\tau))$. For every $a>0$ by $Q_a$, $\widetilde Q_a$ 
we will denote the rectangle $Q_a$ considered in the coordinates $(x,y)$ and $(\tilde x,\tilde y)$ respectively, i.e., centered 
at $(0,0)$ and $(2\pi,0)$ respectively. Thus, 
$$Q_a=\{|\theta|<2\sqrt{\alpha^{-1}\omega}a, \ |\tau|<\sqrt{\alpha^{-1}\omega}a\}.$$
We identify every subset in $\wt Q_a$ with its image in $Q_a$ by the translation $z\mapsto z-(2\pi,0)$, or equivalently, by the Riccati coordinate equality $(x,y)=(\wt x,\wt y)$. 

The main result of this subsection is the next lemma. 
\begin{lemma} \label{ml2} {\bf (The Second Main Lemma).} Let $s_*\in U$, and let $g(y)$ be a solution of the Riccati equation (\ref{ricric}) with $c=c(s_*)$.  Let $a>0$, and consider the intersection of the graph 
$\{ x=g(y)\}$ with the rectangle $Q_a$ in the Riccati coordinates. Let us order  components of the intersection 
in the sense of orientation 
of the graph by the $y$-coordinate. Let some two subsequent components $X_0=X_{0,a}$, $X_1=X_{1,a}$ be graphs of restrictions of the solution 
$g(y)$ to two intervals $(\alpha_0,\beta_0)$, $(\alpha_1,\beta_1)$ separated by exactly one pole $y_*\in(\beta_0,\alpha_1)$ of $x(y)$; then $X_0$ ends at the right lateral side $Q_a^{right}$ of $Q_a$ 
and $X_1$ starts at its left lateral side $Q_a^{left}$. 
Consider arbitrary converging 
sequences $s_n\to s_*$,  $\omega_n\to0$ and a sequence of orbits $\mco_{n,0}$ 
of systems (\ref{jos2g}) in $Q_a$ with $s=s_n$, $\omega=\omega_n$. Let $\mco_{n,0}\to X_0$ in the Riccati coordinates $(x,y)$. Then for  every $n$ large enough there exists a 
forward  extension $\wh\mco_n$ of the arc $\mco_{n,0}$ that intersects  $\widetilde Q_a$ by an arc $\mco_{n,1}$  converging to $X_1$, as $n\to\infty$, see the above identification. Namely, each $\wh\mco_n$ consists of three subsequently adjacent arcs: the arc $\mco_{n,0}$ ending 
at  $Q_a^{right}$;  
an  arc starting at its endpoint and ending at $\widetilde Q_a^{left}$, diffeomorphically projected to  the segment of the $\theta$-axis connecting $Q_a^{right}$ and $\widetilde Q_a^{left}$; the arc $\mco_{n,1}$. 
\end{lemma}
As is shown below, the lemma is implied by the two following propositions.
\begin{proposition} In the conditions of Lemma \ref{ml2} the arc $X_0$ ends at $Q_a^{right}$ and $X_1$ starts at 
$Q_a^{left}$.
\end{proposition}
\begin{proof} The arc $X_1$ starts at $Q_a^{left}$, since it clearly cannot start at any other side of $Q_a$, being not the first arc of its intersection with the graph.  
The fact that $X_0$ ends at $Q_a^{right}$ is proved analogously. 
\end{proof}
\begin{proposition} \label{pl2} In the conditions of Lemma \ref{ml2} there exists a universal constant $\sigma_2>0$ such that
 for every $R>0$ large enough and $n$ large enough depending on $R$ the following statements hold.  There exists a  forward extension of the orbit $\mco_{n,0}$ that is bijectively projected to a segment  in the $\theta$-axis and consists of  two subsequent arcs: 
the immediate extension of the arc $\mco_{n,0}$ to an orbit in $Q_R$ ending at a point $A_n\in Q_R^{right}$; an arc $A_nA'_n$ going to a point $A'_n\in\widetilde Q_R^{left}$. One has 
\begin{equation}0< \tau(A'_n)-\tau(A_n)<\frac{\sigma_2}{R}\sqrt{\omega}.\label{taubnan}\end{equation}
\end{proposition}
 
\begin{proof} For the proof we consider the next auxiliary sector and half-plane: 
$$S=S(\chi):=\{|\tau|<\chi|\theta|\}, \ \chi\in(0,\min\{\frac12, \frac b{3\pi}\}], \ \ \ \mch:=\{\theta<\pi\}.$$
Here  $\chi$ is fixed and independent on $(s;\omega)$. 
For every $R,\omega>0$ we consider the trapezoid  
$$S_R:=(S\setminus Q_R)\cap\mch,$$
see Fig \ref{fig:sr}. Its two lateral sides lie in  $\partial S$, and its left and right vertical bases are 
 $$\mcl_R:=S\cap\{ x=2R\}\subset Q_R^{right}, \ \ \mcr_R:=S\cap\{\theta=\pi\}\Subset\rr_{\theta}\times(-b,b).$$
 The latter inclusions follow from the assumption $\chi\in(0,\min\{\frac12, \frac b{3\pi}\}]$. 
 
  In what follows we prove that for every $R$ big enough and for every $n$ large enough depending on $R$  the forward extension of the orbit $\mco_{n,0}$  enters $S_R$ through $\mcl_R$ and leaves it through $\mcr_R$ in a relatively 
 small time so that the increment of the coordinate $\tau$ along its arc in $S_R$ is no greater than $\frac{\sqrt\omega}R$ 
 up to a universal constant factor. Applying the same argument to the other right flow point $(2\pi,0)$ and to the 
  inverse dynamics yields that the orbit arrives to $\widetilde Q_R$, and the total increment of $\tau$ along its arc between 
  the rectangles $Q_R$ and $\widetilde Q_R$ is bounded by a similar quantity.

  To prove the above statements, we use the next a priori bounds. 
   There exists a  $\sigma_1>0$ such that for every $R$ large enough, $\omega$ small enough depending on $R$ and 
   every $s\in U$  one has 
\begin{equation}f_{s;\omega}(\theta,\tau)>\sigma_1\theta^2 \  \ \ \text{ for every } (\theta,\tau)\in S_R.\label{ineqf1}
\end{equation}
 Indeed, this holds for $f_{s;0}$ with a certain $\sigma_1$, since the Hessian form of the function $f_{s;0}$ at $(0,0)$ is  $\alpha(s)(d\theta^2-d\tau^2)$, $\alpha(s)>0$ being bounded from above and from below by uniform constants.  If $R$ is big enough, then inequality 
 (\ref{ineqf1}) remains valid for $f_{s;\omega}$ with $\omega$ small enough depending on $R$. Indeed,   $f_{s;\omega}(\theta,\tau)$ 
 has derivative in $\omega$  continuous in $(s,\theta,\tau)$, and hence, the  module of derivative is bounded by some $\eta>0$. One has 
 \begin{equation}\sigma_1\theta^2\geq 4\sigma_1\alpha^{-1}R^2\omega \ \ \text{ on } \ S_R.
 \label{si1r}\end{equation}
  Take a big $R$ so that $\sigma_1\alpha^{-1}R^2>\eta$ for every $s\in U$. Then  $|f_{s;\omega}(\theta,\tau)-f_{s;0}(\theta,\tau)|\leq\eta\omega<\sigma_1\alpha^{-1}R^2\omega\leq\frac{\sigma_1\theta^2}4$ on $S_R$, by (\ref{si1r}). Hence,  (\ref{ineqf1}) holds for $f_{s;\omega}$ 
 with $\sigma_1$ replaced by 
 $\frac{\sigma_1}2$.

 \begin{claim} \label{cl11} {\it For every $\chi\in(0,\min\{\frac12, \frac b{3\pi}\}]$, $R>0$ large enough depending on $\chi$,  
 and $\omega$ small enough depending on $R$ and $\chi$  
 the restriction of the field (\ref{jos2g}) to the lateral sides of the trapezoid $S_R$ is directed inside $S_R$.}
  \end{claim}
 \begin{proof} The vectors of the field (\ref{jos2g}) at points in $\overline S_R$ are directed up-right.  
 Their slopes  have modules less than $\frac{\omega}{4\sigma_1\alpha^{-1}R^2\omega}=\frac1{4\sigma_1\alpha^{-1} R^2}$, by 
  (\ref{ineqf1}) and (\ref{si1r}). The latter quantity is less than the module 
 $\chi$ of slopes of the lateral sides of $S_R$, if $R$ is big enough. Then the field is directed inside $S_R$ on its lateral sides.
 \end{proof}
 \begin{claim} \label{cl12} {\it In the above claim each orbit of (\ref{jos2g}) in $S_R$ starting at a point  $A\in\mcl_R$ arrives at 
 a point  $B\in\mcr_R$  in time 
 \begin{equation}t<\frac1{\sigma_1R\sqrt{\alpha^{-1}\omega}},\label{t<}\end{equation} 
 whenever $\omega$ is small enough depending on $R$, and one has 
 \begin{equation}0<\tau(B)-\tau(A)\leq\nu=\nu(R,\omega):=\frac{\sqrt{\alpha\omega}}{\sigma_1R}, \ \ \ \alpha:=\max_s\alpha(s).\label{tau<}\end{equation}
The above orbit is diffeomorphically projected to a segment in $\rr_\theta$.}
 \end{claim}
 \begin{proof} The fact that the orbit in question remains in $S_R$ and then arrives to $\mcr_R$ follows from 
 Claim \ref{cl11}. The upper bound of time follows by integrating the differential inequality implied by the equation 
 $\dot\theta=f_{s;\omega}(\theta,\tau)$ and  (\ref{ineqf1}): 
 $$t=\int_{\theta_0}^{\pi}\frac{d\theta}{f_{s;\omega}(\theta,\tau)}\leq\int_{\theta_0}^{\pi}\frac{d\theta}{\sigma_1\theta^2}\leq\frac1{\sigma_1}\left(\frac1{\theta_0}-\frac1{\pi}\right), \ \ \theta_0=\theta(A)=2R\sqrt{\alpha^{-1}\omega},$$
 $$t\leq\frac1{\sigma_1}\left(\frac1{2R\sqrt{\alpha^{-1}\omega}}-\frac1{\pi}\right)<\frac1{\sigma_1R\sqrt{\alpha^{-1}\omega}}\ \ \  \ \ \text{ for small } \omega.$$
 This proves (\ref{t<}), which together with the equation $\dot\tau=\omega$ implies (\ref{tau<}). The last statement of the claim follows from (\ref{ineqf1}). 
 \end{proof} 
 
  Applying the above arguments to the point $(2\pi,0)$ and  system (\ref{jos2g}) in the inverse time  
   yields similar statements for 
a sector 
$$S'=S'(\chi')=-S(\chi')+(2\pi,0), \ \ \chi'\in(0,\min\{\frac12, \frac b{3\pi}\}],$$ 
see Fig. \ref{fig:sr}. 
This is a sector directed to the left with vertex $(2\pi,0)$. By  $S_R'$ we denote the new trapezoid $S_R'=-S_R+(2\pi,0)$, with $S$ replaced by $S'$ and $Q_R$ replaced by $\widetilde Q_R$. By $\mcl_R'$ we denote its right base, which lies in 
$\wt Q_R^{left}$. The constant  $\sigma_1$ corresponding to  $S'$ will be denoted by $\sigma_1'$. 
\begin{figure}
 \begin{center}
 \vspace{-0.7cm}
\epsfig{file=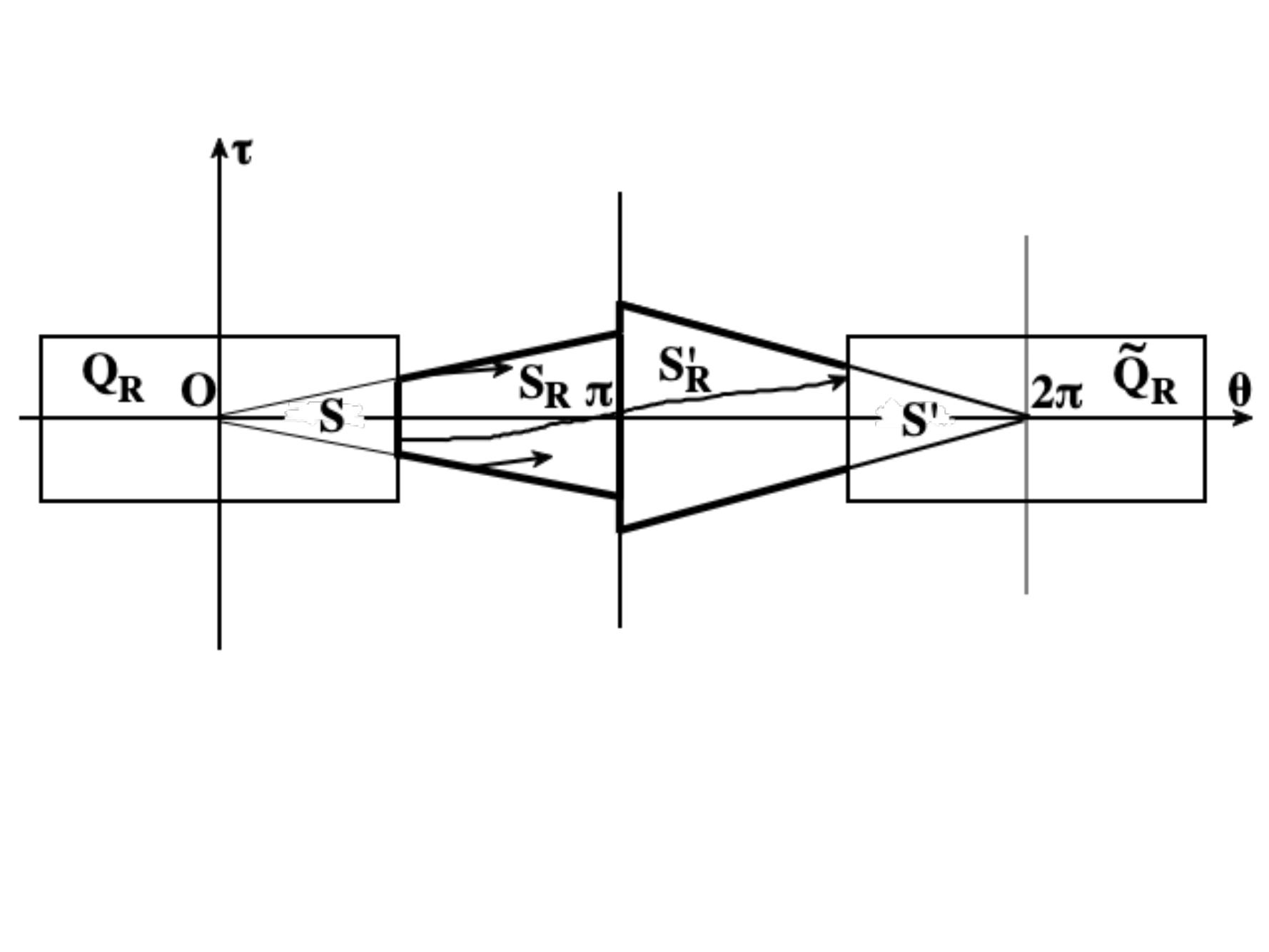, width=23em}
\vspace{-1.5cm}
\caption{The trapezoids $S_R$ and $S'_R$.}
\label{fig:sr}
  \end{center}
\end{figure} 
For the proof of Proposition \ref{pl2} we consider the  Poincar\'e maps 
$\mcp_0:\mcl_R\to\mathcal Y:=\{\theta=\pi\}$, $\mcp_1:\mcl_R'\to\mathcal Y$.  Each of them sends a given initial condition in $\mcl_R$, $\mcl_R'$
  to the point of the of first arrival to $\mathcal Y$ of the corresponding forward, respectively backward orbit of system (\ref{jos2g}). They are diffeomorphisms onto segments denoted $J_R, J_R'\subset\mathcal Y$ respectively. 

\begin{claim} 
{\it If  $\chi<\chi'$,  then $J_R\Subset J_R'$ for every $R$  large enough depending on $\chi$, $\chi'$ and every $\omega$ small enough depending on 
$\chi$, $\chi'$, $R$.}
 \end{claim}
 
 \begin{proof} We identify each vertical segment with its projection to the $\tau$-axis. After this identification, each point of the segment $\mcl_R$ differs from 
 its  $\mcp_0$-image  by a quantity of module less than  $\nu=\nu(R,\omega)$, see (\ref{tau<}). 
 Similar statement holds for $\mcl_R'$ and  $\mcp_1$-image, and $\sigma_1$ replaced by $\sigma_1'$; the corresponding gap will be denoted by $\nu'$. If $\chi<\chi'$, then $\mcl_R$ lies in the interior of the segment 
 $\mcl_R'$, and the gap between them is equal to $2(\chi'-\chi)R\sqrt{\alpha^{-1}(s)\omega}$. It is clearly bigger than $4\nu$ and $4\nu'$, 
if $R$ is large enough depending on $\chi$ and $\chi'$. Then the inclusion  of the claim holds.
\end{proof}
 
 Let us return to the proof of Proposition \ref{pl2}. Fix $\chi,\chi'\in(0,\min\{\frac12,\frac b{3\pi}\}]$ with $\chi'>\chi$, e.g., with $\chi=\frac{\chi'}2$, and consider the corresponding sectors $S$ and $S'$.  
 For every  $R$ large enough and every $n$ large enough depending on  $R$ the forward extension of the arc $\mco_{n,0}$ crosses  $S_R$ by an arc going from a point 
 $A_{n,R}\in\mcl_R$ to a point $B_{n,R}\in\mcr_R\subset\mathcal Y=\{\theta=\pi\}$. Indeed, the limit of orbits $\mco_{n,0}$ is an orbit $\{ x=g(y)\}$ of the Riccati equation that goes to infinity in a finite time $y_*$. It thus crosses the interior of the side $\mcl_R$, if $R$ is big enough. Then the forward extension of the orbit $\mco_{n,0}$ also crosses $\mcl_R$ at 
 some point $A_n$, if $n$ is large enough. Hence, it arrives to some point $B_n\in\mcr_R$. The arrival point $B_n$ lies in $J_R$. We take $R$ so large that  $J_R'\Supset J_R$ for $n$ big enough depending on $R$, see the above claim. Then   the forward extension of the orbit arc $A_nB_n$  arrives at a point 
 $A_n'\in\mcl'_R$ and 
 $$\tau(A_n')-\tau(A_n)=(\tau(A_n')-\tau(B_n))+(\tau(B_n)-\tau(A_n))\leq \frac{\sigma_2}R\sqrt\omega,$$
  $\sigma_2=2\sqrt\alpha\max\{\sigma_1^{-1},\sigma_1'^{-1}\}$, 
 by (\ref{tau<}). Proposition \ref{pl2} is proved.
 \end{proof}
 
 \begin{proof} {\bf of Lemma \ref{ml2}.} Let $\mco_{n,0}$ be the orbits from the lemma in $Q_a$. Take an $R>a$ 
 large enough. Both  $\mco_{n,0}$ and its limit $X_0$ extend  to $Q_R$ as orbit arcs, and let $\mco_{n,0}$, 
 $X_{0,R}$  denote the latter extended orbits.  In the Riccati coordinates $\mco_{n,0}\to X_{0,R}$ by assumption.  
 Let, as above, $\widetilde Q_R$ denote the translation copy   
 of the rectangle $Q_R$ centered at $(2\pi,0)$. Let $A_n,A\in Q_R^{right}$ denote the endpoints of the orbits 
 $\mco_{0,n}$, $X_{0,R}$.  
 For every $n$ large enough the forward extension of the orbit $\mco_{n,0}$ contains an arc $A_nA_n'$ with 
 $A_n'\in\widetilde Q_R^{left}$, and $0<\tau(A_n')-\tau(A_n)\leq\frac{\sigma_2}{R}\sqrt{\omega}$, by Proposition \ref{pl2}. In the Riccati coordinates $(x,y)$ one has $A_n\to A$, and the sequence $\wt y(A_n')$ 
 is bounded  in the Riccati coordinates $(\tilde x,\tilde y)$ centered at $(2\pi,0)$. 
 
 Thus, the forward extension of the orbit $\mco_{n,0}$ intersects $\wt Q_R$ by an arc $\wt X^n$. Passing to  a subsequence we can and will 
 consider that $A_n'$ converge to a point $A'$ in the Riccati coordinates. Thus, $\wt X^n$  converge to the orbit $\wt X_{1,R}$  of Riccati equation (\ref{ricric}) passing through $A'$ in the closure of $\wt Q_R$.  One has 
 $0<\tilde y(A')-y(A)\leq\frac{\sigma_2}{R}\sqrt{\alpha(s)}$, by construction and the above inequality on the $\tau$-coordinates, passing to limit. If $R$ is large, then  $A$ is close  to the infinite point $(\infty,y_*)$ separating the arcs $X_{0,R}$ and $X_{1,R}$ in the ambient orbit of (\ref{ricric}). But then $A'$ is also close to $(\infty,y_*)$, by the above inequality on $y(A')$ and since $\wt x(A')=-2R$.  One has $X_{1,R}\cap\wt Q_a= X_1=X_{1,a}$. 
 Thus, the limit  $\wt X_{1,R}$ starts at an initial condition $A'$ close to the initial condition 
 $(\infty,y_*)\in\rp^1\times\rr_y$ of the Riccati equation orbit crossing the rectangle $\wt Q_a$ by  $X_{1}$. Hence, $\wt X_{1,R}$ crosses $\wt Q_a$ by an arc close to $X_{1}$, if $R$ is big. Therefore, $\wt X^n\cap\wt Q_a\to X_{1}$, since $R$ is arbitrarily large. Lemma \ref{ml2} 
 is proved.
\end{proof}

\subsection{The parquet. Proofs of Theorems \ref{mt2}, \ref{mt}, \ref{thm2} and Statement 1) of Theorem \ref{thm1}}

First we prove Theorem \ref{mt2}, for which Theorem \ref{mt} is a particular case. Then we deduce 
Theorem \ref{thm2} and 
Statement 1) of Theorem \ref{thm1}, which is its particular case, using discussion in Subsection 1.3 and the argument below.  

\begin{proof} {\bf of Theorem \ref{mt2}.}   Let us prove the statement of Theorem \ref{mt2} for the segment $I_0$ and the right flow point $z_0$ for every $\omega$ small enough depending on $m$, $k$, $\var$. Without loss of generality we consider that $z_0$ is the origin $O=(0,0)$. We fix a cylinder $S^1_{\theta}\times(-b,b)$ containing $I_0$ and $I_1$. We  
will work in the coordinates $(\wt\theta,\wt\tau)=(\theta+\mu\tau,\tau)$ 
on the latter cylinder  
in which the Hessian form of the function $f_{s;0}$ at $O$ takes the form $\alpha(s)(d\wt\theta^2-d\wt\tau^2)$, 
and from now on we will denote the latter coordinates by $(\theta,\tau)$, as in the previous subsection. Fix a rectangle 
$W=\{|\theta|<2M, \ |\tau|<M\}$ as at the beginning of Subsection 2.5 that contains the segments $I_0$ and $I_1$. Recall that the Riccati coordinates are $(x,y)=(\sqrt{\alpha\omega^{-1}}\theta,\sqrt{\alpha\omega^{-1}}\tau)$. Set 
$$Q_R=\{|x|<2R, \ |y|<R\}=\{|\theta|<2\sqrt{\alpha^{-1}\omega}R, \ |\tau|<\sqrt{\alpha^{-1}\omega}R\},$$
$$z_j=(2\pi j,0)\in\rr^2_{\theta,\tau},  \ \ Q_{R,j}:=Q_R+(2\pi j,0), \ \ W_{j}=W+(2\pi j,0), 
 \ \ j\in\zz,$$
$$(x_j,y_j)=(x(\theta-2\pi j),y) \ \text{ are Riccati coordinates on } \ Q_{R,j} \ \text{ centered at } z_j.$$

We consider that $s\in\wt U\Subset U$ is such that $c_+\in Z_{m,\var}$, $c_-\in Z_{k,\var}$. 

Everywhere below for any two values $\tau_1<\tau_2$ by $\mcp^{\tau_1,\tau_2}$ we denote the Poincar\'e map: 
the flow map in time $\frac{\tau_2-\tau_1}{\omega}$ of vector field (\ref{jos2g}) from the line $\{\tau=\tau_1\}$ to the line $\{\tau=\tau_2\}$. 

Let $\wh\tau_0$, $\wh\tau_1$ denote the values of the coordinate $\tau$ on the horizontal segments $I_0$ and $I_1$ respectively. Let $\tau_{R,+}=\sqrt{\alpha^{-1}\omega}R$ denote the value of $\tau$ on the upper side $Q_R^{up}$ 
of the rectangle $Q_R$. 

Step 1: the image $\mcp^{\wh\tau_0,\tau_{R,+}}(I_0)$ lies in the interior of the  subsegment $I_{0,R}:=(-\frac{11}{10}R,-\frac9{10}R)$ of the side 
$Q_{R,m}^{up}$ in the coordinate $x_m$, whenever $R$ is large enough and $\omega$ is small enough depending on $R$. 

To prove this, we consider the stable solution $x_-(y)$ of the Riccati equation (\ref{ricric}) corresponding to $O$. 
We consider that  $R$ is greater than the constant $a_0(m,\var)$ from Proposition \ref{cal2}: then the rectangle $Q_R$ satisfies its statements. Namely, 
let $X_0,\dots, X_m$ denote the corresponding components of the intersection of the graph of the solution $x_-(y)$ with $Q_R$. Then $X_m$ ends at a point in $Q_R^{up}$ with $x$-coordinate lying in the interval 
$(-\frac{11}{10}R,-\frac9{10}R)$. The  intersection $\mco(I_0)\cap Q_R$ converges  to $X_0$ in the Riccati coordinates $(x,y)=(x_0,y_0)$, as $\omega\to0$, by Lemma \ref{lconv1}. 
The  orbit $\mco(I_0)$ crosses each $Q_{R,j}$ by a family of orbit arcs that converge uniformly to 
$X_j$ in the Riccati coordinates $(x_j,y_j)$, as $\omega\to0$. Indeed, this holds after passing to any subsequence of orbits $\mco_n$ in $\mco(I_0)$ corresponding to converging parameter sequences $\omega_n\to0$ and $s_n$. This follows by Lemma \ref{ml2} and induction in $j$. 
Applying this statement to $j=m$ and taking into account that the  endpoint of the arc $X_m$ lies in $(-\frac{11}{10}R,-\frac9{10}R)$, we get that $\mcp^{\wh\tau_0,\tau_{R,+}}(I_0)\subset \opint(I_{0,R})$, whenever $\omega$ is small enough. 

Step 2:  $\mcp^{\tau_{R,+},\wh\tau_1}(I_{0,R})\subset \opint(I_1')=\opint(I_1)+(2\pi m,0)$, whenever $R$ is large enough and 
$\omega$ is small enough depending on $R$. To prove this, we consider the upper left component  $\gamma$ of the curve  $[(C_{s;\omega}\cap W)\setminus Q_R]+(2\pi m,0)\subset W_m$.  
It starts at a point  $z\in Q_{R,m}^{up}$ lying on the left from the $y_m$-axis and crosses the interior of the segment $I_1'$ at a point $q$ uniformly bounded away from its ends: the point $q$ converges to the intersection $I_1'\cap [C_{s;0}+(2\pi m,0)]$, 
as $\omega\to0$. 
The coordinate $x_m(z)$ is $\frac R{20}$-close to $-R$, whenever $R$ is big enough and $\omega$ is small enough depending on $R$, see Proposition \ref{proapr}, Statement 3). Thus, we can and will consider that  $z\in I_{0,R}$. The curves  
$\gamma_{\pm}=\gamma\pm(\frac15\sqrt{\alpha^{-1}\omega}R,0)$ start at points $z_{\pm}\in Q_{R,m}^{up}$ and cross $\opint(I_1')$ 
at some points $q_{\pm}$.  
 Let $\Pi_{\delta}=z_+q_+q_-z_-$ denote the curvilinear trapezoid with bases contained in $Q_{R,m}^{up}$ and $\opint(I_1')$ and lateral sides lying in $\gamma_{\pm}$. The field (\ref{jos2g}) is directed inside $\Pi_{\delta}$ on its lateral sides, 
 whenever $R$ is big enough and $\omega$ is small enough depending on $R$, as in Proposition \ref{cl1}. Its lower 
 base contains $I_{0,R}$, by construction. Therefore, the forward orbit of the segment $I_{0,R}$ exits $\Pi_{\delta}$ through 
 its upper side lying in $\opint(I_1')$. The inclusion of Step 2 is proved.
 
 The Poincar\'e map $\mcp^{\wh\tau_0,\wh\tau_1}$ is the composition of the Poincar\'e maps from Steps 1 and 2. 
 It sends $I_0$ strictly inside the segment $I_1'$ for small $\omega$, by results of theses steps. This proves Statement 1) of Theorem 
 \ref{mt2}. 
 
 To outline the next three steps, which will prove Statement 2) of Theorem \ref{mt2}, let us do the following auxiliary construction. Consider the closed singular orbit 
 $\phi$ of the slow-fast system (\ref{jos2}) on $\tt^2$ starting at $z_0$. It consists of two arcs: the arc $\phi_0$ going from 
 $z_0$ to $w_0$; the arc $\phi_1$ going from $w_0$ to $z_0$. The lifting to $\rr^2$ of the arc $\phi_0$ starting at 
 $z_0$ ends at a point $w_{-j}=w_0+(-2\pi j,0)$ for some $j\in\zz$, while the lifting of the arc $\phi_1$ starting at $w_0$ 
 ends at $z_{j,1}=z_0+(2\pi j,2\pi)$ with the same $j$. This follows from the condition of Theorem \ref{thm2}, and hence, 
 of Theorem \ref{mt2}, saying that $\phi=\phi_0\phi_1$ is homotopic to the circle $\{0\}\times S^1_\tau$. 
 
  Consider the lifting $\wt\phi_0$ of the arc $\phi_0$ to $\rr^2$ that starts at $z_m$. The arc $\wt\phi_0$  ends at the left flow point $w_{m-j}=w_0+(2\pi(m-j),0)$. Let us 
 introduce a rectangular neighborhood $\wt W$ of the 
 point $w_{m-j}$ and auxiliary horizontal segments $J_0, J_1\subset \wt W$, 
 analogous to the segments $I_0$, $I_1$. But now the interiors of the segments
 $J_0$, $J_1$ intersect the stable curves $C_{\mcl,-}(w_{m-j})$ and $C_{\mcr,+}(w_{m-j})$ respectively, instead of 
 $C_{\mcr,-}$ and $C_{\mcl,+}$ in the case of point $z_0$. Let $\wt\tau_0$, $\wt\tau_1$ denote the $\tau$-coordinates 
 of the segments $J_0$ and $J_1$. 
 
 Step 3: $\mcp^{\wh\tau_1,\wt\tau_0}(I_1')\Subset J_0$ for $\omega$ small enough. This follows by the standard slow-fast system theory. Namely, 
 the forward orbit of the segment $I_1'$ drifts along the arc $\wt\phi_0$ as a narrow flowbox that becomes 
 exponentially narrow during drifts along stable arcs of the singular orbit $\wt\phi_0$. Thus, it will enter the interior 
 of the segment $J_0$, since $\phi_0$ crosses it. 
 
 Step 4: $\mcp^{\wt\tau_0,\wt\tau_1}(J_0)\Subset J_1':=J_1-(2\pi k,0)$ for $\omega$ small enough. This is proved analogously to Step 1. But now 
 the arguments are applied to the vector field family obtained from family (\ref{jos2g}) by the coordinate change 
 $\theta\mapsto-\theta$, which interchanges left flow points and right flow points.
 
 Step 5: $\mcp^{\wt\tau_1,\wh\tau_0+2\pi}(J_1')\Subset I_0'=I_0+(2\pi(m-k),2\pi)$. Indeed, consider the lifting $\wt\phi_1$ 
 to $\rr^2$ starting at $w_{m-j-k}$ of the arc $\phi_1$. It ends at $z_{m-k,1}=z_{m-k}+(0,2\pi)$ and crosses $\opint(I_0')$, by the above discussion and construction. 
 The forward orbit of the segment $J_1'$ drifts along the arc $\wt\phi_1$ and enters the interior of the segment 
 $I_0'$, as in Step 3. This proves the above inclusion. 
 
 The Poincar\'e map $\mcp^{\wh\tau_1,\wh\tau_0+2\pi}$ is the composition of the Poincar\'e maps from Steps 3--5. 
 This together with results of Steps 3--5 implies that it sends $I_1'$ strictly inside the segment  $I_0'$ and finishes the proof of Theorem \ref{mt2}.
  \end{proof}
 
  \begin{proof} {\bf of Theorem \ref{thm2}.} Let $\wt U\subset U$ be an arbitrary compact subset of the $s$-parameter space $U$. For every $m,k\in\zz_{\geq0}$, $\var\in(0,1)$, $\omega$ small enough depending on $\wt U$, $m$, $k$, $\var$,  for every $s\in\wt U$ with 
  $\Sigma(s)=(c_+(s), c_-(s))\in Z_{m,\var}\times Z_{k,\var}$  
  the Poincar\'e map $\mcp^{\tau(b_0),\tau(b_0)+2\pi}$ sends $I_0$ to $I_0'=I_0+(2\pi r,2\pi)$, $r=m-k$, by Statements 1) and 2) of 
  Theorem \ref{mt2}. This implies that after the above shift identification $I_0'\simeq I_0$ the Poincar\'e map becomes a self-map $I_0\to I_0$. Hence, it has a fixed point. This implies that the corresponding flow has a  $2\pi$-periodic orbit intersecting $I_0$ with rotation number $r$, and  $s$ lies in the phase-lock area $L_r(\omega)$.  Thus, $\Sigma^{-1}(Z_{m,\var}\times Z_{k,\var})\cap\wt U\subset L_r(\omega)$.  
  
  Now for the proof of Theorem \ref{thm2} it   remains to show that $L_r(\omega)\cap \wt U$ accumulates to no point $z\in \wt U\setminus\cup_{\mu-\kappa=r}\Sigma^{-1}(\overline Z_{\mu}\times \overline Z_{\kappa})$. That is the place where we use submersivity of the map 
 $\Sigma$, its local non-constance in some variable $s_j$, say, $s_1$, and monotonicity of the function $f_{s;\omega}(\theta,\tau)$  in the same variable $s_1$ for small $\omega$.  The parameter space $U$ is the disjoint union of open sets $\Sigma^{-1}(Z_\mu\times Z_\kappa)$, 
 $\mu,\kappa\in\zz_{\geq0}$, and their boundaries $\Sigma^{-1}(\partial(Z_\mu\times Z_\kappa))$, since $\Sigma$ is a submersion. Suppose the contrary: an accumulation point $z\in\wt U$ 
as above exists. It cannot lie in  $\Sigma^{-1}(Z_\mu\times Z_\kappa)$ with $\mu-\kappa\neq r$, 
since $\Sigma^{-1}(Z_{\mu,\var}\times Z_{\kappa,\var})\subset L_{r'}(\omega)$, $r'=\mu-\kappa\neq r$ for arbitrarily small $\var>0$ for every $\omega$ small enough depending on $\var$, and the phase-lock areas $L_r(\omega)$ and $L_{r'}(\omega)$, $r\neq r'$, are disjoint. Thus, the only a priori possible case is when $z=\Sigma^{-1}(w)$, where $w$ lies in the boundary of several subsets $Z_\mu\times Z_\kappa$, and for all the latter subsets adjacent to $w$ one has 
  $\mu-\kappa\neq r$.  Let us show that this is impossible.

Fix a  segment $[\alpha_1,\alpha_2]\subset\rr_{s_1}$ of small length $\delta=\alpha_2-\alpha_1>0$, centered at $s_1(z)=\frac{\alpha_1+\alpha_2}2$. Let $[A_1,A_2]$ denote the segment 
  through $z$ parallel to the $s_1$-axis and projected to $[\alpha_1,\alpha_2]$.  Passing to a sequence of values 
  $\delta=\delta_n\to0$, one can achieve that $c_+(A_j)$, $c_-(A_j)$ are non-integer and thus, lie in some sets $Z_{\mu_j}$, 
  $Z_{\kappa_j}$, $\mu_j,\kappa_j\in\zz_{\geq0}$, by local non-constance of $c_{\pm}$ in $s_1$, see condition (vi) of Theorem \ref{thm2}.  Moreover, passing to a subsequence we can and will consider that $\mu_j$, $\kappa_j$, and hence $r_j=\mu_j-\kappa_j$ remain 
  the same: independent on $n$.    Passing to limit, we get that $r_j\neq r$, since $z\notin\overline{L_r^0}$. 
 The domains $Z_{\mu_j}\times Z_{\kappa_j}$ are thus 
 adjacent to each other at $w$. This implies that either $r_1=r_2$, or $|r_1-r_2|=1$, or $|r_1-r_2|=2$ and 
 $w$ also lies in the boundary of another domain $Z_{\mu}\times Z_{\kappa}$ with $\mu-\kappa=r':=\frac{r_1+r_2}2$. 
By assumption, $w$ is not adjacent to $Z_{\mu}\times Z_{\kappa}$ with $\mu-\kappa=r$. This together with 
the above statements implies that $r\notin[r_1,r_2]$. 
  
  Fix   $\delta=\delta_n$ and $\mu_j$, $\kappa_j$, $r_j$ as above. 
  Take a small neighborhood $W$ of the projection 
  of the point $z$ to the $(s_2,\dots,s_n)$-subspace. For every $\xi\in W$ let $[A_{\xi,1},A_{\xi,2}]$ denote the segment 
  in the $s$-space parallel to the $s_1$-axis, projected to $[\alpha_1,\alpha_2]$ and centered at $(s_1(z),\xi)\in\rr^n_s$. The intervals $(A_{\xi,1},A_{\xi,2})$ saturate a neighborhood of the point $z$ in $\rr^n_s$, which will be denoted by $V$. Taking $W$ small depending on $z$ and  $\var>0$ small depending on 
  $z$, $\delta$, $W$ one can achieve that 
  $A_{\xi,j}\in\Sigma^{-1}(Z_{\mu_j,\var}\times Z_{\kappa_j,\var})$ for all $\xi\in W$. Then one has $A_{\xi,j}\in L_{r_j}(\omega)$ for all $\omega$ small enough and all $\xi\in W$, since 
  $\Sigma^{-1}(Z_{\mu_j,\var}\times Z_{\kappa_j,\var})\subset L_{r_j}(\omega)$ for small $\omega$, see the beginning of the proof of Theorem \ref{thm2}. Thus,  $\rho(A_{\xi,j};\omega)=r_j$ 
  for every $\omega$  less than some value $\omega_0>0$. But the rotation number is monotonous in $s_1$, as is $f_{s;\omega}$. Therefore, 
$\rho([A_{\xi,1},A_{\xi,2}]\times(0,\omega_0))\subset[r_1,r_2]$ for every $\xi\in W$. 
  Thus, for every $\omega\in(0,\omega_0)$ the rotation number function $\rho(s)$ restricted to $V$ takes values 
  in the segment $[r_1,r_2]$, which is at least 1-distant from the number $r$. This contradicts to the assumption that 
  $L_r(\omega)$ accumulates to $z$, as $\omega\to0$. The contradiction thus obtained proves Theorem \ref{thm2}.
\end{proof}

Statement 1) of Theorem \ref{thm1} follows from Theorem \ref{thm2} proved above: in its conditions  
$(s_1,s_2)=(\ell,u)$,  $c_\pm=u\pm\ell$ are strictly monotonous in $\ell$ and $f_{\ell,u;\omega}(\theta,\tau)=\cos\theta+\ell\omega+(1+u\omega)\cos\tau$ is monotonous in $\ell$.

 \subsection{Vertices as limits of constrictions. Proofs of Proposition \ref{adcomp} and Statement 2) of Theorem \ref{thm1}}
 \begin{proof} {\bf of Proposition \ref{adcomp}.} For every parameter value lying in the interior of a phase-lock area the Poincar\'e map of the corresponding vector field (\ref{jos2}) has one  attracting and one repelling fixed points. This follows from its M\"obiusness and strict monotonicity in the parameter $\ell$. 
 As we stay in a connected component of interior of a phase-lock area, the coordinate of the attractor of the Poincar\'e map depends analytically on the parameters by the Implicit Function Theorem. It  crosses neither $0$, nor $\pi$, since otherwise the Poincar\'e map would degenerate to either a parabolic map, or the identity, by 
 symmetry $(\theta,\tau)\mapsto(-\theta,-\tau)$ of system (\ref{jos2}) \cite{RK}. This is impossible for 
 a point from the interior of the phase-lock area, since one can perturb a parabolic Poincar\'e map to elliptic by 
 either increasing, or decreasing the parameter $\ell$. This follows by monotonicity in $\ell$ of system (\ref{jos2}). 
 Therefore, it suffices to consider a vertex $(r,u_0)$, $u_0=r+2k-1$, $r\in\zz_{\geq0}$, $k\in\nn$, and the case, when $\psi_1=(r,u_0-\var)$, $\psi_2=(r,u_0+\var)$ with small $\var$ and $\omega$ is small depending on $r$ and $\var$. The segment $I_0$ contains an attractor $a(\omega)$ of the Poincar\'e map $h_{I_0}:I_0\to I_0$ for both parameter values $\psi_1$, $\psi_2$, and the corresponding rotation number is equal to $r$, whenever $\omega$ is small enough, by Corollary \ref{corparq} of Theorem \ref{mt}. 
 
 The Poincar\'e map  $h:S^1_\theta\times\{0\}\to S^1_\theta\times\{0\}$, which is the flow map in time $\frac{2\pi}{\omega}$, is conjugated to  $h_{I_0}$ via the flow map sending $I_0$ to the $\theta$-circle. Therefore, the attractor  of the Poincar\'e map $h$ is the  intersection $\beta(\omega)$ of the forward orbit $\mco(a(\omega))$  with the $\theta$-circle. In the Riccati  coordinates centered at the right flow point $(\pi,0)$ the intersection of the latter forward orbit lifted to $\rr^2$ with the rectangle $Q_a$ converges to the arc $X_{0,a}$ of the graph of the stable solution $x_-(y)$ of the Riccati equation with $c=c_+=\ell+u$, by Lemma \ref{lconv1}. Its intersections with $O(\sqrt\omega)$-neighborhoods of the other right flow points converge to graphs  of restrictions of the same solution $x_-$ to successive disjoint intervals, any two neighbor interval pair being separated by a pole of  $x_-$, by Lemma \ref{ml2}.  In our case, for $\ell=r$ and $u=u_0=r+2k-1$, one has $c_+=u+\ell=2m-1$, $m=k+r>0$, the graph of solution $x_-$ meets the abscissa $x$-axis at $0$, if $m$ is odd, and at $\infty$, if $m$ is even. 
 Let us treat the case of odd $m$; the case of even $m$ is treated analogously. Then as $u=u_0\pm\var$ with small $\var$, the graph of solution $x_-(y)$ meets the abscissa axis at a point with small coordinates $d_{\pm}$, $d_-<0<d_+$, 
 see Proposition \ref{pmon}, thus lying in different intervals $(0,\pm\pi)$. Proposition \ref{adcomp} is proved. 
 \end{proof}
 
 \begin{proof} {\bf of Statement 2) of the Parquet Theorem.} Fix a vertex $X=(r,u_0)$,  $u_0=r+2k-1$, $k\in\nn$, and a small 
 $\var>0$ such that $(r,u_0\pm\var)$ lie in two distinct components of the limit domain $L_r^0$ adjacent to $X$ from above and from below. Then for every fixed $\omega\neq0$ small enough 
the  attractors of the Poincar\'e maps $h:S^1_\theta\times\{0\}\to S^1_\theta\times\{0\}$  corresponding to system (\ref{jos2}) with $(\ell,u)=(r,u_0\pm\var)$  lie in different half-circles $(0,\pm\pi)$, by Proposition \ref{adcomp}. Therefore, there exists 
a $u=u(\omega)\in(u_0-\var,u_0+\var)$ such that  the Poincar\'e map corresponding to the parameter 
$(r,u(\omega))$ has a fixed point either at $0$, or at $\pi$. Thus, it is either parabolic, or the identity, due to the symmetry 
$(\theta,\tau)\mapsto(-\theta,-\tau)$ of system (\ref{jos2}), see \cite{RK}, which fixes  the points $(0,0)$, $(\pi,0)$ and  conjugates the Poincar\'e map with its inverse. Hence, $(r,u(\omega))$ lies in the boundary of a phase-lock area, whenever $\omega$ is small enough. 

Thus, for arbitrarly given small $\var$ 
we found that for every $\omega$ small  enough there is a boundary point $(r,u(\omega))$ that is $\var$-close to $(r,u_0)$. Let us show that it is a constriction of the phase-lock area $L_r(\omega)$. Indeed, for every $\omega$ small enough the ray issued vertically up from the point $(r,r-\frac12)$ lies in $L_r(\omega)$, by  \cite[lemma 5.1]{bibgl}. 
The only points of boundaries of phase-lock areas that lie in the interior of the latter ray are constrictions of the phase-lock area $L_r(\omega)$, by the same lemma and its proof given in loc. cit. 
The boundary point $(r,u(\omega))$ in question lies there, since 
 $u(\omega)>u_0-\var=r+2k-1-\var>r$. Hence, it is a constriction. Statement 2) is proved. The proof of the Parquet Theorem is complete. 
\end{proof}
\section{Acknowledgements} 
I am grateful to V.M.Buchstaber for statement of the problem on asymptotics of the phase-lock areas in the RSJ model 
and for helpful discussions. I am very grateful to A.A.Alexandrov for numerous helpful discussions and for 
 Figures 3 of the phase-lock areas. I am grateful to Yu.S.Ilyashenko, J.-P.Ramis, V.A.Kleptsyn, I.V.Schurov, A.S.Gorsky, Ya.V.Fominov  for helpful discussions. The author is a winner of the Leader contest of the Theoretical Physics and Mathematics Advancement Foundation BASIS and expresses his gratitude to the foundation.

\end{document}